\documentclass[10pt, reqno]{amsart}

\usepackage{amsmath,amssymb,amsthm,amsfonts,verbatim}
\usepackage{microtype}
\usepackage[all,2cell]{xy}
\usepackage{mathtools}
\usepackage{graphicx}
\usepackage{pinlabel}
\usepackage{hyperref}
\usepackage{mathrsfs}
\usepackage{color}
\usepackage{enumerate}
\usepackage{cite}
\usepackage{subcaption}
\usepackage{tikz}
\usetikzlibrary{cd}
\tikzcdset{every label/.append style = {font = \small}}

\CompileMatrices

\usepackage[top=1.2in,bottom=1.8in,left=1.2in,right=1.2in]{geometry}

\theoremstyle{plain}
\newtheorem{theorem}{Theorem}[section]
\newtheorem{maintheorem}{Theorem}

\newtheorem{proposition}[theorem]{Proposition}
\newtheorem{lemma}[theorem]{Lemma}

\newtheorem{corollary}[theorem]{Corollary}

\theoremstyle{definition}
\newtheorem{definition}[theorem]{Definition}

\newtheorem{construction}[theorem]{Construction}

\newtheorem{remark}[theorem]{Remark}

\newcommand{\nc}{\newcommand}
\nc{\dmo}{\DeclareMathOperator}

\nc{\Q}{\mathbb{Q}}
\nc{\F}{\mathbb{F}}
\nc{\R}{\mathbb{R}}
\nc{\Z}{\mathbb{Z}}
\nc{\C}{\mathbb{C}}
\nc{\Ell}{\mathcal{L}}
\nc{\M}{\mathcal{M}}
\nc{\K}{\mathcal{K}}
\nc{\I}{\mathcal{I}}
\nc{\T}{\mathcal T}
\nc{\U}{\mathcal U}
\nc{\disk}{\mathbb{D}}
\nc{\hyp}{\mathbb{H}}

\nc{\CP}{\mathbb{CP}}
\nc{\cS}{\mathcal{S}}
\dmo{\Mod}{Mod}
\dmo{\PMod}{PMod}
\dmo{\LMod}{LMod}
\dmo{\Diff}{Diff}
\dmo{\Homeo}{Homeo}
\dmo{\dist}{dist}
\dmo\BDiff{BDiff}
\dmo\SO{SO}
\dmo\Hom{Hom}
\dmo\SL{SL}
\dmo\Sp{Sp}
\dmo\rank{rank}
\dmo\sig{sig}
\dmo\Out{Out}
\dmo\Aut{Aut}
\dmo\Inn{Inn}
\dmo\GL{GL}
\dmo\PSL{PSL}
\dmo\BHomeo{BHomeo}
\dmo\EHomeo{EHomeo}
\dmo\EDiff{EDiff}
\nc\Sig{\Sigma}
\dmo\Teich{Teich}
\dmo\Fix{Fix}
\nc{\pair}[1]{\langle #1 \rangle}
\nc{\abs}[1]{\left| #1 \right|}
\nc{\action}{\circlearrowright}
\nc{\norm}[1]{\left | \left | #1 \right | \right |}
\nc{\abcd}[4]{\left(\begin{array}{cc} #1 & #2 \\ #3 & #4 \end{array}\right)}
\nc{\into}\hookrightarrow
\dmo{\Isom}{Isom}
\nc{\normal}{\vartriangleleft}
\dmo{\Vol}{Vol}
\dmo{\im}{Im}
\dmo{\Push}{Push}
\dmo{\Conf}{Conf}
\dmo{\PConf}{PConf}
\dmo{\id}{id}
\dmo{\Jac}{Jac}
\dmo{\Pic}{Pic}
\dmo{\Stab}{Stab}
\dmo{\Arf}{Arf}
\dmo{\End}{End}
\dmo{\Gal}{Gal}
\dmo{\lcm}{lcm}
\dmo{\ab}{ab}
\dmo{\opp}{op}
\dmo{\SU}{SU}
\dmo{\OT}{\Omega \mathcal{T}}
\dmo{\OM}{\Omega \mathcal{M}}
\dmo{\spin}{spin}
\dmo{\even}{even}
\dmo{\odd}{odd}
\dmo{\comp}{\mathcal{H}}
\dmo{\Mgk}{\mathcal{M}_{g, \underline{\kappa}}}
\dmo{\orb}{orb}
\dmo{\AJ}{AJ}
\dmo{\Ck}{\mathsf{C}(\underline{\kappa})}
\dmo{\Int}{Int}
\dmo{\pr}{pr}
\dmo{\lab}{lab}
\dmo{\Sym}{Sym}

\nc{\Span}[1]{\operatorname{Span}(#1)}

\newcommand{\sing}{\underline{\kappa}}
\renewcommand{\epsilon}{\varepsilon}
\renewcommand{\tilde}{\widetilde}
\renewcommand{\le}{\leqslant}
\nc{\coloneq}{\mathrel{\mathop:}\mkern-1.2mu=}
\nc{\margin}[1]{\marginpar{\scriptsize #1}}
\nc{\para}[1]{\medskip\noindent\textbf{#1.}}
\nc{\red}[1]{\textcolor{red}{#1}}
\nc{\blue}[1]{\textcolor{blue}{#1}}
\nc{\proofof}[1]{\noindent {\em Proof (of #1).}}

\nc{\lb}{[}
\nc{\rb}{]}

\title[Framed mapping class groups and Abelian differentials]{Framed mapping class groups and the monodromy of \\strata of Abelian differentials}

\author{Aaron Calderon and Nick Salter}
\email{aaron.calderon@yale.edu}
\email{nks@math.columbia.edu}
\thanks{AC is supported by NSF Award No. DGE-1122492. NS is supported by NSF Award No. DMS-1703181.}

\address{AC: Department of Mathematics, Yale University, 10 Hillhouse Ave, New Haven, CT 06511}
\address{NS: Department of Mathematics, Columbia University, 2990 Broadway, New York, NY 10027}
\date{May 13, 2020}

\begin{document}
\begin{abstract}
This paper investigates the relationship between strata of abelian differentials and various mapping class groups afforded by means of the topological monodromy representation. Building off of prior work of the authors, we show that the fundamental group of a stratum surjects onto the subgroup of the mapping class group which preserves a fixed framing of the underlying Riemann surface, thereby giving a complete characterization of the monodromy group. In the course of our proof we also show that these ``framed mapping class groups'' are finitely generated (even though they are of infinite index) and give explicit generating sets.
\end{abstract}
\vspace*{-2.5em}
\maketitle
\vspace*{-2em}

\section{Introduction}

The moduli space $\OM_g$ of holomorphic 1--forms ({\em abelian differentials}) of genus $g$ is a complex $g$--dimensional vector bundle over the moduli space $\M_g$. The complement of its zero section is naturally partitioned into {\em strata}, sub-orbifolds with fixed number and degree of zeros. Fixing a partition $\sing:=(\kappa_1, \ldots, \kappa_n)$ of $2g-2$, we let $\OM_g(\sing)$ denote the stratum consisting of those pairs $(X, \omega)$ where $\omega$ is an abelian differential on $X \in \M_g$ with zeros of orders $\sing$.

As strata are quasi-projective varieties their (orbifold) fundamental groups are finitely presented. Kontsevich and Zorich famously conjectured that strata should be $K(G, 1)$'s for ``some sort of mapping class group'' \cite{KZ_strings}, but little progress has been made in this direction. This paper continues the work begun in \cite{C_strata1} and \cite{CS_strata2}, where the authors investigate these orbifold fundamental groups by means of a ``topological monodromy representation.'' 

Any (homotopy class of) loop in $\OM_g(\sing)$ based at $(X, \omega)$ gives rise to a(n isotopy class of) self--homeomorphism $X \rightarrow X$ which preserves $Z = \text{Zeros}(\omega)$. This gives rise to the {\em topological monodromy representation} 
\[\rho: \pi_1^{\orb}(\comp) \rightarrow \Mod_g^n\]
where $\comp$ is the component of $\OM_g(\sing)$ containing $(X, \omega)$ and $\Mod_g^n$ is the mapping class group of $X$ relative to $Z$.

\para{The fundamental invariant: framings}
The horizontal vector field $1/\omega$ of any $(X, \omega) \in \OM_g(\sing)$ defines a trivialization $\bar \phi$ of the tangent bundle of $X \setminus Z$, or an ``absolute framing'' (see \S\S\ref{subsection:framings} and \ref{subsection:punctured}; the terminology reflects the finer notion of a ``relative framing'' to be discussed below). The mapping class group $\Mod_g^n$ generally does not preserve (the isotopy class of) this absolute framing, and its stabilizer $\Mod_g^n[\bar \phi]$ is of infinite index. On the other hand, the canonical nature of $1/\omega$ means that the image of $\rho$ does leave some such absolute framing fixed. Our first main theorem identifies the image of the monodromy representation as the stabilizer of an absolute framing.

\begin{maintheorem}\label{mainthm:absolute_monodromy}
Suppose that $g \ge 5$ and $\sing$ is a partition of $2g-2$. Let $\comp$ be a non-hyperelliptic component of $\OM_g(\sing)$; then 
\[\rho\left( \pi_1^{\orb} ( \comp ) \right) \cong \Mod_{g}^n[\bar \phi]\]
where $\bar \phi$ is the absolute framing induced by the horizontal vector field of any surface in $\OM_g(\sing)$.
\end{maintheorem}

For an explanation as to why we restrict to the non-hyperelliptic components of strata, see the discussion after Theorem \ref{thm:KZ_strata}. The bound $g \ge 5$ is an artifact of the method of proof and can probably be relaxed to $g \ge 3$. We invoke $g \ge 5$ in Proposition \ref{lemma:pushmakesT}, Proposition \ref{proposition:minsep}, and Lemma \ref{lemma:genus2connected}; it is not needed elsewhere.

Equivalently, the universal property of $\M_{g,n}$ implies there is a map 
\[\begin{array}{rrcl}
\ell: & \comp  & \rightarrow & \Mgk \\
		& (X, \omega) & \mapsto & (X, \text{Zeros}(\omega)).
\end{array}\]
where $\Mgk$ denotes the moduli space of Riemann surfaces with $n$ marked points labeled by $(\kappa_1, \ldots, \kappa_n)$, in which one may only permute marked points if they have the same label.
\footnote{Of course, $\Mgk$ is a finite cover of $\M_{g,n}$, corresponding to the group $\Sym(\sing)$ of label--preserving permutations.}
 In this language, Theorem \ref{mainthm:absolute_monodromy} characterizes the image of $\ell$ at the level of (orbifold) fundamental groups.

\begin{remark}
Via the mapping class group's action on relative homology, Theorem \ref{mainthm:absolute_monodromy} also implicitly characterizes the image of the {\em homological} monodromy representation
\[\rho_H: \pi_1^{\orb} (\comp) \rightarrow \Aut(H_1(X, Z; \Z)).\]
In a companion paper \cite{CS_framedhom}, we give an explicit description of the image of $\rho_H$ as the kernel of a certain crossed homomorphism on $\Aut(H_1(X, Z; \Z))$.
\end{remark}

\para{Application: realizing curves and arcs geometrically} Using the identification of Theorem \ref{mainthm:absolute_monodromy}, we can apply a framed version of the ``change of coordinates'' principle (see Proposition \ref{proposition:CCP}) to deduce the following characterization of which curves can be realized as the core curves of embedded cylinders. To formulate this, we observe that the data of an absolute framing $\overline{\phi}$ gives rise to a ``winding number function'' (also denoted $\overline \phi$) that sends an oriented simple closed curve to the $\Z$--valued holonomy of its forward--pointing tangent vector relative to the framing (c.f. \S\ref{subsection:framings}). 

\begin{corollary}[c.f. Corollary 1.1 of \cite{CS_strata2}]\label{cor:when_cyl}
Fix $g \ge 5$ and a partition $\sing$ of $2g-2$. Pick some $(X, \omega)$ in a non-hyperelliptic component $\comp$ of $\OM_g(\sing)$ and let $\bar \phi$ denote the induced (absolute) framing. Pick a nonseparating simple closed curve $c \subset X \setminus Z$.

Then there is a path $\gamma:[0,1] \rightarrow \comp$ with $\gamma(0)=(X,\omega)$ and such that the parallel transport of $c$ along $\gamma$ is a cylinder on $\gamma(1)$ if and only if the winding number of $c$ with respect to $\bar \phi$ is $0$.
\end{corollary}
\begin{proof}
The condition that $\bar \phi(c)=0$ is necessary, as the core curve of a cylinder has constant slope.

To see that it is sufficient, we note that there is some cylinder on $X$ with core curve $d$ and the winding number of $d$ with respect to $\bar \phi$ is $0$. Therefore, by the framed change--of--coordinates principle (Proposition \ref{proposition:CCP}), there is some element $g \in \Mod_g^n[\bar\phi]$ taking $d$ to $c$. By Theorem \ref{mainthm:absolute_monodromy}, $g$ lies in the monodromy group, so there is some $\gamma \in \pi_1^{\orb}(\comp)$ whose monodromy is $g$. This $\gamma$ is the desired path.
\end{proof}

We can also deduce a complementary result for arcs using the same principle. Recall that a {\em saddle connection} on an abelian differential is a nonsingular geodesic segment connecting two zeros.

\begin{corollary}\label{cor:when_saddle}
Fix $g \ge 5$ and a partition $\sing$ of $2g-2$. Pick some $(X, \omega)$ in a non-hyperelliptic component $\comp$ of $\OM_g(\sing)$ and fix a nonseparating arc $a$ connecting distinct zeros of $Z$.

Then there is a path $\gamma:[0,1] \rightarrow \comp$ with $\gamma(0)=(X,\omega)$ and such that the parallel transport of $a$ along $\gamma$ is a realized as a saddle connection on $\gamma(1)$.
\end{corollary}

The proof of this corollary uses machinery developed throughout the paper and is therefore deferred to Section \ref{subsec:saddleCCP}. In Section \ref{section:corollaries} we collect other corollaries that we can obtain by the methods of the paper: we also give a classification of components of strata with marking data (Section \ref{section:markingcor}) and we show that for $\mathcal H$ a sufficiently general stratum--component, the monodromy image $\rho(\pi_1^{orb}(\mathcal H)) \le \PMod_g^n[\bar \phi]$ is {\em not} generated by shears about cylinders (Section \ref{section:gencor}). 

\para{Other monodromy groups} Theorem \ref{mainthm:absolute_monodromy} is a consequence of our characterization of the images of certain other monodromy representations:
in Theorem \ref{thm:relative_monodromy}, we compute the monodromy of a ``prong--marked'' stratum into the mapping class group $\Mod_{g,n}$ of a surface with boundary. Theorem \ref{thm:blowup_monodromy} computes the monodromy of a stratum into the ``pronged mapping class group,'' denoted $\Mod^*_{g,n}$, a refinement which captures the combinatorics of the zeros of the differential.
In both Theorems \ref{thm:relative_monodromy} and \ref{thm:blowup_monodromy} we find that the monodromy group is the stabilizer $\Mod_{g,n}[\phi]$, respectively $\Mod_{g,n}^*[\phi]$, of an appropriate ``relative framing'' $\phi$. 

A {\em relative framing} is an isotopy class of framing of $\Sigma_{g,n}$ where the isotopies are required to be trivial on $\partial \Sigma_{g,n}$ (see \S\ref{section:framings}). To promote an absolute framing $\bar \phi$ into a relative framing $\phi$ we ``blow up'' the zeros of a differential (see \S\ref{subsec:flattoframed}); under this transformation, a zero of order $k$ becomes a boundary component with 
winding number $-1-k$, so an element of $\OM_g(\sing)$ induces a relative framing $\phi$ on its blow-up with ``signature'' $-1-\sing := (-1 - \kappa_1, \ldots, -1-\kappa_n)$ (see \S\ref{subsection:framings}). Thus each boundary component has negative winding number; a framing with this property is said to be {\em of holomorphic type}. 

\para{Generating the framed mapping class group} The monodromy computations in Theorems \ref{mainthm:absolute_monodromy}, \ref{thm:relative_monodromy}, and \ref{thm:blowup_monodromy} rest on a development in the theory of stabilizers of relative framings as subgroups of $\Mod_{g,n}$: we determine simple explicit finite generating sets. 

We introduce some terminology used in the statement. Let $\mathcal C = \{c_1, \dots, c_k\}$ be a collection of curves on a surface $\Sigma_{g,n}$, pairwise in minimal position, with the property that the geometric intersection number $i(c_i, c_j)$ is at most $1$ for all pairs $c_i, c_j \in \mathcal C$. Associated to such a configuration is its {\em intersection graph} $\Lambda_\mathcal C$, whose vertices correspond to the elements of $\mathcal C$, with $c_i$ and $c_j$ joined by an edge whenever $i(c_i, c_j) = 1$. Such a configuration $\mathcal C$ {\em spans} $\Sigma_{g,n}$ if there is a deformation retraction of $\Sigma_{g,n}$ onto the union of the curves in $\mathcal C$. We say that $\mathcal C$ is {\em arboreal} if the intersection graph $\Lambda_{\mathcal C}$ is a tree, and {\em $E$-arboreal} if $\Lambda_{\mathcal C}$ moreover contains the $E_6$ Dynkin diagram as a subgraph. See Figure \ref{figure:genset} for the examples of spanning configurations we exploit in the pursuit of Theorem \ref{mainthm:absolute_monodromy}. 

When working with framings of meromorphic type we will need to consider sets of curves more general than spanning configurations (see the discussion in Section \ref{subsection:finishB}). To that end we define an {\em $h$-assemblage of type $E$} on $\Sigma_{g,n}$ as a set of curves $\mathcal C = \{c_1, \dots, c_k, c_{k+1}, \dots, c_\ell\}$ such that (1) $\mathcal C_1 = \{c_1,\dots, c_k\}$ is an $E$-arboreal spanning configuration on a subsurface $S \subset \Sigma_{g,n}$ of genus $g(S) = h$,  (2) for $j \ge k$, let $S_j$ denote a regular neighborhood of the curves $\{c_1, \dots, c_j\}$; then for $j > k$, we require that $c_j \cap S_{j-1}$ be a single arc (possibly, but not necessarily, entering and exiting along the same boundary component of $S_j$), and (3) $S_\ell = \Sigma_{g,n}$. In other words, an assemblage of type $E$ is built from an $E$-arboreal spanning configuration on a subsurface by sequentially attaching (neighborhoods of) further curves, decreasing the Euler characteristic by exactly one at each stage but otherwise allowing the new curves to intersect individual old curves arbitrarily.

\begin{maintheorem}\label{mainthm:genset}
Let $\Sigma_{g,n}$ be a surface of genus $g \ge 5$ with $n \ge 1$ boundary components.

\noindent (I) Suppose $\phi$ is a framing of $\Sigma_{g,n}$ of holomorphic type. Let
$\mathcal C = \{c_1, \dots, c_k\}$ be an $E$-arboreal spanning configuration of curves on $\Sigma_{g,n}$ such that $\phi(c) = 0$ for all $c \in \mathcal C$. Then 
\[
\Mod(\Sigma_{g,n})[\phi] = \pair{T_c \mid c \in \mathcal C}.
\]
\noindent (II) If $\phi$ is an arbitrary framing (of holomorphic or meromorphic type) and $\mathcal C = \{c_1,\dots, c_\ell\}$ is an $h$-assemblage of type $E$ for $h \ge 5$ of curves such that $\phi(c) = 0$ for all $c \in \mathcal C$, then
\[
\Mod(\Sigma_{g,n})[\phi] = \pair{T_c \mid c \in \mathcal C}.
\]
\end{maintheorem}

Theorem \ref{mainthm:genset} also implies a finite generation result for stabilizers of absolute framings.

\begin{corollary}\label{cor:absstab_fg}
Let $g, \sing$ and $\phi$ be as above. Let $\bar \phi$ be the absolute framing on $\Sigma_g^n$ obtained by shrinking the boundary components of $\Sigma_{g,n}$ to punctures; then $\Mod_g^n[\bar \phi]$ is generated by finitely many Dehn twists.
\end{corollary}

An explicit finite generating set for $\Mod_g^n[\bar \phi]$ is given in Corollary \ref{cor:genset_abs}. In general, the set of Dehn twists described in Theorem \ref{mainthm:genset} only generates a finite--index subgroup of $\Mod_g^n[\bar \phi]$ (Proposition \ref{proposition:relimg}).

Our methods of proof also yield a generalization of the main mapping class group--theoretic result of \cite{CS_strata2}, allowing us to greatly expand our list of generating sets for ``$r$-spin mapping class groups,'' the analogue of framed mapping class groups for closed surfaces (\S \ref{subsection:framings}). See Corollary \ref{corollary:rspingen}.

\begin{remark}
Both $\Mod_{g,n}[\phi]$ and $\Mod_g^n[\bar \phi]$ are of infinite index in their respective ambient mapping class groups, and so {\em a priori} could be infinitely generated. To the best of the authors' knowledge, Theorem \ref{mainthm:genset} and Corollary \ref{cor:absstab_fg} are the first proofs that these groups are finitely generated. This is another instance of a surprising and poorly--understood theme in the study of mapping class groups: {\em stabilizers of geometric structures often have unexpectedly strong finiteness properties}. The most famous instance of this principle is Johnson's proof that the Torelli group is finitely generated for all $g \ge 3$ \cite{johnson_FG}; this was recently and remarkably improved by Ershov--He and Church--Ershov--Putman to establish finite generation for each term in the Johnson filtration \cite{ershov_he, CEP}. 
\end{remark}

\begin{remark}
In contemporary work \cite{PS_singularities}, the second author and P. Portilla Cuadrado apply Theorem \ref{mainthm:genset} to give a description in the spirit of Theorem \ref{mainthm:absolute_monodromy} of the geometric monodromy group of an arbitrary isolated plane curve singularity as a framed mapping class group. The counterpart to Corollary \ref{cor:when_cyl} then yields an identification of the set of vanishing cycles for Morsifications of arbitrary plane curve singularities. 
\end{remark}

\para{Context} As mentioned above, this paper serves as a sequel to \cite{CS_strata2}. The main result of that work considers a weaker version of the monodromy representation attached to a stratum of abelian differentials. In \cite{CS_strata2}, we study the monodromy representation valued in the {\em closed} mapping class group $\Mod_g$; here we enrich our monodromy representation so as to track the location of the zeroes.
There, we find that an object called an ``$r$--spin structure'' (c.f. Section \ref{subsection:framings}) governs the behavior of the monodromy representation. Here, the added structure of the locations of the zeroes allows us to refine these $r$--spin structures to the more familiar notion of a globally invariant framing of the fibers. Where the technical core of \cite{CS_strata2} is an analysis of the group theory of the stabilizer in $\Mod_g$ of an $r$--spin structure, here the corresponding work is to understand these ``framed mapping class groups'' and to work out their basic theory, including the surprising fact that these infinite--index subgroups admit the remarkably simple finite generating sets described in Theorem \ref{mainthm:genset}. 

 In recent preprints, Hamenst\"adt has also analyzed the monodromy representation into $\Mod_g^n$. In \cite{Hamen_strata} she gives generators for the image in terms of ``completely periodic admissible configurations,'' which are analogous to the spanning configurations appearing in Theorem \ref{mainthm:genset}. In \cite{Hamen_strata2}, she identifies the image monodromy into the {\em closed} mapping class group $\Mod_g$ as the stabilizer of an ``$r$--spin structure,'' recovering and extending work of the authors (see \S \ref{subsection:framings} as well as \cite{CS_strata2}). The paper \cite{Hamen_strata2} also contains a description of generators for the fundamental groups of certain strata.

\subsection{Structure of the paper}
This paper is roughly divided into two parts: the first deals exclusively with relative framings on surfaces with boundary and their associated framed mapping class groups, while the second deals with variations on framed mapping class groups and their relationship with strata of abelian differentials. Readers interested only in Theorem \ref{mainthm:genset} can read Sections \ref{section:framings}--\ref{section:arcgraph} independently, while readers interested only in Theorem \ref{mainthm:absolute_monodromy} need only read the introductory Section \ref{section:framings} together with Sections \ref{section:absolute} -- \ref{section:corollaries} (provided they are willing to accept Theorem \ref{mainthm:genset} as a black box).

\para{Outline of Theorem \ref{mainthm:genset}}
The proof of Theorem \ref{mainthm:genset} has two steps, these steps roughly parallel those of \cite[Theorem B]{CS_strata2}. For the first step, we show in Proposition \ref{proposition:fingen} that the Dehn twists on a spanning configuration of admissible curves as specified in the Theorem generate the ``admissible subgroup'' $\mathcal T_\phi \le \Mod_{g,n}[\phi]$ (see Section \ref{section:framedmcg}). The proof of this step relies on the theory of ``subsurface push subgroups'' from \cite{salter_toric} and extends these results, establishing a general inductive procedure to build subsurface push subgroups from admissible twists and sub-subsurface push subgroups (Lemma \ref{lemma:inductive}).

The second step is to show that the admissible subgroup is the {\em entire} stabilizer of the relative framing; the proof of this step spans both Sections \ref{section:basecase} and \ref{section:arcgraph}. In \cite{CS_strata2} and \cite{salter_toric}, the analogous step is accomplished using the ``Johnson filtration'' of the mapping class group, a strategy which does not work for surfaces with multiple boundary components. Instead, we prove that $\mathcal T_\phi = \Mod_{g,n}[\phi]$ by induction on the number of boundary components of $\Sigma_{g,n}$.

The base case of the induction (when there is a single boundary component of winding number $1-2g$) takes place in Section \ref{section:basecase}. Its proof relies heavily on the analysis of \cite{CS_strata2} and the relationship between framings and ``$r$--spin structures,'' their analogues on closed surfaces (see the end of \S\ref{subsection:framings}). Using a version of the Birman exact sequence adapted to framed mapping class groups (Lemma \ref{lemma:framedBES}), we show that the equality $\mathcal T_\phi = \Mod_{g,n}[\phi]$ is equivalent to the statement that $\mathcal T_\phi$ contains ``enough separating twists.'' We directly exhibit these twists in Proposition \ref{proposition:minsep}, refining \cite[Lemma 6.4]{CS_strata2} and its counterpart in \cite{salter_toric}; the reader is encouraged to think of Proposition \ref{proposition:minsep} as the ``canonical version'' of this statement.

The inductive step of the proof that $\mathcal T_\phi = \Mod_{g,n}[\phi]$ is contained in Section \ref{section:arcgraph}. The overall strategy is to introduce a connected graph $\mathcal A^s$ on which $\Mod_{g,n}[\phi]$ acts vertex-- and edge--transitively (see Sections \ref{section:action} through \ref{section:asl}). The heart of the argument is thus to establish these transitivity properties (Lemma \ref{lemma:actionproperties}) and the connectedness of $\mathcal A^s$ (Lemma \ref{lemma:connected}); these both require a certain amount of care, and the arguments are lengthy. Standard techniques then imply $\Mod_{g,n}[\phi]$ is generated by the stabilizer of a vertex (which we can identify with $\Mod_{g,n-1}[\phi']$ for some $\phi'$) together with an element that moves along an edge (Lemma \ref{lemma:actiongenset}). Applying the inductive hypothesis and explicitly understanding the action of certain Dehn twists on $\mathcal A^s$ together yield that $\T_\phi = \Mod_{g,n}[\phi]$, completing the proof of Theorem \ref{mainthm:genset}.

\para{Variations on framed mapping class groups}
Section \ref{section:absolute} is an interlude into the theory of other framed mapping class groups. In Section \ref{section:pronged} we introduce the theory of {\em pronged surfaces,} surfaces with extra tangential data which mimic the zero structure of an abelian differential. After discussing the relationship between the mapping class groups $\Mod_{g,n}^*$ of pronged surfaces and surfaces with boundaries or marked points, we introduce a theory of relative framings of pronged surfaces and hence a notion of framed, pronged mapping class group $\Mod_{g,n}^*[\phi]$. The main result of this subsection is Proposition \ref{theorem:pronged}, which exhibits $\Mod_{g,n}^*[\phi]$ as a certain finite extension of $\Mod_{g,n}[\phi]$.

We then proceed in Section \ref{subsection:punctured} to a discussion of absolute framings of pointed surfaces, as in the beginning of this Introduction. When a surface has marked points instead of boundary components, framings can only be considered up to absolute isotopy. Therefore, the applicable notion is not a relative but an absolute framing $\bar \phi$. In this section we prove Theorem \ref{theorem:absolute}, which states that the (pronged) relative framing stabilizer $\Mod_{g,n}^*[\phi]$ surjects onto the (pointed) absolute framing stabilizer $\PMod_{g}^n[\bar \phi]$. Combining this theorem with work of the previous subsection also gives explicit generating sets for $\PMod_{g}^n[\bar \phi]$ (see Corollary \ref{cor:genset_abs}).

\para{Outline of Theorem \ref{mainthm:absolute_monodromy}}
The proof of Theorem \ref{mainthm:absolute_monodromy} is accomplished in Section \ref{section:monodromy}. After recalling background material on abelian differentials (\S\ref{subsec:flatbasics}) and exploring the different sorts of framings a differential induces (\S\ref{subsec:flattoframed}), we record the definitions of certain marked strata, first introduced in \cite{BSW_horo} (\S\ref{subsec:markings}). These spaces fit together in a tower of coverings \eqref{spaces} which evinces the structure of the pronged mapping class group, as discussed in Section \ref{section:pronged}. By a standard continuity argument, the monodromy of each covering must stabilize a framing (see Lemma \ref{lemma:mon1_preserves_framing} and Corollaries \ref{cor:mon2_preserves_framing} and \ref{cor:mon3_preserves_framing}).

Using these marked strata, we can upgrade the $\Mod_g^n$--valued monodromy of $\comp$ into a $\Mod^*_{g,n}$--valued homomorphism, and passing to a certain finite cover of the stratum therefore results in a space $\comp^{\pr}$ whose monodromy lies in $\Mod_{g,n}$. By realizing the generating set of Theorem \ref{mainthm:genset} as cylinders on a prototype surface in $\comp^{\pr}$, we can explicitly construct deformations whose monodromy is a Dehn twist, hence proving that the $\Mod_{g,n}$--valued monodromy group of $\comp^{\pr}$ is the entire stabilizer of the appropriate framing (Theorem \ref{thm:relative_monodromy}).

To deduce Theorem \ref{mainthm:absolute_monodromy} from the monodromy result for $\comp^{\pr}$ requires an understanding of the interactions between all three types of framed mapping class groups. Using the diagram of coverings \eqref{spaces} together with the structural results of Section \ref{section:absolute}, we conclude that the $\Mod^*_{g,n}$--valued monodromy of $\comp$ is exactly the framing stabilizer $\Mod_{g,n}^*[\phi]$ (Theorem \ref{thm:blowup_monodromy}). An application of Theorem \ref{theorem:absolute} together with a discussion of the permutation action of $\Mod_g^n[\phi]$ on $Z$ finishes the proof of Theorem \ref{mainthm:absolute_monodromy}.

The concluding Section \ref{section:corollaries} contains applications of our analysis to the classification of components of certain covers of strata (Corollaries \ref{cor:compOTlab} and \ref{cor:compOTpr}) as well as to the relationship between cylinders and the fundamental groups of strata (\S \ref{section:gencor}). This section also contains the proof of Corollary \ref{cor:when_saddle}.

\subsection{Acknowledgments}
Large parts of this work were accomplished when the authors were visiting MSRI for the ``Holomorphic Differentials in Mathematics and Physics'' program in the fall of 2019, and both authors would like to thank the venue for its hospitality, excellent working environment, and generous travel support. The first author gratefully acknowledges financial support from NSF grants DMS-1107452, -1107263, and -1107367 ``RNMS: Geometric Structures and Representation Varieties'' (the GEAR Network) as well as from NSF grant  DMS-161087.

The first author would like to thank Yair Minsky for his support and guidance. The authors would also like to thank Barak Weiss for prompting them to consider monodromy valued in the pronged mapping class group, as well as for alerting them to Boissy's work on prong--marked strata \cite{Boissy_framed}. They are grateful to Ursula Hamenst\"adt for some enlightening discussions and productive suggestions.

\section{Framings and framed mapping class groups}\label{section:framings}

\subsection{Framings}\label{subsection:framings}
We begin by recalling the basics of framed surfaces. Our conventions ultimately follow those of Randal-Williams \cite[Sections 1.1, 2.3]{RW}, but we have made some convenient cosmetic alterations and use language compatible with our previous papers \cite{salter_toric, CS_strata2}. See Remark \ref{remark:RW} below for an explanation of how to reconcile these two presentations. 

\para{Framings, (relative) isotopy} Let $\Sigma_{g,n}$ denote a compact oriented surface of genus $g$ with $n\ge1$ boundary components $\Delta_1, \dots, \Delta_n$. Through Section \ref{section:arcgraph} we will work exclusively with boundary components, but in Section \ref{section:absolute}, we will also consider surfaces equipped with marked points. We formulate our discussion here for surfaces with boundary components; we will briefly comment on the changes necessary to work with marked points in Section \ref{section:absolute}.

 Throughout this section we fix an orientation $\theta$ and a Riemannian metric $\mu$ of $\Sigma_{g,n}$, affording a reduction of the structure group of the tangent bundle $T\Sigma_{g,n}$ to $\SO(2)$. A {\em framing} of $\Sigma_{g,n}$ is an isomorphism of $\SO(2)$--bundles
\[
\phi: T\Sigma_{g,n} \to \Sigma_{g,n} \times \R^2.
\]
With $\theta, \mu$ fixed, framings are in one-to-one correspondence with nowhere-vanishing vector fields $\xi$; in the sequel we will largely take this point of view. In this language, we say that two framings $\phi, \psi$ are {\em isotopic} if the associated vector fields $\xi_\phi$ and $\xi_\psi$ are homotopic through nowhere-vanishing vector fields.
 
 Suppose that $\phi$ and $\psi$ restrict to the same framing $\delta$ of $\partial\Sigma_{g,n}$.
In this case, we say that $\phi$ and $\psi$ are {\em relatively isotopic} if they are isotopic through framings restricting to $\delta$ on $\partial\Sigma_{g,n}$. With a choice of $\delta$ fixed, we say that $\phi$ is a {\em relative framing} if $\phi$ is a framing restricting to $\delta$ on $\partial\Sigma_{g,n}$. 

\para{(Relative) winding number functions} Let $(\Sigma_{g,n}, \phi)$ be a (relatively) framed surface. We explain here how the data of the (relative) isotopy class of $\phi$ can be encoded in a topological structure known as a {\em (relative) winding number function}. Let $c: S^1 \to \Sigma_{g,n}$ be a $C^1$ immersion. Given two vectors $v,w \in T_x \Sigma_{g,n}$, we denote the angle (relative to the metric $\mu$) between $v,w$ by $\angle(v,w)$. We define the {\em winding number} $\phi(c)$ of $c$ as the degree of the ``Gauss map'' restricted to $c$:
\[
\phi(c) = \int_{S^1} d \angle(c'(t), \xi_\phi(c(t))).
\]
The winding number $\phi(c)$ is clearly an invariant of the isotopy class of $\phi$, and is furthermore an invariant of the isotopy class of $c$ as an immersed curve in $\Sigma_{g,n}$.\footnote{It is not, however, an invariant of the {\em homotopy} class of the map $c$, since the winding number will change under the addition or removal of small self-intersecting loops. In this paper we will be exclusively concerned with winding numbers of {\em embedded} curves and arcs, so we will not comment further on this. See \cite{HJ} for further discussion.} 

Possibly after altering $\phi$ by an isotopy, we can assume that each component $\Delta_i$ of $\partial\Sigma_{g,n}$ contains a point $p_i$ such that $\xi_{\phi}(p_i)$ is orthogonally inward-pointing. 
We call such a point $p_i$ a {\em legal basepoint} for $\Delta_i$. We emphasize that even though $\Delta_i$ may contain several legal basepoints, we choose {\em exactly one} legal basepoint on each $\Delta_i$, so that all arcs based at $\Delta_i$ are based at the same point.

Let $a: [0,1] \to \Sigma_{g,n}$ be a $C^1$ immersion with $a(0), a(1)$ equal to distinct legal basepoints $p_i, p_j$; assume further that $a'(0)$ is orthogonally inward-pointing and $a'(1)$ is orthogonally outward pointing. We call such an arc {\em legal}. Then the winding number
\[
\phi(a) : = \int_0^1 d \angle(a'(t), \xi_\phi(a(t)))
\]
is necessarily half-integral, and is invariant under the {\em relative} isotopy class of $\phi$ and under isotopies of $a$ through legal arcs. 

Thus a framing $\phi$ gives rise to an {\em absolute winding number function} which we denote by the same symbol. Let $\mathcal S$ denote the set of isotopy classes of oriented simple closed curves on $\Sigma_{g,n}$. Then the framing $\phi$ determines the winding number function
\[
\phi: \mathcal S \to \Z; \qquad x \mapsto \phi(x).
\]
Likewise, let $\mathcal S^+$ be the set obtained from $\mathcal S$ by adding the set of isotopy classes of legal arcs. Then $\phi$ also determines a {\em relative winding number function}
\[
\phi: \mathcal S^+ \to \tfrac{1}{2} \Z; \qquad x \mapsto \phi(x).
\] 

\para{Signature; holomorphic/meromorphic type} The {\em signature} of a framing $\delta$ of $\partial\Sigma_{g,n}$ (or of a framing $\phi$ of $\Sigma_{g,n}$ restricting to $\delta$ on $\partial\Sigma_{g,n}$) is the vector 
\[
\sig(\delta):=(\delta(\Delta_1), \dots \delta(\Delta_n)) \in \Z^n,
\]
where each $\Delta_i$ is oriented with $\Sigma_{g,n}$ lying to the left. A relative framing $\phi$ is said to be {\em of holomorphic type} if $\sig(\Delta_i) \le -1$ for all $i$ and is {\em of meromorphic type} otherwise. In Section \ref{subsec:flattoframed} we will see that if $\omega$ is an Abelian differential on a Riemann surface $X$, then the relative framing induced by $\omega$ is indeed of holomorphic type. Given a partition $\sing= (\kappa_1, \dots, \kappa_n)$ of $2g-2$, we say that a relative framing $\phi$ has signature $-1 - \sing$ if the boundary components have signatures $(-1 - \kappa_1, \dots, -1 - \kappa_n)$. 

\begin{remark}\label{remark:RW}
For the convenience of the reader interested in comparing the statements of this section with their counterparts in \cite{RW}, we briefly comment on the places where the two expositions diverge. We have used the term ``framing'' where Randal-Williams uses ``$\theta^r$--structure'' with $r = 0$, and we use the term ``(relative) winding number function''  where Randal-Williams uses an equivalent structure denoted ``$q_\xi$.'' Randal-Williams also adopts some different normalization conventions. If $x$ is a curve, then $q_\xi(x) = \phi(x) - 1$, and if $x$ is an arc, $q_\xi(x) = \phi(x) - 1/2$ (in particular, $q_\xi$ is integer-valued on arcs).
\end{remark}

The lemma below allows us to pass between framings and winding number functions. Its proof is a straightforward exercise in differential topology (c.f. \cite[Proposition 2.4]{RW}).
\begin{lemma}\label{lemma:framingWNF}
Let $\Sigma_{g,n}$ be a surface with $n \ge 1$ boundary components and let $\phi$ and $\psi$ be framings. Then $\phi$ and $\psi$ are isotopic as framings if and only if the associated absolute winding number functions are equal. If $\phi \mid_{\partial \Sigma_{g,n}} = \psi \mid_{\partial \Sigma_{g,n}}$, then $\phi$ and $\psi$ are relatively isotopic if and only if the associated relative winding number functions are equal. 

Moreover, if $\Sigma_{g,n}$ is endowed with the structure of a CW complex for which each $0$--cell is a legal basepoint and each $1$--cell is either isotopic to a simple closed curve or a legal arc, then $\phi$ and $\psi$ are (relatively) isotopic if and only if the (relative) winding numbers of each $1$--cell are equal.
\end{lemma}

\begin{remark} Following Lemma \ref{lemma:framingWNF}, we will be somewhat lax in our terminology. Often we will use the term ``(relative) framing'' to refer to the entire (relative) isotopy class, or else conflate the (relative) framing with the associated (relative) winding number function.
\end{remark}

\para{Properties of (relative) winding number functions} The terminology of ``winding number function'' originates with the work of Humphries and Johnson \cite{HJ} (although we are discussing what they call {\em generalized} winding number functions). We recall here some properties of winding number functions which they identified. \footnote{In the non-relative setting.}

\begin{lemma}\label{lemma:HJ}
Let $\phi$ be a relative winding number function on $\Sigma_{g,n}$ associated to a relative framing of the same name. Then $\phi$ satisfies the following properties.
\begin{enumerate}
\item \label{item:TL} (Twist--linearity) Let $a \subset \Sigma_{g,n}$ be a simple closed curve, oriented arbitrarily. Then for any $x \in \mathcal S^+$,
\[
\phi(T_a(x)) = \phi(x) + \pair{x,a} \phi(a),
\]
where $\pair{\cdot, \cdot}: H_1(\Sigma_{g,n}, \partial \Sigma_{g,n};\Z) \times H_1(\Sigma_{g,n}; \Z) \to \Z$ denotes the relative algebraic intersection pairing.
\item \label{item:HC} (Homological coherence) Let $S \subset \Sigma_{g,n}$ be a subsurface with boundary components $c_1, \dots, c_k$, oriented so that $S$ lies to the left of each $c_i$. Then
\[
\sum_{i = 1}^k \phi(c_i) = \chi(S),
\]
where $\chi(S)$ denotes the Euler characteristic. 
\end{enumerate}
\end{lemma}

\para{Functoriality} In the body of the argument we will have occasion to consider maps between surfaces equipped with framings and related structures. We record here some relatively simple observations about this. Firstly, as we have already implicitly used, if $S \subset \Sigma_{g,n}$ is a subsurface, any (isotopy class of) framing $\phi$ of $\Sigma_{g,n}$ restricts to a (isotopy class of) framing of $S$. There is a converse as well; the proof is an elementary exercise in differential topology.

\begin{lemma}\label{lemma:extension}
Let $S \subset \Sigma_{g,n}$ be a subsurface, and let $\phi$ be a framing of $S$. Enumerate the components of $\Sigma_{g,n} \setminus S$ as $S_1, \dots, S_k$. Call such a component {\em relatively closed} if $\partial S_i \subset \partial S$. Then $\phi$ extends to a framing $\tilde \phi$ of $\Sigma_{g,n}$ if and only if for each relatively closed component $S_i$, there is an equality
\[
\sum_{c\text{ a component of }\partial S_i}\phi(c) = \chi(S_i),
\]
with each $c$ oriented with $S_i$ to the left. 

In particular, suppose that $S = \Sigma_{g,n} \setminus D$, where $D$ is an embedded disk, and let $\phi$ be a framing of $S$. Then $\phi$ extends over $D$ to give a framing of $\Sigma_{g,n}$ if and only if $\phi(\partial D) = 1$ when $\partial D$ is oriented with $D$ to the left. 
\end{lemma}

\para{Closed surfaces; $\mathbf{r}$--spin structures} For $g \ge 2$, the closed surface $\Sigma_g$ does not admit any nonvanishing vector fields, but there is a ``mod $r$ analogue'' of a framing called an {\em $r$--spin structure}. As $r$--spin structures will play only a passing role in the arguments of this paper (c.f. Section \ref{section:minstrat}), we present here only the bare bones of the theory. See \cite[Section 3]{salter_toric} for a much more complete discussion.
\begin{definition}\label{definition:rspin}
Let $\Sigma_g$ be a closed surface, and as above, let $\mathcal S$ denote the set of oriented simple closed curves on $\Sigma_g$. An {\em $r$--spin structure} is a function
\[
\widehat \phi: \mathcal S \to \Z/r\Z
\]
that satisfies the twist--linearity and homological coherence properties of Lemma \ref{lemma:HJ}.
\end{definition}

Above we saw how a nonvanishing vector field $\xi_\phi$ on $\Sigma_{g,n}$ gives rise to a winding number function $\phi$ on $\Sigma_{g,n}$. Suppose now that $\xi$ is an arbitrary vector field on $\Sigma_g$ with isolated zeroes $p_1, \dots, p_k$. For $i = 1, \dots, k$, let $\gamma_i$ be a small embedded curve encircling $p_i$ (oriented with $p_i$ to the left) and define
\[
r = \gcd\{WN_\xi(\gamma_i) -1 \mid 1 \le i \le k\},
\]
where $WN_\xi(\gamma_i)$ means to take the $\Z$-valued winding number of $\gamma_i$ viewed as a curve on $\Sigma_g \setminus \{p_1, \dots, p_n\}$ endowed with the framing given by $\xi$. Then it can be shown (c.f. \cite[Section 1]{HJ}) that the map
\[
\widehat \phi: \mathcal S \to \Z/r\Z; \qquad \phi(x) = WN_\xi(x) \pmod r
\]
determines an $r$--spin structure. 

\subsection{The action of the mapping class group}\label{section:framingaction}
Recall that the mapping class group $\Mod_{g,n}$  of $\Sigma_{g,n}$ is the set of isotopy classes of self--homeomorphisms of $\Sigma_{g,n}$ which restrict to the identity on $\partial \Sigma_{g,n}$. Therefore $\Mod_{g,n}$ acts on the set of relative (isotopy classes of) framings, and hence the set of relative winding number functions, by pullback. As we require $\Mod_{g,n}$ to act from the left, there is the formula 
\[
(f \cdot \phi)(x) = \phi(f^{-1}(x)).
\]
 
We recall here the basic theory of this action, as developed by Randal-Williams \cite[Section 2.4]{RW}. Throughout this section, we fix a framing $\delta$ of $\partial \Sigma_{g,n}$ and may therefore choose a legal basepoint on each boundary component once and for all.

\para{The (generalized) Arf invariant} The $\Mod_{g,n}$ orbits of relative framings are classified by a generalization of the classical Arf invariant. To define this, we introduce the notion of a {\em distinguished geometric basis} for $\Sigma_{g,n}$. For $i = 1, \dots, n$, let $p_i$ be a legal basepoint on the $i^{th}$ boundary component of $\Sigma_{g,n}$. A {\em distinguished geometric basis} is a collection 
\[
\mathcal B = \{x_1, y_1, \dots, x_g, y_g\} \cup \{a_2, \dots, a_{n}\}
\]
of $2g$ oriented simple closed curves $x_1, \dots, y_g$ and $n-1$ legal arcs $a_2, \dots, a_{n}$ that satisfy the following intersection properties.
\begin{enumerate}
\item $\pair{x_i, y_i} = i(x_i, y_i) = 1$ (here $i(\cdot, \cdot)$ denotes the geometric intersection number) and all other pairs of elements of $\{x_1, \dots, y_g\}$ are disjoint.
\item Each arc $a_i$ is a legal arc running from $p_1$ to $p_{i}$ and is disjoint from all curves $x_1, \dots, y_g$.
\item The arcs $a_i, a_j$ are disjoint except at the common endpoint $p_1$. 
\end{enumerate}
\begin{remark}\label{remark:gsbCW}
A distinguished geometric basis can easily be used to determine a CW--structure on $\Sigma_{g,n}$ satisfying the hypotheses of Lemma \ref{lemma:framingWNF}. In particular, a (relative) winding number function (and hence the associated (relative) isotopy class of framing) is determined by its values on a distinguished geometric basis. Moreover, for any vector $(w_1, \dots, w_{2g+n-1}) \in \Z^{2g} \times (\Z+\tfrac{1}{2})^{n-1}$, there exists a framing $\phi$ of $\Sigma_{g,n}$ realizing the values $(w_1, \dots, w_{2g+n-1})$ on a chosen distinguished geometric basis (this is a straightforward construction). 
\end{remark}

Let $\mathcal B$ be a distinguished geometric basis for the framed surface $(\Sigma_{g,n}, \phi)$. The {\em Arf invariant} of $\phi$ relative to $\mathcal B$ is the element of $\Z/2\Z$ given by
\begin{equation}\label{equation:arf}
\Arf(\phi, \mathcal B) = \sum_{i = 1}^g (\phi(x_i)+1)(\phi(y_i) + 1) + \sum_{j = 2}^{n} (\phi(a_j)+\tfrac{1}{2})(\phi(\Delta_j)+1) \pmod 2
\end{equation}
(compare to \cite[(2.4)]{RW}).

\begin{lemma}[c.f. Proposition 2.8 of \cite{RW}]\label{lemma:invariant}
Let $(\Sigma_{g,n}, \phi)$ be a framed surface. Then the Arf invariant $\Arf(\phi, \mathcal B)$ is independent of the choice of distinguished geometric basis $\mathcal B$.
\end{lemma}
\begin{remark}
We caution the reader that while the Arf invariant does not depend on the choice of basis, it {\em does} depend on the choice of legal basepoints on each boundary component. Since $\Mod_{g,n}$ fixes the boundary pointwise it preserves our choice of legal basepoint on each boundary component.
\end{remark}

Following Lemma \ref{lemma:invariant}, we write $\Arf(\phi)$ to indicate the Arf invariant $\Arf(\phi, \mathcal B)$ computed on an arbitrary choice of distinguished geometric basis. The next result shows that for $g \ge 2$, the Arf invariant classifies $\Mod_{g,n}$ orbits of relative isotopy classes of framings.

\begin{proposition}[c.f. Theorem 2.9 of \cite{RW}]\label{proposition:orbits}
Let $g \ge 2$ and $n \ge 1$ be given, and let $\phi, \psi$ be two relative framings of $\Sigma_{g,n}$ which agree on $\partial \Sigma_{g,n}$. Then there is an element $f \in \Mod_{g,n}$ such that $f\cdot \psi = \phi$ if and only if $\Arf(\phi) = \Arf(\psi)$. 
\end{proposition}

For surfaces of genus $1$, the action is more complicated. In this work we will only need to study the case of one boundary component; this was treated by Kawazumi \cite{kawazumi}. Let $(\Sigma_{1,1}, \phi)$ be a framed surface. Consider the set
\[
\mathfrak a(\phi) = \{\phi(x) \mid x \subset \Sigma_{1,1}\mbox{ an oriented s.c.c}\}.
\]
The twist-linearity formula (Lemma \ref{lemma:HJ}.\ref{item:TL}) implies that $\mathfrak a(\phi)$ is in fact an ideal of $\Z$. We define the {\em genus-$1$ Arf invariant} of $\phi$ to be the unique nonnegative integer $\Arf_1(\phi) \in \Z_{\ge 0}$ such that 
\begin{equation}\label{equation:arf1}
\mathfrak a (\phi) = \Arf_1(\phi) \Z.
\end{equation}
\begin{remark}
The normalization conventions for $\Arf$ as in \eqref{equation:arf} and $\Arf_1$ as in \eqref{equation:arf1} are different. In Section \ref{section:septwist} we will reconcile them, but here we chose to present the ``natural'' definition of $\Arf_1$.
\end{remark}

\begin{lemma}[c.f. Theorem 0.3 of \cite{kawazumi}]\label{lemma:genus1arf}
Let $\phi$ and $\psi$ be relative framings of $\Sigma_{1,1}$. Then there is $f \in \Mod_{1,1}$ such that $f \cdot \psi = \phi$ if and only if $\Arf_1(\phi) = \Arf_1(\psi)$. 
\end{lemma}

\subsection{Framed mapping class groups}\label{section:framedmcg}
Having studied the orbits of the $\Mod_{g,n}$ action on the set of framings in the previous section, we turn now to the stabilizer of a framing.
\begin{definition}[Framed mapping class group]
Let $(\Sigma_{g,n}, \phi)$ be a (relatively) framed surface. The {\em framed mapping class group} $\Mod_{g,n}[\phi]$ is the stabilizer of the relative isotopy class of $\phi$:
\[
\Mod_{g,n}[\phi] = \{ f \in \Mod_{g,n} \mid f \cdot \phi = \phi\}.
\]
\end{definition}

\begin{remark}
We pause here to note one somewhat counterintuitive property of {\em relatively} framed mapping class groups. Suppose that $\phi, \phi'$ are distinct as relative isotopy classes of framings, but are equal as absolute framings (in terms of relative winding number functions, this means that $\phi$ and $\phi'$ agree when restricted to the set of simple closed curves but assign different values to arcs). Then the associated relatively framed mapping classes are {\em equal}: $\Mod_{g,n}[\phi] = \Mod_{g,n}[\phi']$. This is not hard to see: allowing the framing on the boundary to move under isotopy changes the winding numbers of all arcs in the same way, so that the $\phi'$-winding number of an arc can be computed from the $\phi$-winding number by adjusting by a universal constant. Necessarily then $\phi(f(\alpha)) = \phi(\alpha)$ for a mapping class $f$ and an arc $\alpha$ if and only if $\phi'(f(\alpha)) = \phi(\alpha)$.

\end{remark}

\para{Admissible curves, admissible twists, and the admissible subgroup} In our study of $\Mod_{g,n}[\phi]$, a particularly prominent role will be played by the Dehn twists that preserve $\phi$. An {\em admissible curve} on a framed surface $(\Sigma_{g,n}, \phi)$ is a nonseparating simple closed curve $a$ such that $\phi(a) = 0$. It follows from the twist--linearity formula (Lemma \ref{lemma:HJ}.\ref{item:TL}) that the associated Dehn twist $T_a$ preserves $\phi$. We call the mapping class $T_a$ an {\em admissible twist}. Finally, we define the {\em admissible subgroup} to be the group generated by all admissible twists:
\[
\mathcal T_\phi := \pair{ T_a \mid a \mbox{ admissible for }\phi}.
\]

\para{Change--of--coordinates for framed surfaces} The classical ``change--of--coordinates principle'' for surfaces is a body of techniques for constructing special configurations of curves and subsurfaces on a fixed surface (c.f. \cite[Section 1.3]{FarbMarg}). The underlying mechanism is the classification of surfaces, which provides a homeomorphism between a given surface and a ``reference surface;'' if a desired configuration exists on the reference surface, then the configuration can be pulled back along the classifying homeomorphism. 

A similar principle exists for framed surfaces, governing when configurations of curves with prescribed winding numbers exist on framed surfaces. The classification results Proposition \ref{proposition:orbits} and Lemma \ref{lemma:genus1arf} assert that the Arf invariant provides the only obstruction to constructing desired configurations of curves in the presence of a framing. We will make extensive and often tacit use of the ``framed change--of--coordinates principle'' throughout the body of the argument. Here we will illustrate some of the more frequent instances of which we avail ourselves. Recall that a {\em $k$--chain} is a sequence $c_1, \dots, c_k$ of curves such that $i(c_i, c_{i+1}) = 1$ for $i = 1, \dots, k-1$, and $i(c_i, c_j) = 0$ for $\abs{i-j} \ge 2$.

\begin{proposition}[Framed change--of--coordinates]\label{proposition:CCP}
Let $(\Sigma_{g,n}, \phi)$ be a relatively framed surface with $g \ge 2$ and $n \ge 1$. A configuration $x_1, \dots, x_k$ of curves and/or arcs with prescribed intersection pattern and winding numbers $\phi(x_i) = s_i$ exists if and only if
\begin{enumerate}
\item[(a)] a configuration $\{x_1', \dots, x_k'\}$ of the prescribed topological type exists in the ``unframed'' setting where the values $\phi(x_i')$ are allowed to be arbitrary,
\item[(b)] there exists some framing $\psi$ such that $\psi(x_i) = s_i$ for all $i$, and
\item[(c)] if $\Arf(\psi)$ is determined by the constraints of (b), then $\Arf(\phi) = \Arf(\psi)$. 
\end{enumerate}
In particular:
\begin{enumerate}
\item For $s \in \Z$ arbitrary, there exists a nonseparating curve $c \subset \Sigma_{g,n}$ with $\phi(c) = s$.
\item For $n = 1$, there exists a $2g$--chain of admissible curves on $\Sigma_{g,1}$ if and only if the pair \\$(g \pmod 4, \Arf(\phi)) \in (\Z/4\Z, \Z/2\Z)$ is one of the four listed below:
\begin{equation}\label{equation:goodarf}
(0,0), (1,1), (2,1), (3,0).
\end{equation}
Such a chain is called a {\em maximal} chain of admissible curves.
\end{enumerate}
\end{proposition}
\begin{proof}
We will prove (1) and (2), from which it will be clear how the general argument works. We begin with (1). Let $\mathcal B = \{x_1 \dots, y_g, a_2, \dots, a_n\}$ be a distinguished geometric basis. Following Remark \ref{remark:gsbCW}, a relative framing $\psi$ of $\Sigma_{g,n}$ can be constructed by (freely) specifying the values $\psi(b)$ for each element $b \in \mathcal B$. Set $\psi(x_1) = s$ and let $\psi(y_1)$ be arbitrary. Since $g \ge 2$, it is possible to choose the values of $\psi(x_2), \psi(y_2)$ such that $\Arf(\psi) = \Arf(\phi)$. By Proposition \ref{proposition:orbits}, there exists a diffeomorphism $f: \Sigma_{g,n} \to \Sigma_{g,n}$ such that $f\cdot \psi = \phi$. We see that $f(x_1)$ is the required curve: 
\[
\phi(f(x_1)) = (f^{-1} \cdot \phi)(c_1) = \psi(x_1) = s
\] 
as required.

For (2), consider a maximal chain $a_1, \dots, a_{2g}$ on $\Sigma_{g,1}$. Define $b_1 : = a_1$ and choose curves $b_2, b_3, \dots, b_g$ each disjoint from all $a_j$ with $j$ odd, such that
\[
\mathcal B = \{b_1, a_2, b_2, a_4,  \dots, b_g, a_{2g}\}
\]
is a distinguished geometric basis. We now construct a framing $\psi$ such that each $a_i$ is admissible. By construction, the curves $b_k, a_{2k+1}, b_{k+1}$ form pairs of pants for each $1 \le k \le g-1$. By the homological coherence property (Lemma \ref{lemma:HJ}.\ref{item:HC}), if each $a_i$ is to be admissible, we must have $\psi(b_k) = 1-k$ for $1 \le k \le g$ when $b_k$ is oriented so that the pair of pants cobounded by $b_{k-1}$ and $a_{2k-1}$ lies to the left. $\Arf(\psi)$ is determined by these conditions, and is computed to be 
\[
\Arf(\psi) = \begin{cases}
		0 & g \equiv 0,3 \pmod 4\\
		1 & g \equiv 1,2 \pmod 4.
\end{cases}
\]
If the pair $(g, \Arf(\phi))$ is one of those listed in \eqref{equation:goodarf}, then by Proposition \ref{proposition:orbits}, there exists $f: \Sigma_{g,1} \to \Sigma_{g,1}$ such that $f \cdot \psi = \phi$. As above, we find that $f(a_1), \dots, f(a_{2g})$ is the required maximal chain of admissible curves. Conversely, if $(g, \Arf(\phi))$ does not appear in \eqref{equation:goodarf}, then the Arf invariant of $\phi$ obstructs the existence of a maximal chain of admissible curves.
\end{proof}

\section{Finite generation of the admissible subgroup}\label{section:fingen}
Theorem \ref{mainthm:genset} asserts that the framed mapping class group $\Mod_{g,n}[\phi]$ is generated by any spanning configuration $\mathcal C$ of admissible Dehn twists so long as the intersection graph $\Lambda_\mathcal C$ is a tree containing $E_6$ as a subgraph (recall the definition of ``spanning configuration'' prior to the statement of Theorem \ref{mainthm:genset}). In this section, we take the first step to establishing this result. Proposition \ref{proposition:fingen} establishes that such a configuration of twists generates the admissible subgroup. In the subsequent sections we will show that there is an equality $\mathcal T_\phi = \Mod_{g,n}[\phi]$, establishing Theorem \ref{mainthm:genset}. 

Recall (c.f. the discussion preceding Theorem \ref{mainthm:genset}) that a collection $\mathcal C$ of curves is said to be an {\em $E$-arboreal spanning configuration} if each pair of curves intersects at most once, and the intersection graph is a tree containing $E_6$ as a subgraph.

\begin{proposition}[Generating the admissible subgroup]\label{proposition:fingen}
Let $\Sigma_{g,n}$ be a surface of genus $g \ge 5$ with $n \ge 1$ boundary components, and let $\phi$ be a framing of holomorphic type. Let $\mathcal C$ be an $E$-arboreal spanning configuration of admissible curves on $\Sigma_{g,n}$, and define
\[
\mathcal T_\mathcal C : = \pair{T_c \mid c \in \mathcal C}.
\]
Then $\mathcal T_\mathcal C= \mathcal T_\phi$.
\end{proposition}

The proof of Proposition \ref{proposition:fingen} closely follows the approach developed in \cite{salter_toric}. The heart of the argument (Lemma \ref{lemma:bpush}) is to show that our finite collection of twists generates a version of a point-pushing subgroup for a subsurface. This will allow us to express all admissible twists supported on this subsurface with our finite set of generators. Having shown this, we can import our method from \cite{salter_toric} (appearing below as Proposition \ref{lemma:pushmakesT}) which allows us to propagate this argument across the set of subsurfaces, proving the result. 

\subsection{Framed subsurface push subgroups}

Let $S \subset \Sigma_{g,n}$ be a subsurface and suppose $\Delta \subset \partial S$ is a boundary component. Let $\overline{S}$ denote the surface obtained from $S$ by capping $\Delta$ with a disk, and let $UT\overline{S}$ denote the associated unit tangent bundle. Recall the {\em disk--pushing homomorphism} $P: \pi_1(UT\overline{S}) \to \Mod(S)$ \cite[Section 4.2.5]{FarbMarg}. The inclusion $S \into \Sigma_{g,n}$ induces a homomorphism $i: \Mod(S) \to \Mod(\Sigma_{g,n})$ which restricts to give a {\em subsurface push homomorphism} $\mathcal P : = i \circ P$:
\[
\mathcal P: \pi_1(UT\overline{S}) \to \Mod(\Sigma_{g,n}).
\]
The {\em framed subsurface push subgroup} $\tilde \Pi(S)$ is the intersection of this image with $\Mod_{g,n}[\phi]$:
\[
\tilde \Pi(S):= \im(\mathcal P) \cap \Mod_{g,n}[\phi].
\]
Note that $\tilde \Pi(S)$ is defined relative to the boundary component $\Delta$, suppressed in the notation. In practice, the choice of $\Delta$ will be clear from context. 

There is an important special case of the above construction. Let $b \subset \Sigma_{g,n}$ be an oriented nonseparating curve satisfying $\phi(b) = -1$. The subsurface $\Sigma_{g,n} \setminus \{b\}$ has a distinguished boundary component $\Delta$ corresponding to the left-hand side of $b$. For this choice of $(S, \Delta)$, we streamline notation, defining
\[
\tilde \Pi(b) := \tilde \Pi(\Sigma_{g,n} \setminus{b})
\]
(constructed relative to $\Delta$). As $\phi(\Delta) = -1$ (oriented so that $S$ lies to the left), it follows from Lemma \ref{lemma:extension} that the framing of $S$ can be extended over the capping disk to $\overline{S}$. Such a framing of $\overline{S}$ gives rise to a section $s: \overline{S} \to UT\overline{S}$, and hence a splitting $s_*: \pi_1(\overline{S}) \to \pi_1(UT\overline{S})$. 

\begin{lemma}
If $\phi(\Delta) = -1$, there is an equality
\[
\tilde \Pi(S)  = \mathcal P (s_*(\pi_1(\overline S))).
\]
\end{lemma}
\begin{proof}
Let $x_1,\dots, x_n \in \pi_1(\overline S)$ be a system of generators such that each $x_i$ is represented by a simple based loop on $\overline S$. Under $\mathcal P \circ s_*$, each such $x_i$ is sent to a multitwist:
\[
\mathcal P (s_*(x_i)) = T_{x_{i}^{L}} T_{x_{i}^{R}}^{-1} T_{\Delta}^{\phi(x_{i}^{L})},
\]
where the curves $x_i^L, x_i^R \subset S$ are characterized by the following two conditions:
\begin{enumerate}
\item $x_i^L \cup x_i^R \cup \Delta$ form a pair of pants (necessarily lying to the left of $\Delta$),
\item $x_i^L$ (resp. $x_i^R$) lies to the left (resp. right) of $x_i$ as a based oriented curve.
\end{enumerate}
By the twist-linearity formula (Lemma \ref{lemma:HJ}.\ref{item:TL}), $\mathcal P(s_*(x_i))$ preserves $\phi$. As the set of $x_i$ generates $\pi_1(\overline S)$, it follows that $\mathcal P(s_*(\pi_1(\overline S))) \le \tilde \Pi(S)$. 

To establish the opposite containment, we recall that $s_*$ gives a splitting of the sequence
\[
\xymatrix{
1 \ar[r] & 	\Z \ar[r] &\pi_1(UT\overline{S}) \ar[r]^-{p_*} & \pi_1(\overline{S}) \ar[r] & 1
}
\]
 and so it suffices to show that $\tilde \Pi(S) \cap \ker(\mathcal P \circ p_*) = \{e\}$. Under $\mathcal P$, the generator of $\ker{p_*}$ is sent to $T_{\Delta}$. As $\phi(\Delta) = - 1$ and $\Delta$ was constructed by cutting along the non-separating curve $b \subset \Sigma_{g,n}$, the twist-linearity formula shows that $\pair{T_\Delta} \cap \Mod_{g,n}[\phi]  = \{e\}$ and the result follows. 
\end{proof}

\subsection{Generating framed push subgroups} We want to show that our finitely--generated subgroup $\mathcal T_\mathcal C$ contains a framed subsurface push subgroup $\tilde \Pi(S)$ for a subsurface $S$ that is ``as large as possible''. In Lemma \ref{lemma:inductive} below, we show that this can be accomplished inductively by successively showing containments $\tilde \Pi(S_i) \le \mathcal T_\mathcal C$ for an increasing union of subsurfaces $\dots \subset S_i \subset S_{i+1} \subset \dots$.

  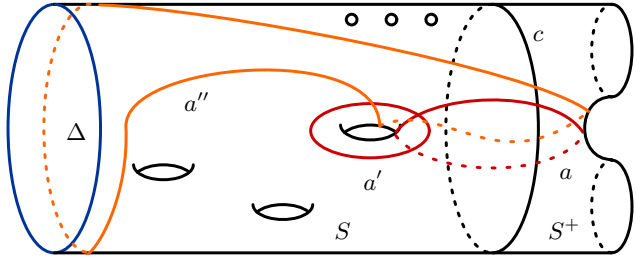
\begin{figure}[ht]
\labellist
\small
\pinlabel $\Delta$ [r] at 33.12 46.08
\pinlabel $c$ [tl] at 196.16 86.40
\pinlabel $a$ [tl] at 206.24 34.56
\pinlabel $a'$ [t] at 138.24 34.56
\pinlabel $S$ [b] at 126.72 2.88
\pinlabel $S^+$ [b] at 210.24 2.88
\pinlabel $a''$ [tl] at 64.80 63.36
\endlabellist
\includegraphics[scale=1]{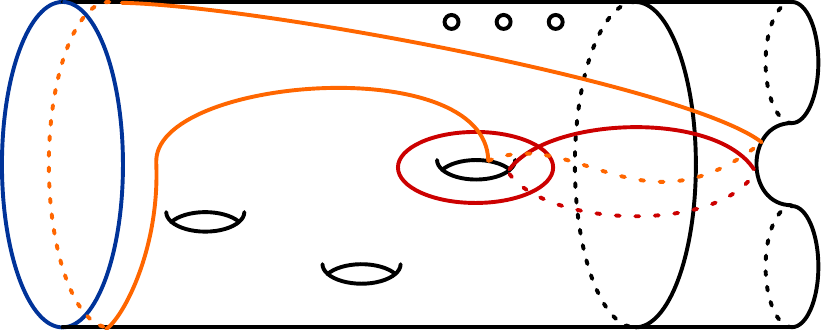}
\caption{The configuration discussed in Lemma \ref{lemma:inductive}, shown in the case where $a$ enters and exits along the same boundary component $c$.}
\label{figure:inductive}
\end{figure}
 \begin{lemma}\label{lemma:inductive} 
 Let $S \subset \Sigma_{g,n}$ be a subsurface and let $\Delta$ be a boundary component of $S$ such that $\phi(\Delta) = -1$, giving rise to the associated framed subsurface push subgroup $\tilde \Pi(S)$. Let $a \subset \Sigma_{g,n}$ be an admissible curve disjoint from $\Delta$ such that $a \cap S$ is a single essential arc (it does not matter if $a$ enters and exits $S$ by the same or by different boundary components). Let $a' \subset S$ be an admissible curve satisfying $i(a, a') = 1$. Let $S^+$ be the subsurface given by a regular neighborhood of $S \cup a$. Then  $\tilde \Pi(S^+) \le \pair{T_a, T_{a'}, \tilde \Pi(S)}$.
 \end{lemma}
 
 \begin{proof}
 Let $a'' \subset S^+$ be a curve such that $a \cup a'' \cup \Delta$ forms a pair of pants and such that $a'' \cap S$ is a single arc based at the same boundary component of $S$ as $a \cap S$. By homological coherence (Lemma \ref{lemma:HJ}.\ref{item:HC}), $a''$ is admissible. Observe that $T_a T_{a''}^{-1} \in \tilde \Pi(S^+)$.  Applying $T_a T_{a'}$ takes both $a$ and $a''$ to admissible curves on $S$, and so
 \[
 (T_a T_{a'}) T_a T_{a''}^{-1} (T_a T_{a'})^{-1} \in \tilde \Pi(S).
 \]
 Consequently $T_a T_{a''}^{-1} \in \pair{T_a, T_{a'}, \tilde \Pi(S)}$. Let $x_1, \dots, x_k \in \tilde \Pi(S)$ be a generating set. The inclusion $S \into S^+$ induces an inclusion $\tilde \Pi(S) \into \tilde \Pi(S^+)$, and $\tilde \Pi(S^+)$ is generated by $x_1, \dots, x_k$ and $T_a T_{a''}^{-1}$. These elements are all contained in the group $\pair{T_a, T_{a'}, \tilde \Pi(S)}$.
 \end{proof}

\para{Generation via networks} The inductive criterion Lemma \ref{lemma:inductive} leads to the notion of a {\em network}, which is a configuration of curves designed such that Lemma \ref{lemma:inductive} can be repeatedly applied. Here we discuss the basic theory.
\begin{definition}[Networks]\label{definition:network}
Let $S$ be a surface of finite type. For the purposes of the definition, punctures and boundary components are interchangeable; we convert both into boundary components. A {\em network} on $S$ is any collection $\mathcal N = \{a_1,\dots, a_n\}$ of simple closed curves ({\em not} merely isotopy classes) on $S$ such that $\#(a_i \cap a_j) \le 1$ for all pairs of curves $a_i, a_j \in \mathcal N$, and such that there are no triple intersections. 

A network $\mathcal N$ has an associated {\em intersection graph} $\Lambda_{\mathcal N}$, whose vertices correspond to curves $x \in \mathcal N$, with vertices $x,y$ adjacent if and only if $\#(x\cap y) = 1$.  A network is said to be {\em connected} if $\Lambda_{\mathcal N}$ is connected, and {\em arboreal} if $\Lambda_{\mathcal N}$ is a tree. A network is {\em filling} if 
\[
S \setminus \bigcup_{a \in \mathcal N} a
\]
is a disjoint union of disks and boundary--parallel annuli.
\end{definition}

A network $\mathcal N$ determines a subgroup $\mathcal T_\mathcal N \le \Mod(\Sigma_{g,n})$ by taking the group generated by the Dehn twists about curves in $\mathcal N$:
\[
\mathcal T_\mathcal N := \pair{T_a \mid a \in \mathcal N}.
\]

The following appears in slightly modified form as \cite[Lemma 9.4]{salter_toric}. 
\begin{lemma}\label{lemma:networkgenset}
Let $S \subset \Sigma_{g,n}$ be a subsurface with a boundary component $\Delta$ satisfying $\phi(\Delta) = -1$. Let $\mathcal N$ be a network of admissible curves on $S$ that is connected, arboreal, and filling, and suppose that there exist $a, a' \in \mathcal N$ such that $a \cup a' \cup \Delta$ forms a pair of pants. Then $\tilde \Pi(S) \le \mathcal T_\mathcal N$. 
\end{lemma}

\subsection{The key lemma} 

Proposition \ref{lemma:pushmakesT}, to be stated below, gives a criterion for a group $H$ to contain the admissible subgroup $\mathcal T_\phi$. It asserts that containing a framed subsurface push subgroup of the form $\tilde \Pi(b)$ is ``nearly sufficient.'' In preparation for this, we show here that $\mathcal T_\mathcal C$ contains such a subgroup. Ideally, we would like to use the network generation criterion (Lemma \ref{lemma:networkgenset}), but the configuration $\mathcal C$ does not satisfy the hypotheses and so more effort is required. 

\begin{figure}[h]
\labellist
\small
\pinlabel $b$ at 60 70
\pinlabel $c$ at 100 22
\endlabellist
\includegraphics[scale=1]{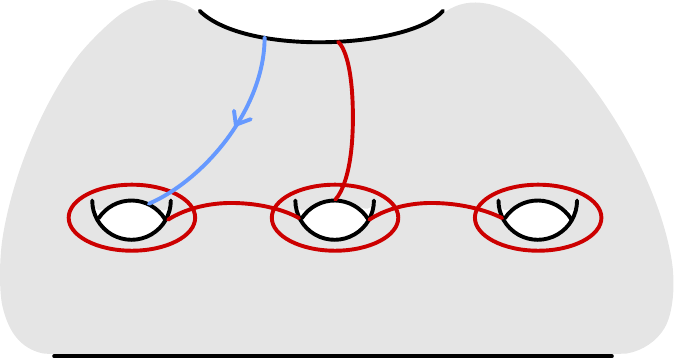}
\caption{The curve $b$ of Lemma \ref{lemma:bpush}, shown in relation to the $E_6$ subgraph. $b$ is constructed so as to be disjoint from all curves intersecting the central vertex $c$ of the $E_6$ subgraph, but it may intersect other elements of $\mathcal C$ not pictured in the figure.}
\label{figure:b}
\end{figure}

\begin{lemma}\label{lemma:bpush}
Let $\mathcal C$ be an $E$-arboreal spanning configuration of admissible curves, and let $b$ be the curve indicated in Figure \ref{figure:b}. Then $\tilde \Pi(b) \le \mathcal T_\mathcal C$.
\end{lemma}

This will be proved in four steps. In each stage, we will consider a subconfiguration $\mathcal C_k \subset \mathcal C$ and the associated subsurface $S_k$ spanned by these curves. We define $S_k' \subset S_k$ by removing a regular neighborhood of $b$, and we show that $\tilde \Pi(S_k') \le \mathcal T_\mathcal C$. 

 In the first step, recorded as Lemma \ref{lemma:S0} below, we take $\mathcal C_0$ to be the $D_5$ subconfiguration of the $E_6$ configuration. In the second step (Lemma \ref{lemma:S1}), we take $\mathcal C_1 = E_6$. In the third step (Lemma \ref{lemma:S2}), we take $\mathcal C_2$ to be the union of $E_6$ with all curves $c \in \mathcal C$ intersecting $b$, and finally the full surface ($\mathcal C_3 = \mathcal C$) is dealt with in Step 4.

\para{\boldmath Step 1: $D_5$}
 \begin{figure}[h]
\labellist
\small
\pinlabel $S_0'$ at 100 60 
\pinlabel $S_1'$ at 350 60
\endlabellist
\includegraphics[scale=1]{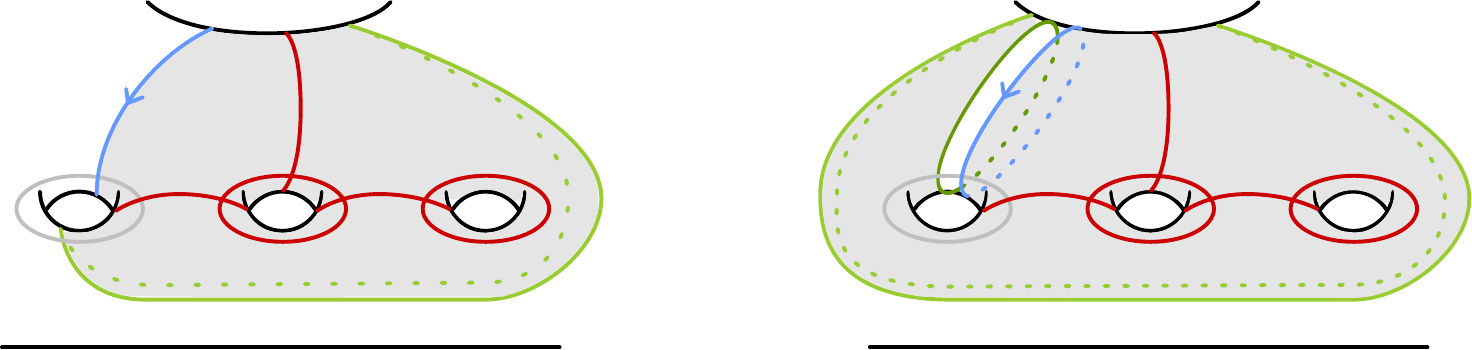}
\caption{At left, the surface $S_0 = S_0'$. At right, the surface $S_1'$.}
\label{figure:S0}
\end{figure}
\begin{lemma}\label{lemma:S0}
Let $S_0' \subset \Sigma_{g,n}$ be the subsurface shown in Figure \ref{figure:S0}. There is a containment $\tilde \Pi(S_0') \le \mathcal T_\mathcal C$.
\end{lemma}

\begin{proof}
Let $\mathcal C_0$ be the network shown in Figure \ref{figure:S0} consisting of the five red (dark) curves. This satisfies the hypotheses of Lemma \ref{lemma:networkgenset}, so that $\tilde \Pi(S_0') \le \mathcal T_{\mathcal C_0}$. Each element of $\mathcal C_0$ is an element of $\mathcal C$, so that the claim follows. 
\end{proof}

\para{\boldmath Step 2: $E_6$} 

\begin{figure}[h]
\labellist
\Small
\pinlabel $d_1$ at 330 20
\pinlabel $d_2$ at 387 70
\pinlabel $a''$ at 405 55
\pinlabel $a$ at 155 70
\pinlabel $a'$ at 115 55
\endlabellist
\includegraphics[scale=1]{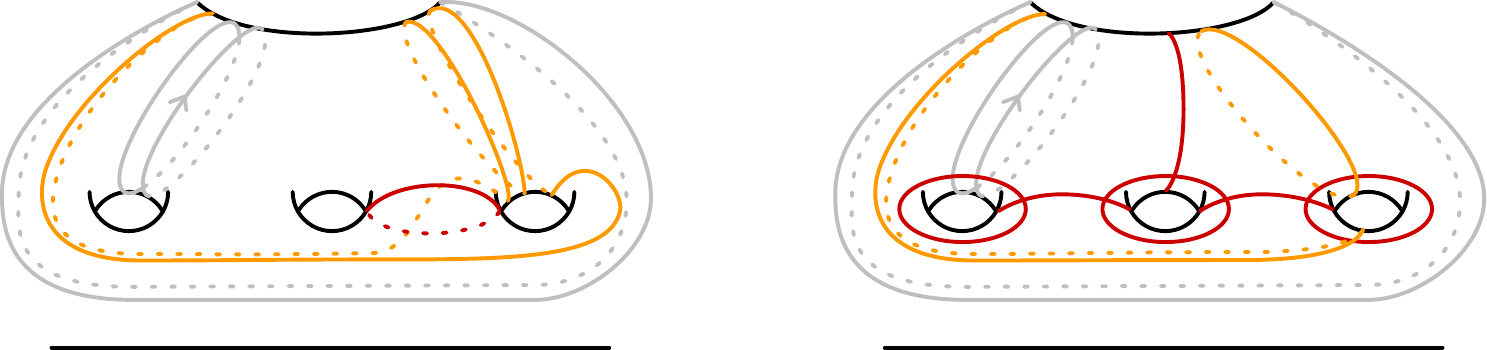}
\caption{At left, the curves $a$ and $a'$; the latter is an element of the $E_6$ configuration inside $\mathcal C$. At right, the boundary components $d_1, d_2$ for the configuration of $D_5$ type, and the curve $a''$, also part of $E_6 \subset \mathcal C$.}
\label{figure:lastcurve}
\end{figure}
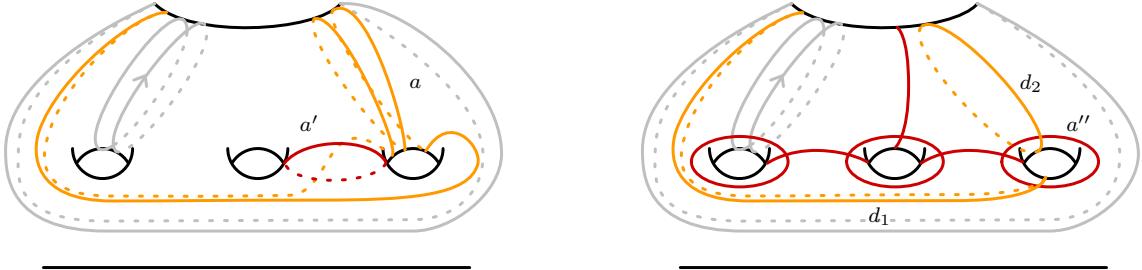

\begin{lemma}\label{lemma:S1}
Let $S_1' \subset \Sigma_{g,n}$ be the subsurface shown in Figure \ref{figure:S0}. There is a containment $\tilde \Pi (S_1') \le \mathcal T_\mathcal C$. 
\end{lemma}

\begin{proof}
We appeal to Lemma \ref{lemma:inductive}. It suffices to find a curve $a \subset \Sigma_{g,n}$ such that (1) $a \cap S_0'$ is a single arc, (2) $S_1$ deformation-retracts onto $S_0' \cup a$, (3) there is a curve $a' \subset S_0'$ such that $i(a,a') = 1$ and $T_{a'} \in \mathcal T_\mathcal C$, and (4) $T_a \in \mathcal T_\mathcal C$. A curve $a$ satisfying (1),(2),(3) is shown in Figure \ref{figure:lastcurve}. 

We claim that $T_a \in \mathcal T_\mathcal C$. To see this, we consider the right-hand portion of Figure \ref{figure:lastcurve}. We see that five of the curves in the configuration of $E_6$ type determine a configuration of type $D_5$; the boundary components of the subsurface spanned by these curves are denoted $d_1$ and $d_2$. Applying the {\em $D_5$ relation} (c.f. \cite[Lemma 5.9]{CS_strata2}) to this configuration shows that $T_{d_1} T_{d_2}^3 \in \mathcal T_\mathcal C$. We then set 
\[
a = T_{d_1} T_{d_2}^3(a''),
\]
and as $a'' \in \mathcal C$, it follows that $T_a \in \mathcal T_ \mathcal C$ as claimed. 
\end{proof}

\para{\boldmath Step 3: Curves intersecting $b$}

\begin{figure}[h]
\labellist
\tiny
\pinlabel $c_i$ at 60 195
\pinlabel $a_1$ at 65 120
\pinlabel $a_2$ at 105 130
\pinlabel $a_3$ at 142 120
\pinlabel $a_0$ at 147 160
\pinlabel $a$ at 405 170
\pinlabel $a_4$ at 395 142
\pinlabel $c_i$ at 65 53
\pinlabel $a$ at 297 43
\endlabellist
\includegraphics[scale=0.8]{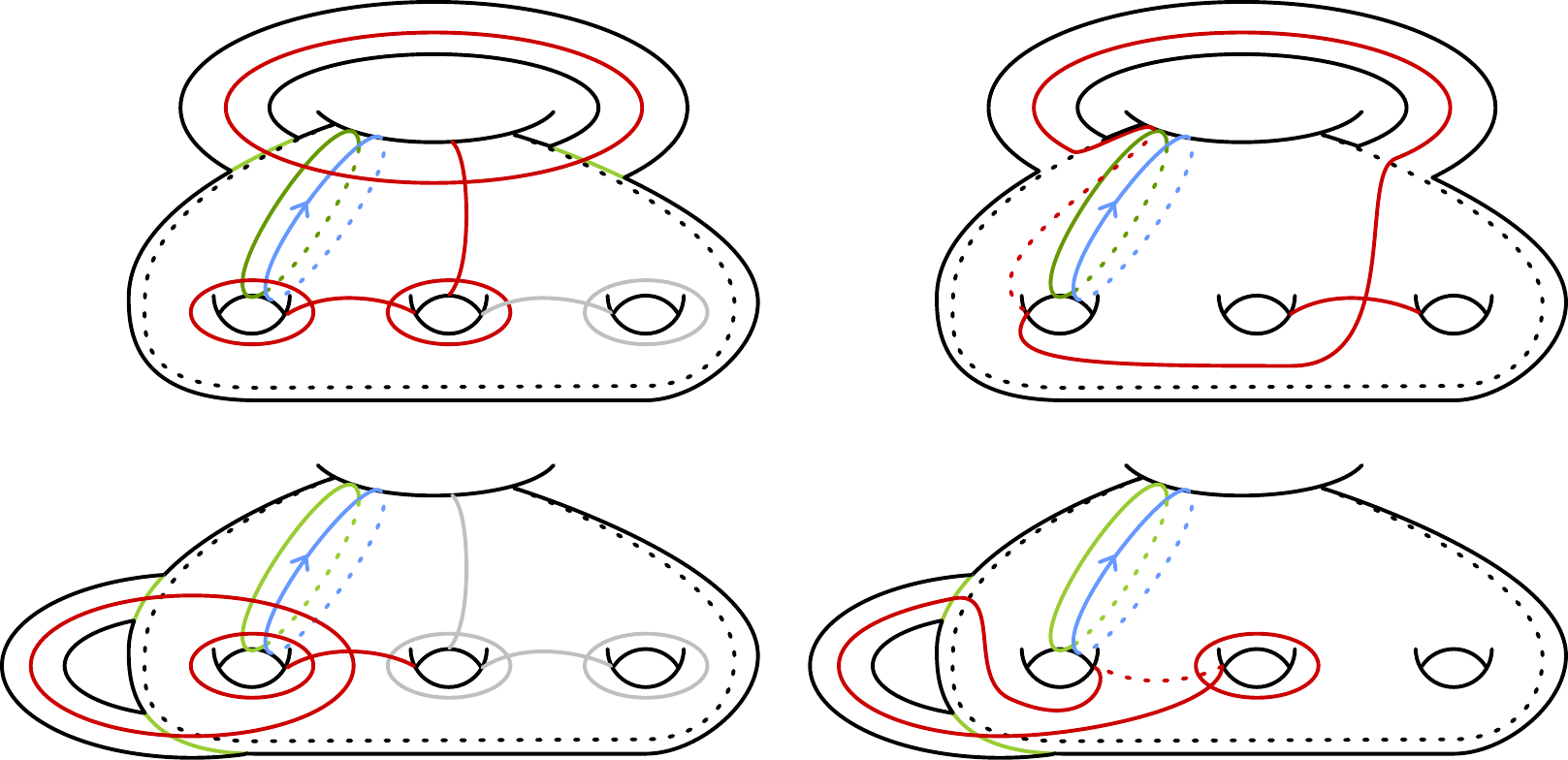}
\caption{At left, a curve $c_i \in \mathcal C_2$. At right, a twist $a$ satisfying the hypotheses of Lemma \ref{lemma:inductive}. The top and bottom of the figure depicts the two possible intersection patterns for $c_i$ with $S_1'$.}
\label{figure:bint}
\end{figure}
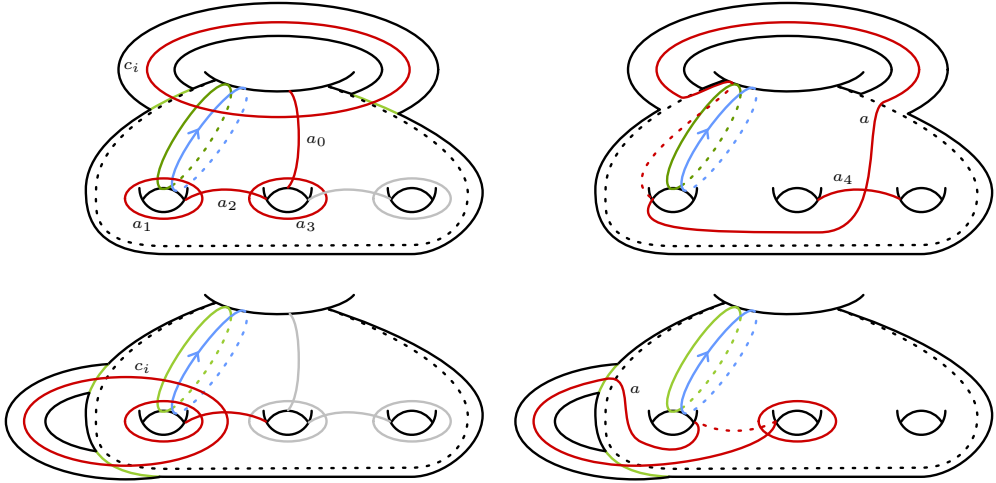

\begin{lemma}\label{lemma:S2}
Let $\mathcal C_2$ be the configuration of curves given as the union of $E_6 \subset \mathcal C$ with all curves $c_i \in \mathcal C$ such that $i(c_i, b) \ne 0$. Let $S_2$ be the surface spanned by these curves, and let $S_2'$ be obtained from $S_2$ by removing a neighborhood of $b$. Then $\tilde \Pi(S_2) \le \mathcal T_\mathcal C$. 
\end{lemma}

\begin{proof}
Since the intersection graph of $\mathcal C$ is a tree, a curve $c_i \in \mathcal C_2 \setminus E_6$ must be in one of the two configurations shown in Figure \ref{figure:bint}: it must intersect exactly one of the curves $a_0$ or $a_2$. Moreover, distinct $c_i, c_j \in \mathcal C_2 \setminus E_6$ must be pairwise-disjoint. Thus we can attach the curves $c_i$ in an arbitrary order to assemble $S_2'$ from $S_1'$, appealing to Lemma \ref{lemma:inductive} at each step. 

The right-hand portion of Figure \ref{figure:bint} shows a curve $a$ that satisfies the hypotheses of Lemma \ref{lemma:inductive} for this pair of subsurfaces. $a$ is obtained from $c_i$ via a sequence of twists about curves in $\mathcal C$:
\[
a = T_{a_1}T_{a_2}T_{a_3}T_{a_0}(c_i)
\]
in the top scenario, and
\[
a = T_{a_1}^{-1} T_{a_2}^{-1}(c_i)
\]
in the bottom. Thus, for each such curve $c_i \in \mathcal C_2 \setminus E_6$, the associated curve $a$ satisfies $T_a \in \mathcal T_\mathcal C$. The claim now follows from repeated applications of Lemma \ref{lemma:inductive}.
\end{proof}

\para{Step 4: Attaching the remaining curves}

The proof of Lemma \ref{lemma:bpush} now follows with no further special arguments.

\begin{proof}[Proof of Lemma \ref{lemma:bpush}] 
Step 3 (Lemma \ref{lemma:S2}) shows that $\tilde \Pi(S_2') \le \mathcal T_\mathcal C$, where $S_2$ is the span of $E_6$ and all curves intersecting $a_0$ and $a_2$, and $S_2'$ is obtained from $S_2$ by removing a neighborhood of $b$. Let $c_i \in \mathcal C \setminus  \mathcal C_2$ be adjacent to some element $c_j \in \mathcal C_2$. Then Lemma \ref{lemma:inductive} applies directly to the pair $a = c_i$ and $a' = c_j$. We repeat this process next with curves $c_i$ of graph distance $2$ to $\mathcal C_2$, then graph distance $3$, etc., until the vertices of $\mathcal C \setminus \mathcal C_2$ are exhausted. At the final stage, we have shown that $\tilde \Pi(\Sigma_{g,n}\setminus \{b\}) = \tilde \Pi(b) \le \mathcal T_\mathcal C$ as claimed.
 \end{proof}

\subsection{Finite generation of the admissible subgroup}

Proposition \ref{lemma:pushmakesT} below is taken from \cite[Proposition 8.2]{salter_toric}. There, it is formulated for $r$--spin structures on closed surfaces of genus $g \ge 5$, but the result and its proof hold {\em mutatis mutandis} for framings of $\Sigma_{g,n}$ with $g \ge 5$. 
\begin{proposition}[C.f. Proposition 8.2 of \cite{salter_toric}]\label{lemma:pushmakesT}
Let $\phi$ be a framing of $\Sigma_{g,n}$ for $g \ge 5$. Let $(a_0, a_1, b)$ be an ordered $3$--chain of curves with $\phi(a_0) = \phi(a_1) = 0$ and $\phi(b) = -1$. Let $H \le \Mod(\Sigma_{g,n})$ be a subgroup containing $T_{a_0}, T_{a_1}$, and the framed subsurface push subgroup $\tilde \Pi(b)$. Then $H$ contains $\mathcal T_\phi$. 
\end{proposition}

\proofof{Proposition \ref{proposition:fingen}} Since $\mathcal T_\mathcal C \le \mathcal T_\phi$ by construction, it suffices to apply Proposition \ref{lemma:pushmakesT} with the subgroup $H = \mathcal T_\mathcal C$. Lemma \ref{lemma:bpush} asserts that $\tilde \Pi(b) \le \mathcal T_\mathcal C$. When $\mathcal C$ is of type 1 (resp. 2), the chain $(a_4, a_3, b)$ (resp. $(a_8,a_7, b)$) satisfies the hypotheses of Proposition \ref{lemma:pushmakesT}. The result now follows by Proposition \ref{lemma:pushmakesT}. \qed

We observe that this argument can also be combined with results of our earlier paper \cite{CS_strata2} to give a vast generalization of the types of configurations which generate $r$-spin mapping class groups (c.f. Section \ref{subsection:framings}). In particular, the following results gives many new generating sets for the closed mapping class group $\Mod_g$.

\begin{corollary}\label{corollary:rspingen}
Let $\mathcal C$ denote a filling network of curves on a {\em closed} surface $\Sigma_g$ with $g \ge 5$. Suppose that the intersection graph $\Lambda_\mathcal C$ is a tree which contains the $E_6$ Dynkin diagram as a subgraph and that $\mathcal C$ cuts the surface into $n$ polygons with $4(k_1 + 1), \ldots, 4(k_n +1)$ many sides.

Set $r = \gcd(k_1, \ldots, k_n)$; then there exists an $r$-spin structure $\hat \phi$ on $\Sigma_g$ such that 
\[\Mod_g[\hat \phi] = \pair{ T_c \mid c \in \mathcal C}.\]
\end{corollary}
\begin{proof}
The $r$-spin structure $\hat \phi$ is uniquely determined by stipulating that each curve $c \in \mathcal C$ is admissible.
To see that the twists in $\mathcal C$ generate the stabilizer of this spin structure, we first observe that by Proposition 6.1 of \cite{CS_strata2}, the $\hat \phi$-admissible subgroup $\mathcal T_{\hat \phi}$ is equal to $\Mod_g[\hat \phi]$. Therefore, it suffices to prove that $\pair{ T_c \mid c \in \mathcal C} =\mathcal T_{\hat \phi} $.

Let $S$ denote a neighborhood of the curve system $\mathcal C$; by insisting that each curve of $\mathcal C$ is admissible in $S$, we see that $S$ is naturally a surface of genus $g$ equipped with a framing $\phi$ of signature $(-1 - k_1, \ldots, -1-k_n)$. Now by Proposition \ref{proposition:fingen},
\[
\mathcal T_\phi = \pair{ T_c \mid c \in \mathcal C}
\]
and as $\Sigma_g$ is obtained from $S$ by capping off each boundary component with a disk, we need only show that $\mathcal T_\phi$ surjects onto $\mathcal T_{\hat \phi}$ under the capping homomorphism.

To show the desired surjection, consider any $\hat \phi$-admissible curve $c$ on $\Sigma_g$. Pick any $\tilde c$ on $S$ which maps to $c$ under capping; then since $\hat \phi$ is just the reduction of $\phi$ mod $r$, necessarily
\[\phi(\tilde c) = rN \text{ for some } N \in \Z.\]
For each boundary component $\Delta_i$ of $S$, pick some loop $\gamma_i$ based on $\Delta_i$ which intersects $\tilde c$ exactly once. Now since $r= \gcd(k_1, \ldots, k_n)$ there is some linear combination 
\[m_1 k_1 + \ldots + m_n k_n = r\]
and so by the twist-linearity formula (Lemma \ref{lemma:HJ}.\ref{item:TL}), the curve
\[\tilde c' = \left( \mathcal P(\gamma_1)^{m_1} \ldots \mathcal P(\gamma_n)^{m_n} \right)^{-N}(\tilde c)\]
must be $\phi$-admissible, where $\mathcal P(\gamma_i)$ denotes the push of the boundary component $\Delta_i$ about $\gamma_i$.

But now $T_{\tilde c'}$ is in $\mathcal T_\phi$ and $\mathcal P(\gamma_1)^{m_1} \ldots \mathcal P(\gamma_n)^{m_n}$ is in the kernel of the boundary-capping map, and so the image of $T_{\tilde c'}$ in $\Mod_g$ is the same as that of $T_{\tilde c}$, which by construction is $T_c$. Hence  $\mathcal T_\phi$ surjects onto $\mathcal T_{\hat \phi}$, finishing the proof.
\end{proof}

\section{Separating twists and the single boundary case}\label{section:basecase}
\subsection{Separating twists}\label{section:septwist}
We come now to the first of two sections dedicated to showing the equality $\mathcal T_\phi = \Mod_{g,n}[\phi]$. This will be accomplished by induction on $n$. In this section, we establish the base case $n = 1$, while in the next section we carry out the inductive step. 

The base case $n = 1$ is in turn built around a close connection with the theory of $r$--spin structures on closed surfaces (c.f. Definition \ref{definition:rspin}) as studied in the prior papers \cite{salter_toric, CS_strata2}. We combine this work with a version of the Birman exact sequence (c.f. \eqref{equation:framedBES}) to reduce the problem of showing $\mathcal T_\phi = \Mod_{g,1}[\phi]$ to the problem of showing that $\mathcal T_\phi$ contains a sufficient supply of Dehn twists about {\em separating} curves. 

Below and throughout, the group $\mathcal K_{g,1}$ is defined to be the group generated by separating Dehn twists:
\[
\mathcal K_{g,1} := \pair{T_c \mid c \subset \Sigma_{g,1} \mbox{ separating}}.
\]
$\mathcal K_{g,1}$ is known as the {\em Johnson kernel}. It is a deep theorem of Johnson that $\mathcal K_{g,1}$ can be identified with the kernel of a certain ``Johnson homomorphism'' \cite{Johnson_kernel}, but we will not need to pursue this any further here. 

\begin{proposition}\label{proposition:minsep}
Fix $g \ge 5$, and let $\phi$ be a framing of $\Sigma_{g,1}$. Then $\mathcal K_{g,1} \le \mathcal T_\phi$.
\end{proposition}

A separating curve $c \subset (\Sigma_{g,1}, \phi)$ has two basic invariants. To define these, let $\Int(c)$ be the component of $\Sigma_{g,1} \setminus c$ that does not contain the boundary component of $\Sigma_{g,1}$. The first invariant of $c$ is the {\em genus} $g(c)$, defined as the genus of $\Int(c)$. The second invariant of $c$ is the {\em Arf invariant} $\Arf(c)$. When $g(c) \ge 2$, define $\Arf(c)$ to be the Arf invariant of $\phi\mid_{\Int(c)}$. As discussed in Section \ref{section:framingaction}, the Arf invariant in genus $1$ is special. For uniformity of notation later, if $g(c) = 1$ and $\Arf_1(\Int(c)) \ne 0$, define
\[
\Arf(c) := \Arf_1(\phi\mid_{\Int(c)}) +1 \pmod 2.
\]
In the special case where $\Arf_1(\phi\mid_{\Int(c)}) = 0$, we declare $\Arf(c)$ to be the symbol $1^+$ (for the purposes of arithmetic, we treat this as $1 \in \Z/2\Z$). Altogether, we define the {\em type} of a separating curve $c$ to be the pair $(g(c), \Arf(c))$. 

\begin{lemma}\label{lemma:basictwists}
Let $\phi$ be a framing of a surface $\Sigma_{g,1}$. Let $c$ be a separating curve of type $(g,\epsilon)$. For the pairs of $(g, \epsilon)$ listed below, the separating twist $T_c$ is contained in $\mathcal T_\phi$.
\begin{enumerate}
\item $(1+4k,1)$ for $k \ge 1$,
\item $(2+4k,1)$ for $k \ge 0$,
\item $(3+4k,0)$ for $k \ge 0$,
\item $(4k,0)$ for $k \ge 1$,
\item $(1,1^+)$,
\item $(3,1)$.
\end{enumerate}
\end{lemma}
\begin{proof}
In cases (1)--(5), the Arf invariant of the surface agrees with the Arf invariant of a surface of the same genus which supports a maximal chain of admissible curves. By the change--of--coordinates principle for framed surfaces (c.f. Proposition \ref{proposition:CCP}.2), it follows that every such surface supports a maximal chain of admissible curves. Applying the chain relation (c.f. \cite[Proposition 4.12]{FarbMarg}) to this maximal chain shows that the separating twist about the boundary component is an element of $\mathcal T_\phi$. 

Consider now case (6), where the subsurface $S_c$ determined by $c$ has genus $3$ and Arf invariant $1$. In this case, the framed change--of--coordinates principle implies that $S_c$ supports a configuration $a_0, \dots, a_5$ of admissible curves with intersection pattern given by the $E_6$ Dynkin diagram. By the ``$E_6$ relation'' (c.f. \cite[Theorem 1.4]{Matsumoto}), $T_c$ can be expressed as a product of the admissible twists $T_{a_0}, \dots, T_{a_5}$.
\end{proof}

In the proof of Proposition \ref{proposition:minsep}, it will be important to understand the additivity properties of the Arf invariant when a surface is decomposed into subsurfaces.

\begin{lemma}[c.f. Lemma 2.11 of \cite{RW}]\label{lemma:Arfadd}
Suppose that $(\Sigma, \phi)$ is a framed surface and $\mathbf{c} = \bigcup c_i$ is a separating multicurve. Let $(S_1, \psi_1)$ and $(S_2, \psi_2)$ denote the two components of $\Sigma \setminus \mathbf{c}$, equipped with their induced framings. Then
\[\Arf(\phi) = \Arf(\psi_1) + \Arf(\psi_2) + \sum_{c_i \in \mathbf{c}} \left( \phi(c_i) + 1 \right) \pmod 2.\]
In particular, if each $c_i$ is separating and bounds a {\em closed} subsurface on one side, then each $\phi(c_i)$ is odd and hence the Arf invariant is additive.
\end{lemma}

The proof of Proposition \ref{proposition:minsep} is built around the well-known {\em lantern relation}. 
\begin{figure}[h]
\labellist
\Small
\pinlabel $a$ [b] at 37.44 61.88
\pinlabel $b$ [c] at 5.32 34.56
\pinlabel $c$ [t] at 37.44 -0.88
\pinlabel $d$ [l] at 70.88 34.56
\pinlabel $x$ [r] at 29.36 33.12
\pinlabel $y$ [tr] at 47.52 44.64
\pinlabel $z$ [br] at 47.52 23.04
\endlabellist
\includegraphics[scale=1]{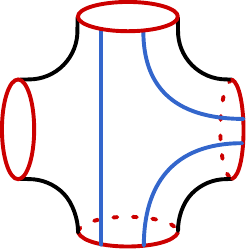}
\caption{The lantern relation. The arcs labeled $x,y,z$ determine curves by taking a regular neighborhood of the arc and the incident boundary components.}
\label{figure:lantern}
\end{figure}
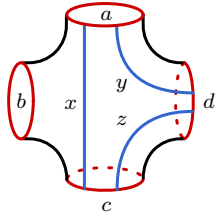
\begin{lemma}[Lantern relation]\label{lemma:lantern}
For the curves $a,b,c,d,x,y,z$ of Figure \ref{figure:lantern}, there is a relation
\[
T_a T_b T_c T_d = T_x T_y T_z.
\]
\end{lemma}

We will use the lantern relation to ``manufacture'' new separating twists using an initially--limited set of twists.

\begin{proof}[Proof of Proposition \ref{proposition:minsep}] According to \cite[Theorem 1]{johnson_PAMS}, it suffices to show that $T_c \in \mathcal T_\phi$ for $c$ a separating curve of genus $1$ or $2$. If $c$ has type $(2,1)$, then $T_c \in \mathcal T_{\phi}$ by Lemma \ref{lemma:basictwists}. For the remaining types $(1,1), (1,0), (2,0)$, we will appeal to a sequence of lantern relations (Configurations (A), (B), (C)) as shown in Figure \ref{figure:configs}. Each of the configurations below occupy a surface of genus $\le 4$, and by hypothesis, $g \ge 5$. Thus, in each configuration, the specified winding numbers do not constrain the Arf invariant. Therefore by the framed change--of--coordinates principle (Proposition \ref{proposition:CCP}), there is no obstruction to constructing such configurations. We also remark that we will use the additivity of the Arf invariant (Lemma \ref{lemma:Arfadd}) without comment throughout.

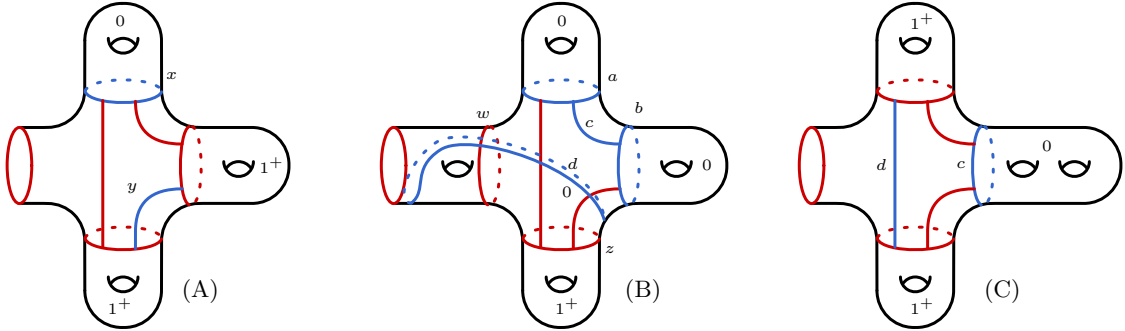
\begin{figure}[h]
\labellist
\small
\pinlabel (A) [l] at 63.36 14.40
\pinlabel (B) [l] at 230.40 14.40
\pinlabel (C) [l] at 365.76 14.40
\tiny
\pinlabel $x$ [bl] at 58.36 92.16
\pinlabel $y$ [br] at 51.84 48.96
\pinlabel $0$ [c] at 43.20 116.08
\pinlabel $1^+$ [c] at 100.24 61.92
\pinlabel $1^+$ [c] at 43.20 8.64
\pinlabel $a$ [bl] at 224.64 91.48
\pinlabel $b$ [bl] at 234.92 79.20
\pinlabel $c$ [br] at 223.76 73.88
\pinlabel $0$ [c] at 209.68 116.08
\pinlabel $0$ [c] at 264.28 61.92
\pinlabel $1^+$ [c] at 211.68 8.64
\pinlabel $0$ [bl] at 207.56 47.60
\pinlabel $d$ [bl] at 209.56 58.36
\pinlabel $1^+$ [c] at 345.48 116.08
\pinlabel $1^+$ [c] at 345.48 8.64
\pinlabel $0$ [c] at 392.44 69.00
\pinlabel $z$ at 227.64 30
\pinlabel $w$ at 180 81
\pinlabel $c$ at 360 62
\pinlabel $d$ at 330 62
\endlabellist
\includegraphics[scale=1]{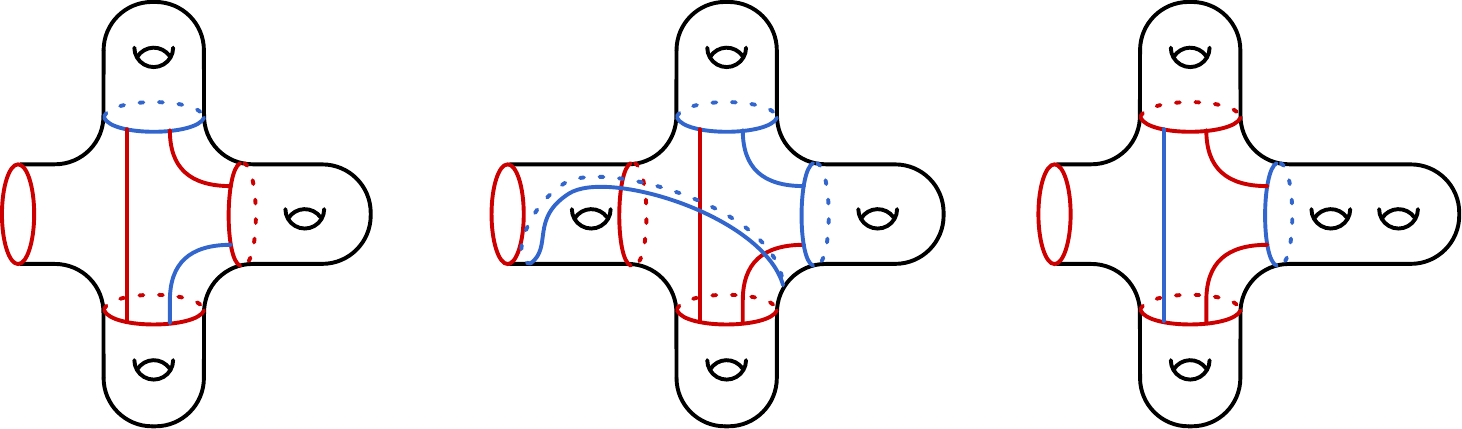}
\caption{The lantern relations used in the proof of Proposition \ref{proposition:minsep}. Curves and arcs colored \red{red} correspond to twists known to be in $\mathcal T_\phi$ from Lemma \ref{lemma:basictwists}, while those in \blue{blue} correspond to twists not yet known to be in $\mathcal T_\phi$. Numbers inside subsurfaces indicate Arf invariants.} 
\label{figure:configs}
\end{figure}

We say that separating curves $x, y \subset \Sigma_{g,1}$ are {\em nested} if there is a containment $\Int(x) \subset \Int(y)$ or $\Int(y) \subset \Int(x)$. Configuration (A) shows that $T_x T_y^{-1} \in \mathcal T_\phi$ for any $x$ of type $(1,0)$ and $y$ of type $(2,0)$ such that $x$ and $y$ are not nested. Turning to Configuration (B), we apply the lantern relation to the curves on the subsurface bounded by $a,b,z,w$ to find that $T_a T_b T_c^{-1} \in \mathcal T_\phi$; here $c$ is any curve of type $(2,0)$ and $a,b$ have type $(1,0)$ and are nested inside $c$. Applying Configuration (A) with $x = a, y = c$ and then with $x = b, y = c$ shows that $T_c^{-1} T_d^2 \in \mathcal T_\phi$ for an arbitrary pair $c,d$ of curves of type $(2,0)$. 

Consider now Configuration (C). The associated lantern relation shows that $T_c T_d^{-1} \in \mathcal T_\phi$ for $c,d$ again both of type $(2,0)$. As also $T_c^{-1} T_d^2 \in \mathcal T_\phi$ by the above paragraph, it follows that $T_c \in \mathcal T_\phi$ for $c$ an arbitrary curve of type $(2,0)$. 

Returning to Configuration (A), it now follows that $T_x \in \mathcal T_\phi$ for $x$ an arbitrary curve of type $(1,0)$. It remains only to show $T_a \in \mathcal T_\phi$ for $a$ a curve of type $(1,1)$. To obtain this, we return to Configuration (B), but replace the curve of type $(1,1^+)$ with a general curve of type $(1,1)$. The remaining twists in the lantern relation are now all known to be elements of $\mathcal T_\phi$, and hence curves of type $(1,1)$ are elements of $\mathcal T_\phi$ as well.
\end{proof}

\subsection{The minimal stratum}\label{section:minstrat}
We are now prepared to prove the main result of the section.
\begin{proposition}\label{proposition:minstrat}
Let $g \ge 5$ be given, and let $\phi$ be a framing of $\Sigma_{g,1}$. Then there is an equality
\[
\mathcal T_\phi = \Mod_{g,1}[\phi].
\]
\end{proposition}

As discussed above, this will be proved by relating the framing $\phi$ on $\Sigma_{g,1}$ to a $(2g-2)$--spin structure on $\Sigma_g$ by way of a version of the Birman exact sequence. In the standard Birman exact sequence for the capping map $p: \Sigma_{g,1} \to \Sigma_g$, the kernel is given by the subgroup $\pi_1(UT\Sigma_g)$. In Lemma \ref{lemma:framedBES}, the subgroup $H_g \le \pi_1(UT\Sigma_g)$ is defined to be the preimage in $\pi_1(UT\Sigma_g)$ of $[\pi_1(\Sigma_g), \pi_1(\Sigma_g)] \le \pi_1(\Sigma_g)$.
\begin{lemma}\label{lemma:framedBES}
Let $\phi$ be a framing of $\Sigma_{g,1}$. Then there is $(2g-2)$--spin structure $\widehat \phi$ on $\Sigma_g$ such that the boundary-capping map $p: \Mod_{g,1} \to \Mod_g$ induces the following exact sequence:
\begin{equation}\label{equation:framedBES}
1 \to H_g \to \Mod_{g,1}[\phi] \to \Mod_{g}[\widehat \phi].
\end{equation}
\end{lemma}

\begin{proof}
The framing $\phi$ determines a nonvanishing vector field on $\Sigma_{g,1}$. Capping the boundary component, this can be extended to  vector field on $\Sigma_g$ with a single zero. This vector field gives rise to the $(2g-2)$-spin structure $\widehat \phi$ (c.f. Section \ref{subsection:framings}), and if $f \in \Mod_{g,1}$ preserves $\phi$, then necessarily $p(f)$ preserves $\widehat \phi$. 

It remains to show that $\ker(p) \cap \Mod_{g,1}[\phi] = H_g$. Since $H_g \le \mathcal K_{g,1} \le \Mod_{g,1}[\phi]$, one containment is clear. For the converse, we first consider the action of a simple based loop $\gamma \in \pi_1(UT\Sigma_g)$ on the winding number of an arbitrary simple closed curve $a$. Let $\gamma_{L}$ (resp. $\gamma_{R}$) denote the curves on $\Sigma_{g,1}$ lying to the left (resp. right) of $\gamma$. Then $\gamma$ acts via the mapping class $\mathcal P(\gamma) = T_{\gamma_L} T_{\gamma_R}^{-1}$. The twist-linearity formula (Lemma \ref{lemma:HJ}.\ref{item:TL}) as applied to $\mathcal P(\gamma)$ shows that
\begin{eqnarray}\label{equation:pushformula}
\phi(\mathcal P(\gamma)(a)) 	=  &\phi(a) + \pair{a, \gamma_{L}} \phi(\gamma_{L}) - \pair{a, \gamma_{R}} \phi(\gamma_{R})\nonumber \\
							= &\phi(a) + \pair{a, \gamma} (\phi(\gamma_{L}) - \phi(\gamma_{R}))\nonumber \\
							= &\phi(a) + \pair{a, \gamma} (2 - 2g).
\end{eqnarray}
Here, the second equality holds since $\gamma_{L}, \gamma_{R}$, and $\gamma$ all determine the same homology class, and the third equality holds by homological coherence (Lemma \ref{lemma:HJ}.\ref{item:HC}), since $\gamma_{L} \cup \gamma_{R} \cup \partial \Sigma_{g,1}$ cobound a pair of pants and necessarily $\phi(\Delta_1) = 1 - 2g$.

This formula will show that $\mathcal P(\gamma) \in \Mod_{g,1}[\phi]$ if and only if $\gamma \in H_g$. To see this,  let $\gamma$ be an arbitrary curve, not necessarily simple, and factor $\gamma = \gamma_1 \dots \gamma_n$ with each $\gamma_i$ simple. Since $\mathcal P(\gamma_i)$ acts trivially on homology, there is an equality $[\mathcal P(\gamma_1 \dots \gamma_i)(a)] = [a]$ of elements of $H_1(\Sigma_{g,1};\Z)$ for each $i = 1, \dots, n$, and hence
\[
\pair{\mathcal P(\gamma_1 \dots \gamma_i)(a), \gamma_{i+1}} = \pair{a, \gamma_{i+1}}.
\]
Thus applying \eqref{equation:pushformula} successively with $\mathcal P(\gamma_i)$ acting on $\mathcal P (\gamma_1 \dots \gamma_{i-1})(a)$ for $i = 1, \dots, n$ shows that
\begin{equation}\label{equation:gammaa}
\phi(\mathcal P(\gamma)(a)) = \phi(a) + \pair{a, \gamma} (2 - 2g).
\end{equation}
If $\gamma$ is not contained in $H_g$, then there exists some simple curve $a$ such that $\pair{a, \gamma} \ne 0$. Therefore, equation \eqref{equation:gammaa} shows that $\phi(\mathcal P(\gamma)(a)) \ne \phi(a)$ and hence $\mathcal P(\gamma) \not \in \Mod_{g,1}[\phi]$. 
\end{proof}

\begin{proof}[Proof of Proposition \ref{proposition:minstrat}]
Consider the forgetful map $p: \Mod_{g,1} \to \Mod_g$, with kernel $\ker(p) = \pi_1(UT\Sigma_g)$. By Lemma \ref{lemma:framedBES}, it suffices to see that (a) 
\[
p(\mathcal T_\phi) = \Mod_g[\widehat{\phi}],
\]
and (b)
\[
\mathcal T_\phi \cap \pi_1(UT\Sigma_g) = H_g.
\]

To see that $p(\mathcal T_\phi) = \Mod_g[\widehat \phi]$, we appeal to \cite[Theorem B]{CS_strata2}. By the framed change--of--coordinates principle (Proposition \ref{proposition:CCP}), there exists a configuration of admissible curves on $\Sigma_{g,1}$ (in the notation of \cite[Definition 3.11]{CS_strata2}) of type $\mathsf{C}(2g-2, \Arf(\phi))$. Either such configuration satisfies the hypotheses of \cite[Theorem B]{CS_strata2}, showing that $p(\mathcal T_\phi) = \Mod_g[\widehat \phi]$. 

The containment $H_g \le \mathcal T_\phi$ will follow from the work of Section \ref{section:septwist}. By Proposition \ref{proposition:minsep}, there is a containment $\mathcal K_{g,1} \le \mathcal T_\phi$. According to \cite[Theorem 4.1]{putman_johnson}, 
\[
\mathcal K_{g,1} \cap \pi_1(UT\Sigma_g) = H_g,
\]
showing the claim. 
\end{proof}

We observe that this proof also shows that \eqref{equation:gammaa} can be upgraded into the following:
\[
1 \to H_g \to \Mod_{g,1}[\phi] \to \Mod_{g}[\widehat \phi] \rightarrow 1.
\]

\section{The restricted arc graph and general framings}\label{section:arcgraph}
The goal of this section is to prove that for $g \ge 5$, given an arbitrary surface $\Sigma_{g,n}$ equipped with a relative framing $\phi$, there is an equality $\Mod_{g,n}[\phi] = \mathcal T_\phi$. We will argue by induction on $n$. The base case $n = 1$ was established in Proposition \ref{proposition:minstrat} above. To induct, we study the action of $\Mod_{g,n}[\phi]$ on a certain subgraph of the arc graph, and identify the stabilizer of a vertex with a certain $\Mod_{g,n-1}[\phi']$. In Section \ref{section:action}, we introduce the {\em $s$-restricted arc graph} $\mathcal A^s(\phi; p,q)$ and show that $\Mod_{g,n}[\phi]$ acts transitively on vertices and edges. In Section \ref{section:connected}, we prove that $\mathcal A^s(\phi; p,q)$ is connected, modulo a surgery argument. In Section \ref{section:asl}, we prove this ``admissible surgery lemma.'' Finally in Section \ref{section:induction} we use these results to prove that $\Mod_{g,n}[\phi] = \mathcal T_\phi$ for $g \ge 5$ and $n \ge 0$ arbitrary. 

\subsection{The restricted arc complex}\label{section:action}
We must first clarify some conventions and terminology about configurations of arcs. If $a \subset \Sigma_{g,n}$ is an arc and $c \subset \Sigma_{g,n}$ is a curve, then the geometric intersection number $i(a,c)$ is defined as for pairs of curves: $i(a,c)$ denotes the minimum number of intersections of transverse representatives of the isotopy classes of $a$ and $c$. If $a, b$ are both arcs (possibly based at one or more common point), we define $i(a,b)$ to be the minimum number of intersections of transverse representatives of the isotopy classes of $a, b$ {\em on the interiors of $a$, $b$}. In other words, intersections at common endpoints are not counted. We say that arcs $a,b$ are {\em disjoint} if $i(a,b) = 0$, so that ``disjoint'' properly means ``disjoint except at common endpoints.'' As usual, we say that an arc $a$ is {\em nonseparating} if the complement $\Sigma_{g,n}\setminus a$ is connected, and we say that a pair of arcs $a,b$ is {\em mutually nonseparating} if the complement $\Sigma_{g,n} \setminus\{a,b\}$ is connected (possibly after passing to well--chosen representatives of the isotopy classes in order to eliminate inessential components). 

Our objective is to identify a suitable subgraph of the arc graph on which $\Mod_{g,n}[\phi]$ acts transitively. Before presenting the full definition (see Definition \ref{definition:complex} below), we first provide a motivating discussion. By definition, an element of $\Mod_{g,n}[\phi]$ must preserve the winding number of every arc, and so to ensure transitivity we must restrict the vertices of this subcomplex to be arcs of a fixed winding number $s \in \Z+\tfrac{1}{2}$. However, this alone is insufficient, as Lemma \ref{lemma:sidednesswn} below makes precise.

If $\alpha$ and $\beta$ are two disjoint (legal) arcs which connect the same points $p$ and $q$ on boundary components $\Delta_p$ and $\Delta_q$, then the action of $\Mod_{g,n}[\phi]$ must preserve the winding number of each boundary component of a neighborhood of $\Delta_p \cup \alpha \cup \beta \cup \Delta_q$. The winding numbers of these curves depends on the $\phi$ values of $\alpha$, $\beta$, $\Delta_p$, and $\Delta_q$, but also on the configuration of $\alpha$ and $\beta$.

\begin{figure}[ht]
\centering
\labellist
\small
\endlabellist
\includegraphics[scale=1]{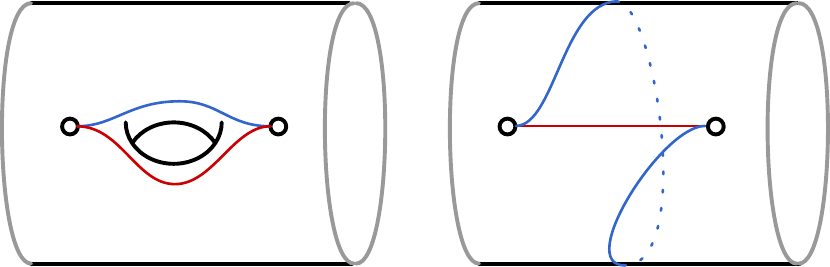}
\caption{Sidedness. Left: a one--sided pair. Right: a two--sided pair.}
\label{figure:sidedness}
\end{figure}

To that end, we say that a pair of arcs $\{\alpha, \beta\}$ as above is {\em one--sided} (respectively, {\em two--sided}) if $\alpha$ leaves $p$ and enters $q$ on the same side (respectively, opposite sides) of $\beta$ for every disjoint realization of $\alpha$ and $\beta$ on $\Sigma$. See Figure \ref{figure:sidedness}.

A quick computation yields an equivalent formulation in terms of winding numbers, which for clarity of exposition we state only in the case when $\phi(\alpha) = \phi(\beta)$. The proof follows by inspection of Figure \ref{figure:sidedness}.

\begin{lemma}\label{lemma:sidednesswn}
Let $\{\alpha, \beta\}$ be a pair of arcs as above with $\phi(\alpha) = \phi(\beta)$. Let $c^\pm$ denote the two curves forming the boundary of a neighborhood of $
\Delta_p \cup \alpha \cup \beta \cup \Delta_q$, oriented so that the subsurface containing $\Delta_p$ and $\Delta_q$ lies on their right. Then $\{\alpha, \beta\}$ is
\begin{itemize}
\item one--sided if and only if $\{\phi(c^+), \phi(c^-)\} = \{1, \phi(\Delta_p) + \phi(\Delta_q) +1\}.$
\item two--sided if and only if $\{\phi(c^+), \phi(c^-)\} = \{\phi(\Delta_p) +1, \phi(\Delta_q) +1\}$.
\end{itemize}
In particular, if $c$ is an admissible curve with $i(\alpha, c)=1$, then $\{\alpha, T_c(\alpha)\}$ is two--sided.
\end{lemma}

Having identified sidedness as a further obstruction to transitivity, we come to the definition of the complex under discussion. For any $s \in \Z+\frac{1}{2}$, we say that an arc $\alpha$ is an $s$--arc if $\phi(\alpha)=s$.

\begin{definition}\label{definition:complex}
Let $(\Sigma_{g,n}, \phi)$ be a framed surface with $n \ge 2$. Suppose that $p$ and $q$ are legal basepoints on distinct boundary components $\Delta_p$ and $\Delta_q$ and fix some $s \in \Z+\frac{1}{2}$. Then the {\em restricted $s$--arc graph} $\mathcal A_\pm^s(\phi; p,q)$ is defined as follows:
\begin{itemize}
\item A vertex of $\mathcal A_\pm^s(\phi; p,q)$ is an isotopy class $\alpha$ of $s$--arcs connecting $p$ and $q$.
\item Two vertices $\alpha$ and $\beta$ are connected by an edge if they are disjoint and mutually nonseparating.
\end{itemize}
The {\em two--sided restricted $s$--arc graph} $\mathcal A^s(\phi; p,q)$ is the subgraph of $\mathcal A_\pm^s(\phi; p,q)$ such that:
\begin{itemize}
\item The vertex set of $\mathcal A^s(\phi; p,q)$ is the same as that of $\mathcal A_\pm^s(\phi; p,q)$.
\item Two arcs $\alpha$ and $\beta$ are connected in $\mathcal A^s(\phi; p,q)$ if and only if they are connected in $\mathcal A^s_\pm(\phi; p,q)$ and the pair $\{\alpha, \beta\}$ is two--sided.
\end{itemize}
\end{definition}

\subsection{Transitivity}\label{subsection:transitive}
In this subsection we prove that the action of $\Mod_{g,n}[\phi]$ on $\mathcal A^s(\phi; p,q)$ is indeed transitive on both edges and vertices. The definition of $\mathcal A^s(\phi; p,q)$ above was rigged so that the proof of Lemma \ref{lemma:actionproperties} follows as an extended consequence of the framed change--of--coordinates principle (Proposition \ref{proposition:CCP}). The length of the proof is thus a consequence more of careful bookkeeping than genuine depth.

\begin{lemma}\label{lemma:actionproperties}
The action of $\Mod_{g,n}[\phi]$ on $\mathcal A^s(\phi; p,q)$ is transitive on vertices and on edges. 
\end{lemma}

We caution the reader that the action of $\Mod_{g,n}[\phi]$ is {\em not} transitive on {\em oriented} edges of $\mathcal A^s(\phi; p,q)$.

\begin{proof}
We begin by showing that edge transitivity implies vertex transitivity. Suppose that $\alpha$ and $\beta$ are two vertices of $\mathcal A^s(\phi; p,q)$; we will exhibit an element of $\Mod_{g,n}[\phi]$ taking $\alpha$ to $\beta$.

By the framed change--of--coordinates principle (Proposition \ref{proposition:CCP}), there is some admissible curve $c$ which meets $\alpha$ exactly once, and so by Lemma \ref{lemma:sidednesswn} the arc $T_c(\alpha)$ is adjacent to $\alpha$ in $\mathcal A^s(\phi; p,q)$. Now choose some $\gamma\in \mathcal A^s(\phi; p,q)$ adjacent to $\beta$. By edge transitivity, there exists a $g \in \Mod_{g,n}[\phi]$ which takes the $\{\alpha, T_c(\alpha)\}$ edge to the $\{\beta, \gamma\}$ edge. If $g(\alpha)=\beta$, then we are done. Otherwise, $g(T_c(\alpha))=\beta$, and since $c$ is admissible, $gT_c \in \Mod_{g,n}[\phi]$.

It remains to establish edge transitivity. Up to relabeling, we may assume that $p \in \Delta_1$ and $q \in \Delta_2$.
Suppose that ${\alpha}=\{\alpha_1, \alpha_2\}$ and ${\beta}=\{\beta_1, \beta_2\}$ are two edges of $\mathcal A^s(\phi; p,q)$. We also assume that $\alpha_1$ leaves $p$ from the right--hand side of $\alpha_2$ and enters $q$ to the left, and the same for $\beta_1$ and $\beta_2$.

For each $\bullet \in \{\alpha, \beta\}$, let $c_\bullet^\pm$ denote the two boundary components of a neighborhood of $\bullet \cup \Delta_1 \cup \Delta_2$. Let $X_\bullet$ (respectively $Y_\bullet$) denote the component of $\Sigma \setminus c_\bullet^\pm$ containing (respectively, not containing) $\bullet$, equipped with the induced framing $\xi_\bullet$ (respectively, $\eta_\bullet)$. Finally, orient each curve of $c_\bullet^\pm$ so that $Y_\bullet$ lies on its left--hand side. See Figure \ref{figure:arcpair}.

\begin{figure}[ht]
\centering
\labellist
\Small
\pinlabel $q$ [br] at 124.71 87.87
\pinlabel $p$ [tl] at 68.03 85.03
\pinlabel $\Delta_1$ [br] at 57.85 90
\pinlabel $\Delta_2$ [bl] at 136.05 90
\pinlabel $\alpha_1$ [t] at 90.70 85.03
\pinlabel $\alpha_2$ [tl] at 107.71 70.86
\pinlabel $X_\alpha$ [c] at 163.06 86.53
\pinlabel $Y_\alpha$ [c] at 17.01 39.68
\pinlabel $c_\alpha^+$ [b] at 24.09 70.86
\pinlabel $c_\alpha^-$ [br] at 148.80 65.19
\endlabellist
\includegraphics[scale=1]{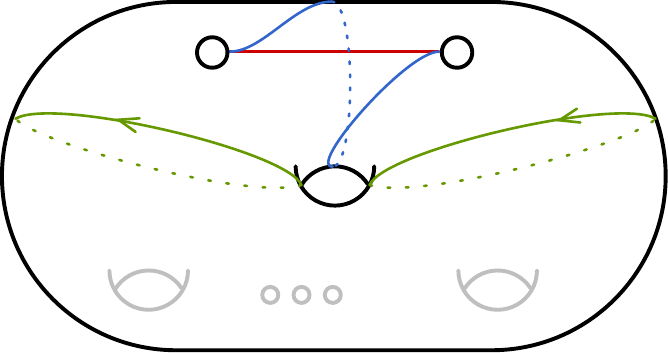}
\caption{The curves and subsurfaces determined by a two--sided pair of disjoint $s$--arcs.}
\label{figure:arcpair}
\end{figure}
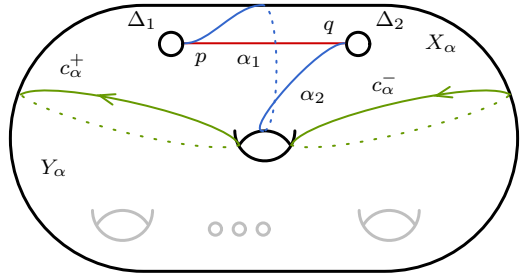

Since both $\alpha$ and $\beta$ are two--sided, Lemma \ref{lemma:sidednesswn} implies that
\begin{equation}\label{eqn:boundwn}
\{ \phi(c_{\bullet}^+),\phi(c_{\bullet}^-)\} = \{\phi(\Delta_1) +1, \phi(\Delta_2) +1\}
\end{equation}
for $\bullet \in \{\alpha, \beta\}$ and we fix the convention that $\phi(c_{\bullet}^+)=\phi(\Delta_1) +1$.

The proof now follows by building homeomorphisms $X_\alpha \rightarrow X_\beta$ and $Y_\alpha \rightarrow Y_\beta$ and gluing them together. 
To that end, we must first describe these subsurfaces in more detail.

\para{Distinguished arcs in $X_\bullet$}
Recall that $c_\bullet^\pm$ are defined as the boundary components of a neighborhood of $\bullet \cup \Delta_1 \cup \Delta_2$. 
If this neighborhood is taken to be very small (with respect to some auxiliary metric on $\Sigma$), then away from $p$ and $q$ the framing restricted to $c_\bullet^\pm$ looks like the framing on segments of $\bullet \cup \Delta_1 \cup \Delta_2$. 
In particular, for each point $p' \neq p$ of $\Delta_1$ with an orthogonally inward-- or outward--pointing framing vector there is a corresponding point of $c_\bullet^+$ with an orthogonally inward-- or outward--pointing framing vector. The analogous statement of course also holds for $q \neq q' \in \Delta_2$ and $c_\bullet^-$.

Pick points $p'$ and $p'' \neq p$ on $\Delta_1$ such that the framing vector at $p'$ points orthogonally outwards and the framing vector at $p''$ points orthogonally inwards. For the sake of concreteness, we will assume that $\phi(\Delta_1)$ is negative and take $p'$ (respectively $p''$) so that the arc $\gamma'$ (respectively $\gamma''$) of $\Delta_1$ which runs clockwise connecting $p$ to $p'$ ($p''$) has winding number equal to $-1/2$ (respectively $-1$). When $\phi(\Delta_1)$ is positive, the proof is identical except the arcs $\gamma'$ and $\gamma''$ will have winding numbers $1/2$ and $1$, respectively.

Now let $x^+_\bullet$ and $y^+_\bullet$ denote the corresponding points of $c_\bullet^+$; by construction, the framing vectors at these points point orthogonally into $X_\bullet$ and $Y_\bullet$, respectively. Using $q'$ and $q'' \neq q$ on $\Delta_2$, one may similarly construct $x^-_\bullet$ and $y^-_\bullet$. See Figure \ref{figure:specialpts}.

\begin{figure}[ht]
\centering
\labellist
\Small
\pinlabel $y_\alpha^+$ [tr] at 8.50 18.42
\pinlabel $p''$ [bl] at 55.27 56.69
\pinlabel $s_\alpha^+$ [t] at 68.03 28.34
\pinlabel $\Delta_1$ [t] at 76.53 104.87
\pinlabel $p'$ [br] at 107.20 40.94
\pinlabel $x_\alpha^+$ [tl] at 123.30 14.17
\pinlabel $p$ [br] at 106.29 75.11
\pinlabel $r_\alpha^+$ [tl] at 130.38 62.36
\pinlabel $\alpha_2$ [tl] at 192.74 113.38
\pinlabel $c_\alpha^+$ [t] at 195.57 53.85
\pinlabel $\alpha_1$ [t] at 223.92 70.86
\pinlabel $c_\alpha^-$ [tl] at 222.50 114.79
\endlabellist
\includegraphics[scale=1]{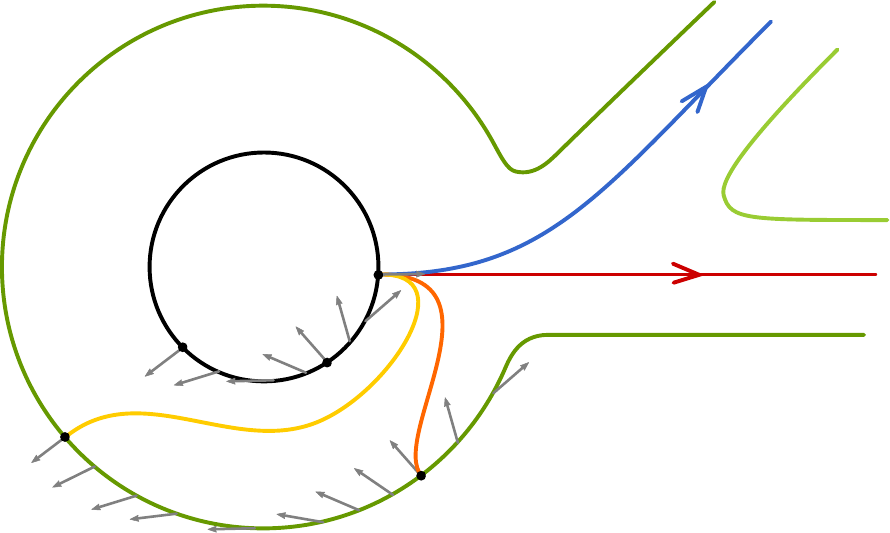}
\caption{Distinguished points and arcs in a neighborhood of $\alpha \cup \Delta_1$.}
\label{figure:specialpts}
\end{figure}
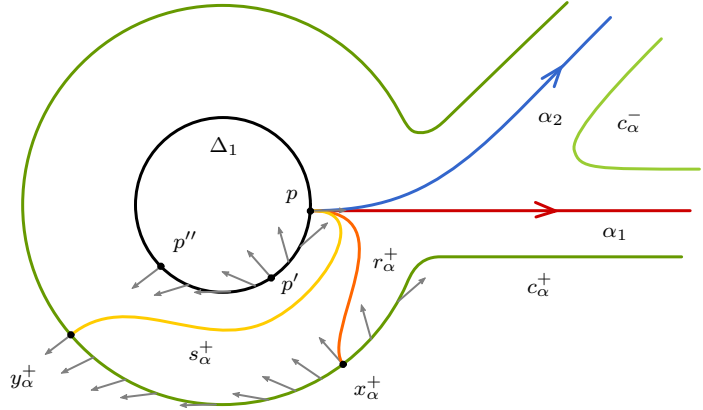

By our choice of $p'$ and $p''$, one may observe that there exist arcs $r_\bullet^\pm$ from $p$ to $x_\bullet^\pm$ with
\[\phi(r_\bullet^+) = -\frac{1}{2} \text{ and }
\phi(r_\bullet^-) = s.\]
Similarly, there exist $s_\bullet^\pm$ from $p$ to $y_\bullet^\pm$ with
\[\phi(s_\bullet^+) = -1 \text{ and }
\phi(s_\bullet^-) = s- \frac{1}{2}.\]
Now $\{\bullet_1, r_\bullet^+, r_\bullet^-\}$ forms a distinguished geometric basis for $X_\bullet$, and hence one can compute that
\begin{equation}\label{eqn:ArfX}
\Arf(\xi_\bullet) = (s+\frac{1}{2})(\phi(\Delta_2)+1) + (0)(\phi(\Delta_1)+2) + (s+\frac{1}{2})(\phi(\Delta_2)+2)
\end{equation}
which in particular does not depend on $\bullet \in \{\alpha, \beta\}$.

\para{Building homeomorphisms on subsurfaces}
By construction both $Y_\alpha$ and $Y_\beta$ are homeomorphic to $\Sigma_{g-1, n}$. By \eqref{eqn:boundwn} their boundary signatures agree:
\[\sig(\eta_\alpha) = (\phi(\Delta_1) +1, \phi(\Delta_2) +1, \phi(\Delta_3), \ldots, \phi(\Delta_n)) = \sig(\eta_\beta).\]
Moreover, by the additivity of the Arf invariant (Lemma \ref{lemma:Arfadd}) together with \eqref{eqn:boundwn} and \eqref{eqn:ArfX}, we have that
\[\Arf(\eta_\alpha)
= \Arf(\phi) + \Arf(\xi_\alpha) + p + q
= \Arf(\phi) + \Arf(\xi_\beta) + p + q
= \Arf(\eta_\beta)\]
and so by the classification of $\Mod_{g,n}$ orbits of framed surfaces (Proposition \ref{proposition:orbits}) there is a homeomorphism $f_Y: Y_\alpha \rightarrow Y_\beta$ such that $f_Y^*(\eta_\beta) = \eta_\alpha$. Moreover, $f_Y(c_\alpha^\pm) = c_\beta^\pm$ and in fact $f_Y(y_\alpha^\pm) = y_\beta^\pm$.

Now in order to extend $f_Y$ to a self--homeomorphism of $\Sigma$ which takes $\alpha$ to $\beta$, we need only specify a homeomorphism $f_X$ of $X_\alpha$ with $X_\beta$. This can be done easily by observing that $\alpha \cup r_\alpha^\pm$ cuts $X_\alpha$ into disks with the same combinatorics as $\beta \cup r_\beta^\pm$ cuts $X_\beta$, and hence there is a unique homeomorphism $f_X: X_\alpha \rightarrow X_\beta$ which takes $\alpha$ to $\beta$ and $r_\alpha^\pm$ to $r_\beta^\pm$.

Pasting $f_X$ and $f_Y$ together without twisting around $c_\alpha^\pm$, we therefore get a homeomorphism $\tilde{f}:\Sigma \rightarrow \Sigma$ which takes $\alpha$ to $\beta$.

\para{Preserving the framing}
It remains to show that $\tilde{f}$ preserves the framing $\phi$. Choose a distinguished geometric basis
\[\mathcal B_\beta = \{x_1,y_1, \ldots, x_{g-1}, y_{g-1}\} \cup \{a_2, \ldots, a_{n}\}\]
for $Y_\beta$ such that all the arcs $a_i$ of $\mathcal B_\beta$ emanate from $y_\beta^+ \in c_\beta^+$. By convention, suppose that $a_2$ runs from $y_\beta^+$ to $y_\beta^-$, and by twisting around $c_\beta^+$ if necessary, suppose that $a_2$ emerges to the left of all other $a_i$. Then $\mathcal B_\beta$ extends to a distinguished geometric basis of $\Sigma$ in the following way:
\[\widetilde{\mathcal{B}}_\beta =
\{x_1, y_1, \ldots, x_{g-1}, y_{g-1}, c_\beta^-, a_2 \cdot \overline{(s_\beta^-)} \cdot s_\beta^+\}
\cup
\{\beta_2 , s_\beta^+ \cdot a_3, \ldots, s_\beta^+ \cdot a_n\}\]
where $a \cdot b$ represents the concatenation of the arcs $a$ and $b$ and $\overline{a}$ represents the arc $a$ traveled backwards.
See Figure \ref{figure:basis_ext}.
By concatenating with $s_\alpha^\pm$ arcs, the basis $f_Y^{-1}(\mathcal{B})$ on $Y_\alpha$ also extends to a basis of $\Sigma$ in a similar fashion:
\begin{equation}
\begin{split}
\widetilde{\mathcal{B}}_\alpha = &
\{f_Y^{-1}(x_1), f_Y^{-1}(y_1), \ldots, f_Y^{-1}(x_{g-1}), f_Y^{-1}(y_{g-1}), c_\alpha^-, f_Y^{-1}(a_2) \cdot \overline{(s_\alpha^-)} \cdot s_\alpha^+\} \\
& \cup
\{\alpha_2 , s_\alpha^+ \cdot f_Y^{-1}(a_3), \ldots, s_\alpha^+ \cdot f_Y^{-1}(a_n)\}
\end{split}
\end{equation}

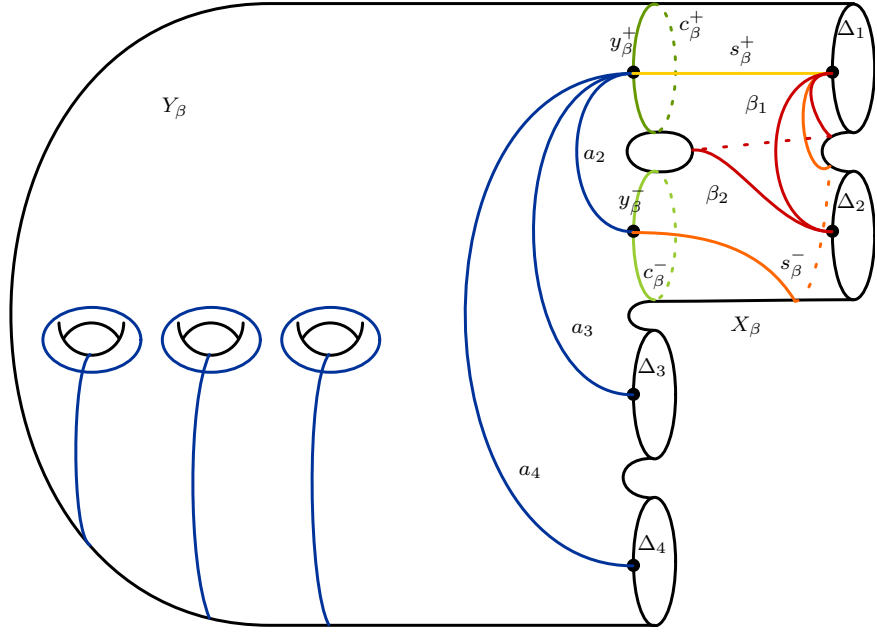
\begin{figure}[ht]
\centering
\labellist
\Small
\pinlabel $Y_\beta$ [c] at 62.36 195.57
\pinlabel $a_2$ [bl] at 214.25 173.23
\pinlabel $a_3$ [bl] at 209.74 107.71
\pinlabel $a_4$ [bl] at 189.90 53.85
\pinlabel $c_\beta^+$ [b] at 257.34 218.09
\pinlabel $y_\beta^+$ [br] at 239.00 212.58
\pinlabel $c_\beta^-$ [c] at 243.76 133.22
\pinlabel $y_\beta^-$ [br] at 242.42 153.06
\pinlabel $s_\beta^-$ [bl] at 288.36 128.96
\pinlabel $s_\beta^+$ [b] at 276.35 208.58
\pinlabel $\beta_1$ [br] at 288.27 191.32
\pinlabel $\beta_2$ [tr] at 273.52 168.65
\pinlabel $\Delta_1$ [c] at 317.45 225.33
\pinlabel $\Delta_2$ [c] at 317.45 160.14
\pinlabel $\Delta_3$ [c] at 242.34 97.79
\pinlabel $\Delta_4$ [c] at 242.34 31.18
\pinlabel $X_\beta$ [t] at 277.77 120.18
\endlabellist
\includegraphics[scale=1]{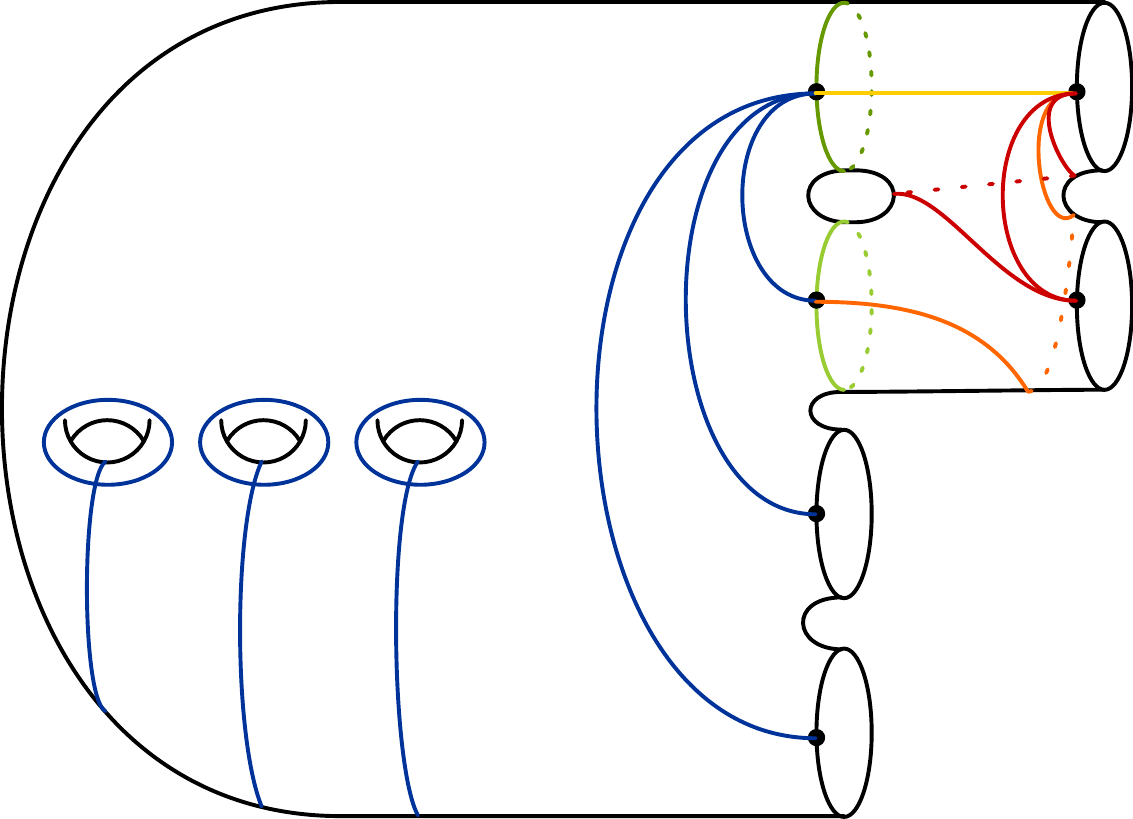}
\caption{Extending a distinguished geometric basis from $Y_\beta$ to $\Sigma$ by using the arcs $s_\beta^\pm$ of $X_\beta$.}
\label{figure:basis_ext}
\end{figure}

Now by construction we have that $f_X(s_\alpha^\pm) = s_\beta^\pm$ and
$\phi(s_\alpha^\pm) = s-1/2= \phi(s_\beta^\pm)$,
so we have that for each element $x \in \widetilde{\mathcal{B}}_\alpha$,
\[\tilde{f}(x) \in \widetilde{\mathcal{B}}_\beta \text{ and } \phi(\tilde{f}(x)) = \phi(x).\]
Therefore $\tilde{f}$ preserves the winding numbers of a distinguished geometric basis, and so by Remark \ref{remark:gsbCW}, $\tilde{f} \in \Mod_{g,n}[\phi]$.
\end{proof}

\subsection{Connectedness}\label{section:connected}
\begin{lemma}\label{lemma:connected}
Let $(\Sigma_{g,n}, \phi)$ be a framed surface with $g \ge 5$ and $n \ge 2$. Let $p,q$ be distinct boundary components, and let $s \in \Z$ be arbitrary. Then $\mathcal A^s(\phi; p,q)$ is connected. 
\end{lemma}

This will require the preliminary Lemmas \ref{lemma:genus2connected}, \ref{lemma:admissiblesurgery}, and \ref{lemma:onetwosided}. The first of these was proved in \cite{salter_toric}. There it was formulated only for closed surfaces, but the same proof applies for surfaces with an arbitrary number of punctures and boundary components.

\begin{lemma}[c.f. Lemma 7.3 of \cite{salter_toric}]\label{lemma:genus2connected}
Let $g \ge 5$ and $n \ge 0$ be given. Let $S$ and $S'$ be subsurfaces of $\Sigma_{g,n}$, each homeomorphic to $\Sigma_{2,1}$. Then there is a sequence $S = S_0, \dots, S_n = S'$ of subsurfaces of $\Sigma_{g,n}$ such that $S_{i-1}$ and $S_{i}$ are disjoint and $S_i\cong \Sigma_{2,1}$ for all $i = 1, \dots, n$.
\end{lemma}

\begin{lemma}[Admissible surgery]\label{lemma:admissiblesurgery}
Fix $g \ge 3, n \ge 2$, and let $(\Sigma_{g,n}, \phi)$ be a framed surface with distinguished legal basepoints $p$ and $q$ on boundary components $\Delta_p$ and $\Delta_q$. Let $S \subset \Sigma_{g,n}$ be a subsurface homeomorphic to $\Sigma_{2,1}$ (necessarily not containing $p$ or $q$). Let $\eta$ be an $s$--arc connecting $p, q$ that is disjoint from $S$. Let $x \subset \Sigma_{g,n}$ be either a separating curve or an arc connecting $p$ to $q$, in either case disjoint from $S$. Then there is a path $\eta = \eta_0, \dots, \eta_k$ in $\mathcal A^s_\pm(\phi; p,q)$ such that $i(\eta_k,x) = 0$. 
\end{lemma}

\begin{lemma}\label{lemma:onetwosided}
With hypotheses as above, if $\mathcal A_\pm^s(\phi; p, q)$ is connected, then also $\mathcal A^s(\phi; p,q)$ is connected. 
\end{lemma}
The proofs of Lemma \ref{lemma:admissiblesurgery} and Lemma \ref{lemma:onetwosided} are deferred to follow the proof of Lemma \ref{lemma:connected}. To prove Lemma \ref{lemma:connected}, we first introduce the notion of a ``curve-arc sum''.

\para{Curve--arc sums} We recall the notion of a ``curve--arc sum'' as discussed in \cite[Section 3.2]{salter_toric} and \cite[Definition 6.18]{CS_strata2}. Let $a$ be an oriented curve or arc, $b$ be an oriented curve, and $\epsilon$ be an embedded arc connecting the left sides of $a$ and $b$ and otherwise disjoint from $a \cup b$ (if $a$ is an arc we require $\epsilon \cap a$ to be a point on the interior of $a$). If $a$ is a curve (resp. arc), the {\em curve-arc sum} $a+_\epsilon b$ is the curve (resp. arc) obtained by dragging $a$ across $b$ along the path $\epsilon$; see Figure \ref{figure:cas}.
\begin{figure}[h]
\labellist
\Huge
\pinlabel $\rightsquigarrow$ at 180 45
\endlabellist
\includegraphics[scale=1]{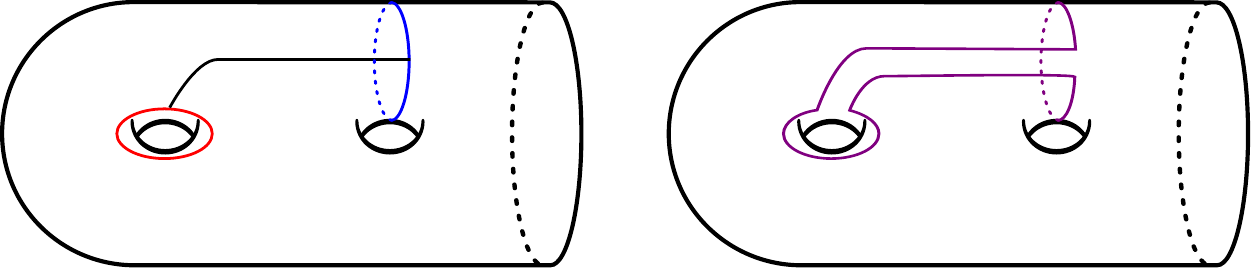}
\caption{The curve-arc sum operation.}
\label{figure:cas}
\end{figure}

\begin{lemma}[c.f. Lemma 3.13 of \cite{salter_toric}]\label{lemma:cas}
Let $a,b, \epsilon$ be as above and let $\phi$ be a relative winding number function. Then
\[
\phi(a+_\epsilon b) = \phi(a) + \phi(b) + 1.
\]
\end{lemma}

\begin{proof}[Proof of Lemma \ref{lemma:connected}] Following Lemma \ref{lemma:onetwosided}, it suffices to show that $\mathcal A_\pm^s(\phi; p,q)$ is connected. Let $\alpha$ and $\omega$ be $s$--arcs. Let $S_\alpha \cong \Sigma_{2,1}$ be disjoint from $\alpha$, and likewise choose $S_\omega \cong \Sigma_{2,1}$ disjoint from $\omega$. By Lemma \ref{lemma:genus2connected}, there is a sequence $S_\alpha = S_0, \dots, S_n = S_\omega$ of subsurfaces such that $S_i \cong \Sigma_{2,1}$ and such that $S_{i-1}$ and $S_{i}$ are disjoint for all $i= 1, \dots, n$. We apply the Admissible Surgery Lemma (Lemma \ref{lemma:admissiblesurgery}) taking $(S, \eta, x) = (S_\alpha, \alpha, \partial S_1)$. This gives a path $\alpha = \alpha_0, \dots, \alpha_m$ in $\mathcal A^s_\pm(\phi; p,q)$ such that $\alpha_m$ is disjoint from $S_1$. We now repeat this process for each $S_i\ (i \ge 1)$, finding intermediate paths of $s$--arcs, beginning with one disjoint from $S_i$ and ending with one disjoint from $S_{i+1}$. 

At the end of this process we have produced a path of $s$--arcs $\alpha, \dots, \psi$ with the final arc $\psi$ disjoint from $S_\omega$. To complete the argument we apply the Admissible Surgery Lemma one final time with $(S,\eta, x) = (S_\omega, \psi, \omega)$. This produces a path $\psi = \psi_0, \dots, \psi_k$ in $\mathcal A^s_\pm(\phi; p,q)$ with $i(\psi_k, \omega) = 0$. If $\psi_k \cup \omega$ is nonseparating, then $\psi_k$ and $\omega$ are adjacent in $\mathcal A^s_\pm(\phi; p,q)$, completing the path from $\alpha$ to $\omega$. If $\psi_k \cup \omega$ is separating, then at least one side of the complement has genus $h \ge 2$, and thus there exists a nonseparating oriented curve $d$ disjoint from $\psi_k \cup \omega$ that satisfies $\phi(d) = -1$. Define $\psi_{k+1} = \psi_k+_{\epsilon} d$ for a suitable arc $\epsilon$. Then $\psi_k, \psi_{k+1}, \omega$ is a path in $\mathcal A^s_\pm(\phi;p,q)$, completing the argument in this case.
\end{proof}

\subsection{Proof of the Admissible Surgery Lemma}\label{section:asl}
The proof will require the preliminary result of Lemma \ref{lemma:goodtwist} below.

\para{A change-of-coordinates lemma} We study the existence of suitable curves on genus $2$ subsurfaces. The proof of Lemma \ref{lemma:goodtwist} is a standard appeal to the framed change--of--coordinates principle (Proposition \ref{proposition:CCP}).
\begin{lemma}\label{lemma:goodtwist}
Let $(\Sigma_{2,1}, \phi)$ be a framed surface, and let $\alpha$ be a nonseparating properly-embedded arc on $\Sigma_{2,1}$. For $t \in \Z$ arbitrary, there is an oriented nonseparating curve $c_t \subset \Sigma_{2,1}$ such that $\phi(c_t) = t$ and such that $\pair{\alpha,c_t}= 1$.
\end{lemma}

\begin{proof}[Proof of Lemma \ref{lemma:admissiblesurgery}] The idea is to perform a sequence of surgeries on $\eta$ in order to successively reduce $i(\eta, x)$. Such surgeries will alter the winding number, but this will be repaired by using the ``unoccupied'' subsurface $S$ to fix the winding number while preserving the intersection pattern with $x$. The care we take below in selecting a suitable location for surgery ensures that the intersection pattern with $S$ remains unaltered. Throughout the proof we will refer to Figure \ref{figure:ASL}.

\begin{figure}[h]
\labellist
\small
\pinlabel (A) [tr] at 14.40 167.04
\pinlabel (B) [tr] at 256.32 167.00
\pinlabel (C) [tr] at 14.40 15.84
\pinlabel (D) [tr] at 288.00 15.84
\tiny
\pinlabel $x$ [bl] at 130.48 225.52
\pinlabel $\eta$ [b] at 158.40 213.00
\pinlabel $\eta'$ [bl] at 100.68 175.56
\pinlabel $d$ [bl] at 54.60 213.00
\pinlabel $\epsilon$ [t] at 74.88 201.60
\pinlabel $\eta_1$ [br] at 49.08 224.52
\pinlabel $d'$ [tl] at 54.72 169.92
\pinlabel $c$ [t] at 293.76 187.20
\pinlabel $\epsilon$ [tl] at 326.32 198.60
\pinlabel $\eta$ [t] at 375.84 210.12
\pinlabel $x$ [l] at 403.84 204.48
\pinlabel $y_1$ [bl] at 145.32 85.84
\pinlabel $y_2$ [tl] at 146.88 23.04
\pinlabel $y_3$ [tl] at 134 74.88
\pinlabel $y_4$ [br] at 156.44 61.36
\pinlabel $y_5$ [tl] at 135 49.96
\pinlabel $y_6$ [br] at 155.44 38
\pinlabel $d_+$ [br] at 28.80 60.48
\pinlabel $y_i$ [br] at 373.28 100.80
\pinlabel $y_{k+1}$ [tr] at 373.28 83.52
\pinlabel $y_k$ [tr] at 373.28 20.16
\pinlabel $y_\ell$ [br] at 373.28 37.44
\pinlabel $y_m$ [tr] at 373.28 8
\pinlabel $\eta_2$ [bl] at 311.04 34.56
\endlabellist
\includegraphics[scale=1]{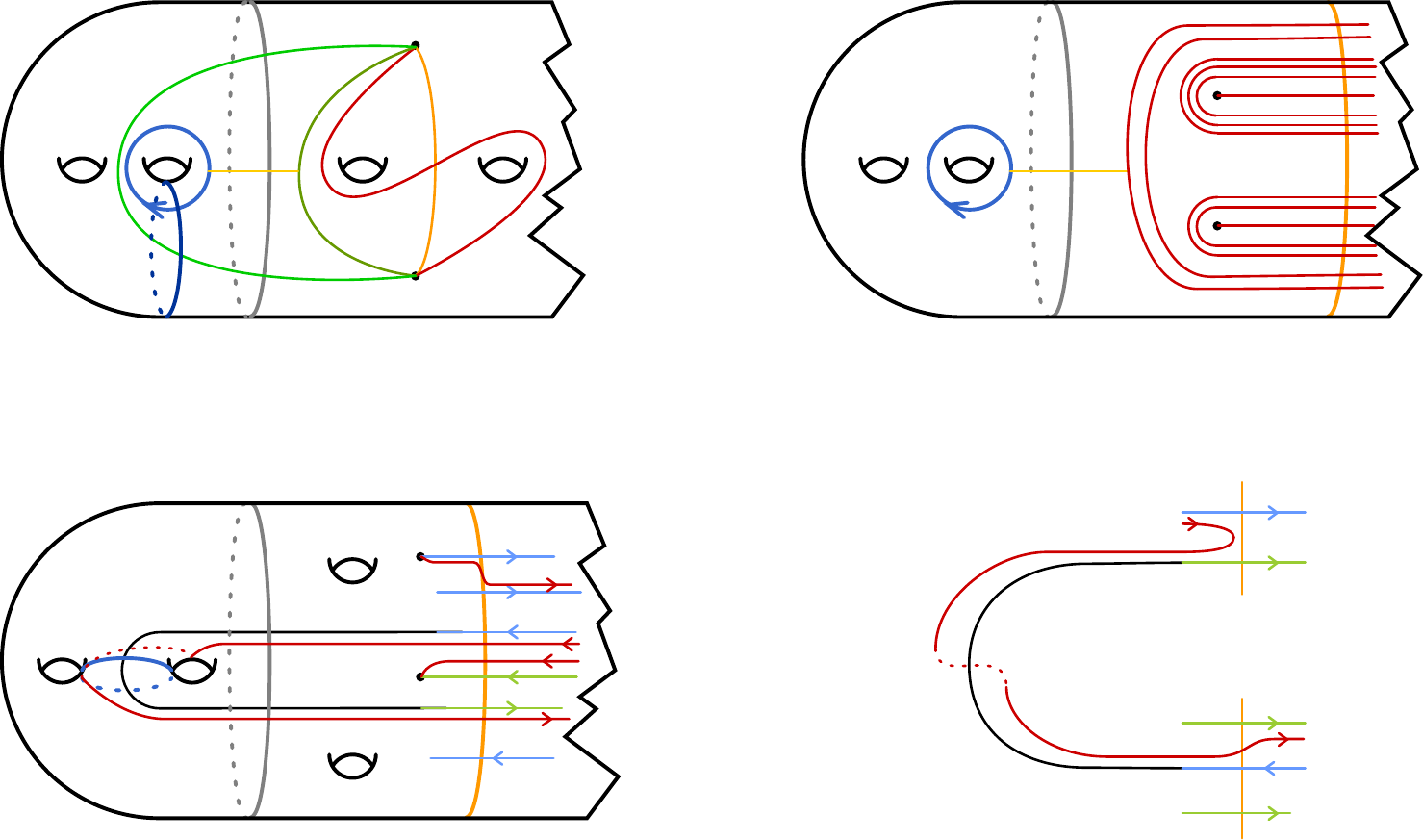}
\caption{(A): The case $i(\eta,x) = 1$, illustrated for $x$ an arc.  (B): The construction of $\eta_1$. (C): The surgery procedure on adjacent initial points ($\eta_1$ and $\eta_2$ are shown, but $\eta_2'$ is not). (D): The surgery procedure when the crossings alternate between initial and terminal. In (C,D), we have used blue to indicate initial points and green to indicate terminal points.}
\label{figure:ASL}
\end{figure}
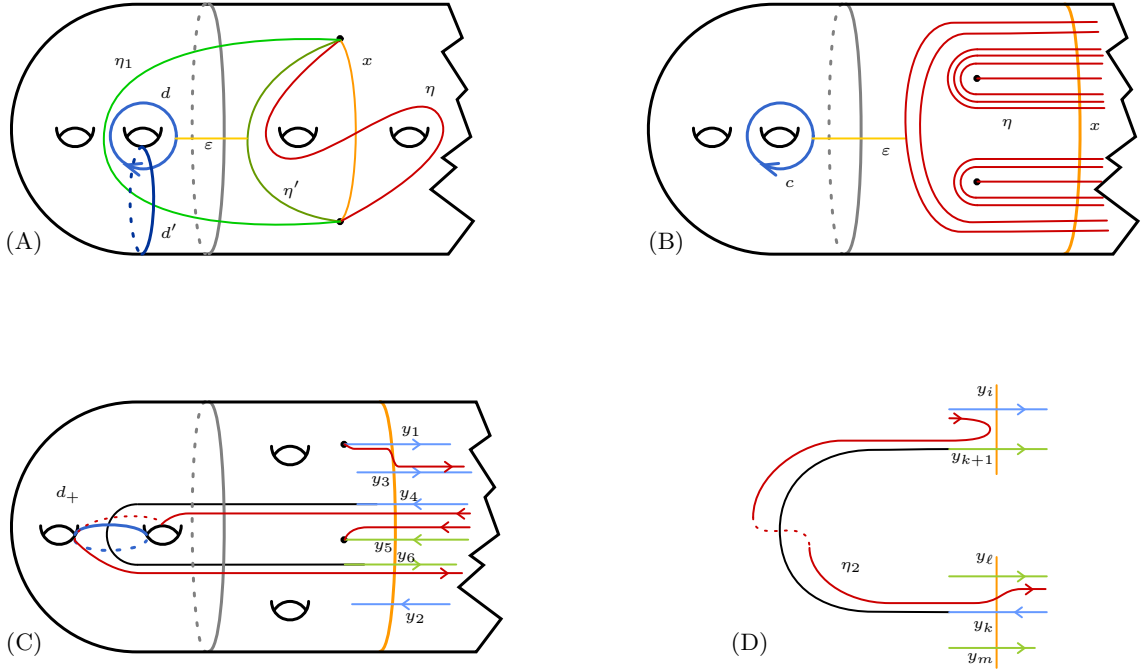

\para{Low intersection number} If $i(\eta,x) = 0$ there is nothing more to be done. If $i(\eta,x) = 1$, then there exists an arc $\eta'$ connecting $p$ to $q$ that is disjoint from $\eta \cup x \cup S$, and such that $\Sigma_{g,n} \setminus \{\eta \cup x \cup \eta'\}$ is connected. See Figure \ref{figure:ASL}(A). Let $d \subset S$ be an oriented nonseparating curve satisfying $\phi(d) = s-\phi(\eta')-1$. Let $\epsilon$ be an arc disjoint from $\eta \cup x$ and such that $i(\epsilon, \partial S) = 1$ that connects $\eta'$ to the left side of $d$, and let $\eta_1 = \eta' +_\epsilon d$. By Lemma \ref{lemma:cas}, $\phi(\eta_1) = \phi(\eta') + \phi(d) + 1 = s$. Since there exists a curve $d' \subset S$ such that $i(d',\eta_1) = 1$ and $i(d',\eta) = 0$, it follows that $\eta \cup \eta_1$ is nonseparating, completing the argument in the case $i(\eta,x) = 1$. 

\para{The general case: outline} We now consider the case $i(\eta, x) = N \ge 2$. We will first pass to an adjacent $s$--arc $\eta_1$ that enters and exits $S$ exactly once. We will use this in combination with a surgery argument to produce an $s$-arc $\eta_2$ that is adjacent to $\eta_1$, satisfies $i(\eta_2,x) < i(\eta,x)$, and also enters and exits $S$ once. As the above arguments (treating the cases $N \le 1$) can easily be adapted to the situation where $\eta$ passes once through $S$, this will complete the proof.

\para{First steps; initial and terminal points} Let $c \subset S$ be an oriented nonseparating curve satisfying $\phi(c) = -1$. Let $\epsilon$ be an arc disjoint from $x$ and such that $i(\epsilon, \partial S) = 1$ connecting $\eta$ to the left side of $c$, and let $\eta_1 = \eta +_\epsilon c$. See Figure \ref{figure:ASL}(B). By Lemma \ref{lemma:cas}, $\eta_1$ is an $s$--arc, and by construction, $\eta \cup \eta_1$ is isotopic to $c$ and is therefore nonseparating.

Enumerate the intersection points of $\eta_1 \cap x$ as $y_1, \dots, y_N$, numbered consecutively as $\eta_1$ runs from $p$ to $q$; further set $y_0 = p$ and $y_{N+1} = q$. For some $0\le k \le N$, the arc $\eta_1$ leaves $y_k$, enters $S$, and crosses back through $y_{k+1}$. The points $y_0, \dots, y_k$ are called {\em initial}, and the points $y_{k+1}, \dots, y_{N+1}$ are called {\em terminal}. We say that points $y_i$ and $y_j$ are {\em $x$--adjacent} if $y_i$ and $y_j$ appear consecutively when running along $x$ (in either direction). 

\para{Case 1: adjacent initial/terminal points} Suppose first that there is a pair of $x$--adjacent points $y_i, y_j$ that are either both initial or both terminal (if $x$ is a curve we consider $1 \le i < j \le N$, but if $x$ is an arc, the surgeries we describe below will work for all $0 \le i < j \le N+1$). In this case, let $\eta_2'$ be obtained from $\eta_1$ by following $\eta_1$ from $p$ to $y_i$, then along $x$ to $y_j$, then finally along $\eta_1$ from $y_j$ to $q$. See Figure \ref{figure:ASL}(C). Note first that $i(\eta_1, \eta_2') = 0$ and that $i(\eta_2',x) < i(\eta_1,x)$. It remains to alter $\eta_2'$ to an arc $\eta_2$ that is also disjoint from $\eta_1$ but which has $\phi(\eta_2) = s$ and such that $\eta_1 \cup \eta_2$ is nonseparating, i.e., such that $\eta_1, \eta_2$ is an edge in $\mathcal A^s_\pm(\phi; p,q)$. 

The method will be to find a curve on $S$ to twist along to correct the winding number of $\eta_2'$, but care must be taken to ensure that the twisted arc remains disjoint from $\eta_1$. Push $\eta_2'$ off of $\eta_1$ so that it runs parallel to $\eta_1$ except at the location of the surgery. As $\eta_2'$ and $\eta_1$ run along the segment between $y_k$ and $y_{k+1}$ through $S$, the pushoff of $\eta_2'$ lies to the left or to the right of $\eta_1$ in the direction of travel. We call the former case {\em positive position} and the latter {\em negative position}. If $c$ is a curve with $i(c, \eta_2') = 1$, observe that $T_c^\pm(\eta_2')$ is disjoint from $\eta_1$ so long as the sign of the twist coincides with the sign of the position.

Define $t: = \phi(\eta_2')$. By Lemma \ref{lemma:goodtwist}, there are nonseparating curves $d_\pm \subset S$ such that $\phi(d_\pm) = \pm (s-t)$ and such that $\pair{\eta_2', d_\pm} = 1$. Set $\eta_2 = T_{d_\pm}^{\pm}(\eta_2')$, where the sign depends on the sign of the position $\eta_2'$. Then $\eta_2$ is an $s$-curve adjacent to $\eta_1$ in $\mathcal A^s(\phi; p,q)$ and $i(\eta_2,x) < i(\eta_1,x)$.

\para{Case 2: alternating initial/terminal points} It remains to consider the case where every pair of $x$--adjacent points $y_i, y_j$ has one initial and one terminal element. Here there are two possibilities to consider: either there is exactly one terminal point (and hence $N = 2$), or else at least two. If there is exactly one terminal point and $x$ is an arc, then necessarily this terminal point is $q$. Then the unique initial point is $p$, and hence $\eta_1$ is disjoint from $x$ except at endpoints and there is nothing left to be done. If $x$ is a separating curve, then necessarily there are at least two terminal points, since if $\eta_1$ crosses into the subsurface bounded by $x$ at a terminal point, it must necessarily exit through another terminal point. 

We therefore assume that every pair of $x$--adjacent points $y_i, y_j$ have one initial and one terminal element and that there are at least two terminal points. The first terminal point $y_{k+1}$ is $x$--adjacent to two distinct initial points $y_i, y_j\ (i < j\le k)$, and likewise the last initial point $y_k$ is adjacent to two distinct terminal points $y_\ell, y_m\ (k+1 \le \ell < m)$. 

A suitable surgery is illustrated in Figure \ref{figure:ASL}(D). The surgered arc $\eta_2'$ begins by following $\eta_1$ forwards from $p$ to $y_i$, then along $x$ from $y_i$ to $y_{k+1}$, continuing {\em backwards} along $\eta_1$ from $y_{k+1}$ to $y_k$. At this point there is a choice: do we follow $x$ to $y_\ell$ or $y_m$ (in both cases continuing from here forwards along $\eta_1$ to $q$)? The orientation of $\eta_1$ endows each $y_i$ with a left and right side. If $y_i$ is adjacent to the left side of $y_{k+1}$, we continue $\eta_2'$ to whichever of $y_\ell, y_m$ lies to the right of $y_k$, and if $y_i$ lies to the right of $y_{k+1}$, we continue $\eta_2'$ along the point $y_\ell, y_m$ to the left. Observe that $i(\eta_2', x) < i(\eta_1,x)$, even in the exceptional case where the chosen terminal point $y_\ell$ happens to be $y_{k+1}$. 

The construction of $\eta_2'$ above facilitates the next step of the argument, which is to adjust $\eta_2'$ to an $s$--arc $\eta_2$ adjacent to $\eta_1$ in $\mathcal A^s_\pm(\phi; p,q)$. As in the prior case, let $\phi(\eta_2') = t$, and select (by Lemma \ref{lemma:goodtwist}) nonseparating curves $d_\pm \subset S$ such that $\phi(d_\pm) = \pm(s-t)$ and such that $\pair{\eta,d_\pm} = 1$. If $y_i$ is adjacent to the left side of $y_{k+1}$, define 
\[
\eta_2 = T_{d_-}^{-1}(\eta_2'),\]
and otherwise define
\[
\eta_2 = T_{d_+}(\eta_2').
\]
In both cases, $\eta_2$ is an $s$--arc adjacent to $\eta_1$ in $\mathcal A^s_\pm(\phi; p,q)$ and $i(\eta_2, x) < i(\eta_1,x)$. 
\end{proof}
 
Having established the admissible surgery lemma (Lemma \ref{lemma:admissiblesurgery}), it remains only to give the proof of Lemma \ref{lemma:onetwosided}, showing that connectivity of the restricted $s$--arc graph $\mathcal A^s_\pm(\phi; p, q)$ implies the connectivity of the two-sided restricted $s$--arc graph $\mathcal A^s(\phi; p,q)$.

\begin{figure}[ht]
\centering
\labellist
\tiny
\pinlabel $\alpha$ [bl] at 87.87 28.34
\pinlabel $\beta$ [tl] at 96.37 12.75
\pinlabel $\gamma$ [br] at 45.35 51.02
\pinlabel $c$ [tr] at 65.19 43.93
\pinlabel $\epsilon$ [r] at 79.36 35.43
\pinlabel $\Delta_1$ at 45.35 22.68
\pinlabel $\Delta_2$ at 115.21 22.68
\endlabellist
\includegraphics[scale=1]{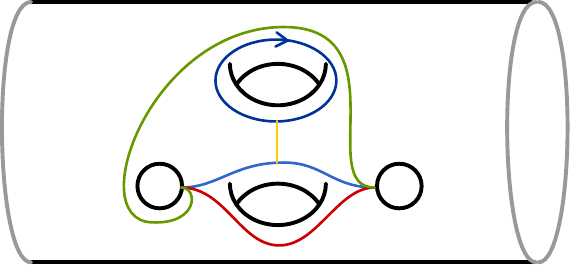}
\caption{Connecting a one--sided pair $\{\alpha, \beta\}$ with two two--sided pairs $\{\alpha, \gamma\}$ and $\{\gamma, \beta\}$.}
\label{figure:1s2s}
\end{figure}
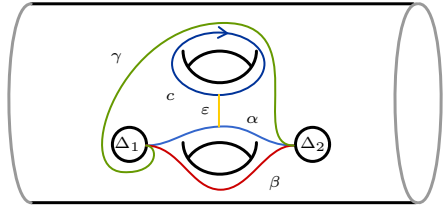

\begin{proof}[Proof of Lemma \ref{lemma:onetwosided}] We will refer to Figure \ref{figure:1s2s} throughout. It suffices to show that if $\{\alpha,\beta\}$ is a one-sided edge in $\mathcal A_\pm^s(\phi; p, q)$, there is a path in $\mathcal A^s(\phi;p,q)$ connecting $\alpha$ to $\beta$. Without loss of generality, suppose that $\alpha$ exits $p$ and enters $q$ on the left--hand side of $\beta$.

As $\Sigma \setminus ( \alpha \cup \beta)$ is a framed surface with genus $g-1 \ge 2$ and $n$ boundary components, we can apply the framed change--of--coordinates principle (Proposition \ref{proposition:CCP}) to deduce that there exists a nonseparating curve $c$ on $\Sigma$, disjoint from $\alpha \cup \beta$, such that 
\[\phi(c)=s-\phi(\alpha)-\phi(\Delta_p)-1.\]
As $\alpha \cup \beta$ and $c$ are both nonseparating, there exists an arc $\varepsilon$ from the left--hand side of $\alpha$ to the left--hand side of $c$. Therefore, by Lemmas \ref{lemma:cas} and \ref{lemma:HJ}.\ref{item:TL}, we see that for $\gamma := T_{\Delta_p}^{-1} ( \alpha +_{\varepsilon} c)$,
\[ \phi(\gamma) = \phi( \alpha +_{\varepsilon} c) + \phi( \Delta_p)
= \phi(c)+ \phi(\alpha)+1  + \phi(\Delta_p) =s.\]
Now since $\alpha +_{\varepsilon} c$ leaves $p$ and enters $q$ on the left--hand side of $\alpha$ (by construction), we see that $\gamma$ leaves $p$ to the right of both $\alpha$ and $\beta$, but enters $q$ on the left of both $\alpha$ and $\beta$. Therefore $\{\alpha, \gamma\}$ and $\{\gamma, \beta\}$ are edges in $\mathcal A^s(\phi; p,q)$.
\end{proof}

\subsection{The inductive step}\label{section:induction}
Having completed the proof of the Admissible Surgery Lemma, we can proceed with the proof of Theorem \ref{mainthm:genset}. Our objective in this subsection is Proposition \ref{prop:admisseverything}, which shows that $\Mod_{g,n}[\phi]$ coincides with the admissible subgroup $\mathcal T_\phi$. We proceed by way of the following standard technique of geometric group theory.
\begin{lemma}\label{lemma:actiongenset}
Let $G$ be a group acting on a connected graph $X$. Suppose that $G$ acts transitively on vertices and edges of $X$. For a vertex $v$, let $G_v$ denote the stabilizer of $v$. Let $e$ be an edge connecting vertices $v,w$, and let $h \in G$ satisfy $h(w) = v$. Then $G = \pair{G_v, h}$.
\end{lemma}
\begin{proof}
This is very similar to \cite[Lemma 4.10]{FarbMarg}. The argument there can easily be adapted to prove this slightly stronger statement. 
\end{proof}

\begin{proposition}\label{prop:admisseverything}
Let $g \ge 5$ and $n \ge 1$ be given, and consider a surface $\Sigma_{g,n}$ equipped with a framing $\phi$ of holomorphic or meromorphic type. Then there is an equality
\[
\mathcal T_\phi = \Mod_{g,n}[\phi].
\]
\end{proposition}
\begin{proof} We argue by induction on the number $n$ of boundary components. The base case $n = 1$ was established above as Proposition \ref{proposition:minstrat}. To proceed, we appeal to Lemma \ref{lemma:actiongenset}, taking $G = \Mod_{g,n}[\phi]$ and $X = \mathcal A^s(\phi; p,q)$ for $s \in \Z+\frac{1}{2}$ arbitrary. 

Lemmas \ref{lemma:actionproperties} and \ref{lemma:connected} combine to show that the hypotheses of Lemma \ref{lemma:actiongenset} are satisfied for $G = \Mod_{g,n}[\phi]$ and $X = \mathcal A^s(\phi; p,q)$. Let $\alpha$ be any $s$-arc connecting $p$ to $q$. By the framed change--of--coordinates principle (Proposition \ref{proposition:CCP}), there exists an admissible curve $a$ such that $i(\alpha, a) = 1$. The arc $T_a(\alpha)$ is disjoint from $\alpha$, the adjacency is two-sided, and the union $T_a(\alpha) \cup \alpha$ is nonseparating. As $a$ is admissible, it follows that $\phi(T_a(\alpha)) = \phi(\alpha) = s$, so that $T_a(\alpha)$ is adjacent to $\alpha$ in $\mathcal A^s(\phi; p,q)$. 

By Lemma \ref{lemma:actiongenset}, it now follows that $\Mod_{g,n}[\phi]$ is generated by $T_a$ and the stabilizer $(\Mod_{g,n}[\phi])_\alpha$. By hypothesis, $T_a \in \mathcal T_\phi$. To complete the inductive step, it remains to see that $(\Mod_{g,n}[\phi])_\alpha \le \mathcal T_\phi$. Let $\Delta$ be the boundary of a neighborhood of $\alpha \cup \Delta_p \cup \Delta_q$, and consider the subsurface $\Sigma_{g,n-1} \le \Sigma_{g,n}$ obtained by ignoring this neighborhood; it inherits a canonical framing $\phi'$ from the framing $\phi$ of $\Sigma_{g,n}$. The inclusion of framed surfaces $(\Sigma_{g,n-1}, \phi') \to (\Sigma_{g,n}, \phi)$ induces inclusions
\[
\Mod_{g,n-1}[\phi'] \into (\Mod_{g,n}[\phi])_\alpha, \qquad \mathcal T_{\phi'} \into \mathcal T_\phi. 
\]

By the inductive hypothesis, $\Mod_{g,n-1}[\phi'] \cong \mathcal T_{\phi'}$. On the other hand, it is easy to see that 
\[
(\Mod_{g,n})_\alpha \cap \Mod_{g,n}[\phi] \cong \Mod_{g,n-1}[\phi'], 
\]
and hence the inclusion $\Mod_{g,n-1}[\phi'] \into (\Mod_{g,n}[\phi])_\alpha$ is an isomorphism. The result follows.
\end{proof}

\subsection{Completing Theorem \ref{mainthm:genset}}\label{subsection:finishB} Theorem \ref{mainthm:genset} has two assertions. Part (I) asserts that if $\phi$ is a framing of holomorphic type, then $\Mod_{g,n}[\phi]$ is generated by the Dehn twists about an $E$-arboreal spanning configuration of admissible curves. This claim follows immediately from the work we have done: by Proposition \ref{proposition:fingen}, the admissible subgroup $\mathcal T_\phi$ is generated by this collection of twists, and the claim now follows from Proposition \ref{prop:admisseverything}. 

It therefore remains to establish claim (II) of Theorem \ref{mainthm:genset}. We recall the statement. Recall (c.f. the paragraph preceding the statement of Theorem \ref{mainthm:genset}) the notion of an {\em $h$-assemblage of type $E$}: begin with a collection $\mathcal C_1= \{c_1, \dots, c_k\}$ forming an $E$-arboreal spanning configuration on a subsurface $S \subset \Sigma_{g,n}$ of genus $h$, and then successively add in curves $c_{k+1}, \dots, c_{\ell}$ to $S$, at each stage attaching a new $1$-handle to the subsurface. Theorem \ref{mainthm:genset}.(II) asserts that if $\mathcal C$ is an $h$-assemblage of type $E$ for $h \ge 5$, and if each $c \in \mathcal C$ is admissible for some framing $\phi$ (of holomorphic or meromorphic type), then $\Mod(\Sigma_{g,n})[\phi]$ is generated by the finite collection of Dehn twists $\{T_c \mid c \in \mathcal C\}$. 

Before proceeding to the (short) proof below, we offer a comment on why we work in such generality. There are two reasons: one of necessity, the other of convenience. On the one hand, the homological coherence criterion (Lemma \ref{lemma:HJ}.\ref{item:HC}) implies that if $\mathcal C$ is an arboreal spanning configuration of {\em admissble} curves on the framed surface $(\Sigma, \phi)$, then necessarily $\phi$ is of holomorphic type (see Lemma \ref{lemma:arbholo} below). Thus for meromorphic type we must consider generating sets built on something more general than arboreal spanning configurations. Secondly, while the results of this paper (specifically Theorem \ref{mainthm:absolute_monodromy}) do not require assemblages, for other applications (especially \cite{PS_singularities}), this more general framework is essential. 

Above we asserted that if the framed surface $(\Sigma, \phi)$ admits an $E$-arboreal spanning configuration, then $\phi$ is necessarily of holomorphic type. In the course of proving Theorem \ref{mainthm:genset}.II, we will make use of this fact.
\begin{lemma}\label{lemma:arbholo}
Let $(\Sigma, \phi)$ be a framed surface. Suppose that $\mathcal C = \{c_1, \dots, c_k\}$ is an arboreal spanning configuration of admissible curves. Then $\phi$ is of holomorphic type.
\end{lemma}
\begin{proof}
The simplest proof uses the perspective of translation surfaces as discussed below in Section \ref{section:monodromy}. We employ the Thurston--Veech construction (see, e.g. \cite[Section 3.3]{CS_strata2}). Given a configuration of admissible curves whose intersection graph is a tree, this produces a translation surface on which each $c_i \in \mathcal C$ is represented as the core of a cylinder. The framing $\phi$ is incarnated as the framing associated to the horizontal vector field. The framing on a translation surface is necessarily of holomorphic type; the result follows.
\end{proof}

We have rigged the definition of an $h$-assemblage of type $E$ so as to make Theorem \ref{mainthm:genset}.II an immediate corollary of the following ``stabilization lemma''. Below, we use $\phi$ to denote both a framing of $\Sigma_{g,n}$ and the induced framing on subsurfaces (the latter was denoted $\phi'$ above).

\begin{lemma}\label{corollary:stabilization}
Let $(\Sigma_{g,n}, \phi)$ be a framed surface of holomorphic or meromorphic type. Let $S \subset \Sigma_{g,n}$ be a subsurface of genus at least $5$ and let $a \subset \Sigma_{g,n}$ be an admissible curve such that $a \cap S$ is a single arc; let $S^+$ denote a regular neighborhood of $S \cup a$. Then
\[
\Mod(S^+)[\phi] = \pair{\Mod(S)[\phi], T_a}.
\]
\end{lemma}
\begin{proof}
Let $\mathcal T_\phi(S^+)$ denote the admissible subgroup of $\Mod(S^+)[\phi]$. Following Proposition \ref{prop:admisseverything}, it suffices to show the containment
\[
\mathcal T_\phi(S^+) \le \pair{\Mod(S)[\phi], T_a}.
\]
To show this, we appeal to the methods of Section \ref{section:fingen}, specifically Proposition \ref{lemma:pushmakesT} and Lemma \ref{lemma:inductive}. Let $b \subset S$ be an arbitrary oriented nonseparating curve satisfying $\phi(b)= -1$, and consider the subsurface push subgroup 
\[
\tilde \Pi(b) =  \tilde \Pi (S^+ \setminus b).
\]
By the framed change-of-coordinates principle (Proposition \ref{proposition:CCP}), $b$ can be extended to a $3$-chain $(a_0,a_1,b)$ with $a_0, a_1$ admissible curves contained in $S$. Proposition \ref{lemma:pushmakesT} asserts that there is a containment 
\[
\mathcal T_\phi (S^+)\le \pair{T_{a_0}, T_{a_1}, \tilde \Pi(b)}.
\]
As $a_0, a_1$ are admissible curves on $S$, by hypothesis, $T_{a_0}, T_{a_1} \in \Mod(S)$, and so it remains to show that $\tilde \Pi(b) \le \pair{\Mod(S)[\phi], T_a}$.

To see this, we observe that $\tilde \Pi(S \setminus b) \le \Mod(S)[\phi]$. Again by the framed change-of-coordinates principle, there is an admissible curve $a' \subset (S \setminus b)$ such that $i(a', a) = 1$. Appealing to Lemma \ref{lemma:inductive}, it follows that $\tilde \Pi(b)$ is contained in the group $\pair{\tilde \Pi(S \setminus b), T_a, T_{a'}}$, and this latter group is contained in $\pair{\Mod(S)[\phi], T_a}$. \end{proof}

\begin{proof}[Proof of Theorem \ref{mainthm:genset}.II]
Let $\mathcal C = \{c_1,\dots, c_k, c_{k+1},\dots,c_\ell\}$ be an $h$-assemblage of type $E$ with $h \ge 5$. Assume further that each $c_i$ is admissible for the framing $\phi$. Recall that for $1 \le j \le \ell$, the regular neighborhood of the curves $c_1, \dots, c_j$ is denoted by $S_j$. By hypothesis, $\{c_1, \dots, c_k\}$ forms an $E$-arboreal spanning configuration on $S_k$, and the genus of $S_k$ is $h \ge 5$. In particular, the restriction of $\phi$ to $S_k$ is of holomorphic type (Lemma \ref{lemma:arbholo}). By Theorem \ref{mainthm:genset}.I, it follows that $\Mod(S_k)[\phi]$ is contained in the group $\mathcal T(\mathcal C) = \pair{T_{c_i} \mid c_i \in \mathcal C}$.

We now argue by induction. Supposing $\Mod(S_j)[\phi] \le \mathcal T(\mathcal C)$ for some $j \ge k$, it follows from Lemma \ref{corollary:stabilization} that since $S_{j+1}$ is the stabilization of $S_j$ along $c_{j+1}$, also $\Mod(S_{j+1})[\phi] \le \mathcal T(\mathcal C)$. As $S_\ell = \Sigma_{g,n}$ by assumption, the result follows. 
\end{proof}

\section{Other framed mapping class groups}\label{section:absolute}

In this section we leverage our work on the framed mapping class group $\Mod_{g,n}[\phi]$ in order to study two variants we will encounter in our investigation of the monodromy groups of strata of abelian differentials. The most straightforward variant we consider is the stabilizer of the framing up to {\em absolute} isotopy, i.e., where the isotopy is not necessarily trivial on the boundary. In this case, we will see that there is a sensible theory even when boundary components are replaced by marked points. We carry this out in Section \ref{subsection:punctured}, culminating in Proposition \ref{theorem:absolute}.

Our analysis of this case is built on a study of an intermediate refinement we call the ``pronged mapping class group.'' This group was introduced in a slightly different form in \cite{BSW_horo}, where it was called the ``mapping class group rel boundary'' of the surface. In Section \ref{section:pronged}, we lay out the basic theory of prong structures, pronged mapping class groups, and framings on pronged surfaces, leading to the structural result Proposition \ref{theorem:pronged}. The material in Section \ref{subsection:punctured} then follows as an easy corollary.

Section \ref{subsec:reltoabs} contains an analysis of the relationship between the (relative) framed mapping class group $\Mod_{g,n}[\phi]$ and its absolute counterpart $\PMod_g^n[\bar \phi]$. The main result here is Proposition \ref{proposition:relimg}, which identifies an obstruction for $\Mod_{g,n}[\phi]$ to surject onto $\PMod_g^n[\bar \phi]$.

\begin{remark}
For clarity of exposition, we restrict our attention throughout this section to framings $\phi$ of holomorphic type. Similar statements hold for arbitrary framings but the corresponding statements become somewhat messier; see Remarks \ref{remark:meroprong} and \ref{remark:mero_absgens}.
\end{remark}

\subsection{Pronged surfaces and pronged mapping class groups}\label{section:pronged} In our study of the monodromy of strata of abelian differentials, we will encounter a variant of a puncture/boundary component known as a {\em prong structure}. Here we outline the basic theory of surfaces with prong structure and their mapping class groups.

\begin{definition}[Prong structure, pronged mapping class group]
Let $\Sigma$ be a surface of genus $g$ equipped with a Riemannian metric and $p_i \in \Sigma$ a marked point. A {\em prong point of order $k_i$} at $p_i$ is a choice of $k_i$ distinct unit vectors ({\em prongs}) $v_1, \dots, v_{k_i} \in T_{p_i} \Sigma$ spaced at equal angles. With this data specified, we will write $\vec p_i$ to refer to the set of prongs based at $p_i$, and $\vec P$ to indicate a set of prong points $\{\vec p_1, \dots, \vec p_n\}$ with underlying points $P=\{p_1, \dots, p_n\}$.

Let $\vec P$ be a collection of prong points. Define $\Diff^+(\Sigma; \vec P)$ to be the group of orientation-preserving diffeomorphisms of $\Sigma$ that preserve each prong point (elements must fix the {\em points} $p_1, \dots, p_n$ pointwise, but can induce a (necessarily cyclic) permutation of the overlying tangent vectors). The {\em pronged mapping class group} is then defined as
\[
\Mod(\Sigma, \vec P) := \pi_0(\Diff^+(\Sigma; \vec P)).
\]
In the interest of compressing notation, this will often be written simply $\Mod_{g,n}^*$, with the underlying prong structure understood from context.
\end{definition}

\para{Prongs vs. boundary components} Here we outline the relationship between prongs and boundary components. First note that a prong point of order $1$ is simply a choice of fixed tangent vector. Also recall that in the case where all prong points have order $1$, there is a natural isomorphism
\[
\Mod(\Sigma, \vec P) \cong \Mod(\Sigma^*),
\]
where $\Sigma^* \cong \Sigma_{g,n}$ is the surface with $n$ boundary components obtained by performing a real oriented blowup at each $p_i$ (see Construction \ref{construction:blowup} below). More generally, let $(\Sigma, \vec P)$ be an arbitrary pronged surface with $\vec p_i$ a prong point of order $k_i$. Let $\mu_{k} \le \C^\times$ denote the group of $k^{th}$ roots of unity, and define the ``prong rotation group''
\[
PR := \prod_{i = 1}^n \mu_{k_i}.
\]
For each $1 \le i \le n$, there is a map $D_i: \Mod_{g,n}^* \to \mu_{k_i}$ given by taking the rotational part of the derivative at $T_{p_i} \Sigma \cong \C$. We define 
\[
D:= \prod_{i = 1}^n D_i: \Mod_{g,n}^* \to PR
\]
to be the product. Then $D$ induces the following short exact sequence (compare \cite[(6)]{BSW_horo}):
\begin{equation}\label{ses:PR1}
1 \to \Mod_{g,n} \to \Mod_{g,n}^* \to PR \to 1.
\end{equation}

\para{Fractional twists} There is an explicit set-theoretic splitting of \eqref{ses:PR1}. We define a {\em fractional twist} $T_{\vec {p_i}}$ at the prong point $\vec p_i$ of order $k_i$ to be the mapping class specified in a local complex coordinate $z \le 1$ near $p_i$ by 
\[
T_{\vec p_i}(z) = z e^{(2 \pi i (1-\abs{z}))/k_i}.
\]
Intuitively, $T_{\vec p_i}$ acts by applying a ``screwing motion'' of angle $2 \pi/k_i$ at $p_i$, viewing a small neighborhood of $p_i$ as being constructed from an elastic material connected to a rigid immobile boundary component. It is then clear that $PR$ embeds into $\Mod_{g,n}^*$ as the set of fractional twists $\{T_{\vec p_i}^{j_i} \mid 0 \le j_i < k_i,\ 1 \le i \le n\}$. Define
\[
FT = \pair{T_{\vec p_i},1 \le i \le n} \le \Mod_{g,n}^*
\]
to be the group generated by the fractional twists.

On the blowup $\Sigma^*$, the fractional twist $T_{\vec p_i}$ acts as a fractional rotation of the corresponding boundary component. It will be convenient to introduce the following notation: if $\Delta_i$ is the boundary component corresponding to $\vec p_i$, write $T_{\Delta_i}^{1/k_i}$ in place of $T_{\vec p_i}$. Note that $T_{\vec p_i}^{k_i}$ can be identified with the full Dehn twist about $\Delta_i$, so that the equation
\[
(T_{\Delta_i}^{1/k_i})^{k_i} = T_{\Delta_i}
\]
holds as it should. We will therefore allow $T_{\Delta_i}$ to assume fractional exponents with denominator $k_i$.

\para{Relative framings of pronged surfaces} Let $(\Sigma, \vec P)$ be a pronged surface. For simplicity's sake, we will formulate the discussion in this paragraph in terms of the blow-up $\Sigma^*$. Consider now a nonvanishing vector field $\xi$ on $\Sigma^*$. As in Section \ref{section:framings}, when a Riemannian metric is fixed, $\xi$ gives rise to a framing $\phi$ of $\Sigma^*$. In the presence of a prong structure, we will impose a further ``compatibility'' requirement on $\phi$ at the boundary.

\begin{definition}[Compatible framing]\label{definition:prong_compat}
Let $\vec p_i$ be a prong point of order $k_i$ on $\Sigma$ and let $\Delta_i$ be the associated boundary component of $\Sigma^*$. There is a canonical identification $\Delta_i \cong UT_{p_i} \Sigma$ between $\Delta_i$ and the space of unit tangent directions at $p_i$. We say that a framing $\phi$ is {\em compatible with $\vec p_i$} if the following conditions hold:
\begin{enumerate}
\item For $v \in UT_{p_i}\Sigma$, the framing vector $\phi(v)$ is orthogonally inward-pointing on $\Sigma^*$ if and only if $v$ is a prong.
\item The restriction of $\phi$ to $UT_{p_i}\Sigma$ is invariant under the action of $\mu_{k_i}$ on $UT_{p_i}\Sigma$.
\end{enumerate}
If $\vec P$ is a prong structure, we say that $\phi$ is compatible with $\vec P$ if $\phi$ is compatible with each $\vec p \in \vec P$ in the above sense.
\end{definition}

\begin{remark}\label{remark:WN}
Observe that if $\phi$ is compatible with a prong point $\vec p_i$ of order $k_i$, then the winding number of the associated boundary component $\Delta_i$ of $\Sigma^*$ is only determined up to sign: $\phi(\Delta_i) = \pm k_i$, depending on which way $\phi$ turns between prong points (recall the standing convention that boundary components are oriented with the interior of the surface lying to the left). Throughout, we will assume that $\phi$ is of holomorphic type so that $\phi(\Delta_i) < 0$ for all $i$.
\end{remark}

Observe that the notion of relative isotopy of framings still makes sense on a pronged surface. If $\phi$ is compatible with $\vec P$, then there is a well--defined action of $\Mod_{g,n}^*$ on the set of relative isotopy classes of framings. Exactly as in Section \ref{section:framings}, we define the {\em framed mapping class group} of the framed pronged surface $(\Sigma^*, \phi)$ to be the stabilizer of $\phi$:
\[
\Mod_{g,n}^*[\phi] = \{f \in \Mod_{g,n}^* \mid f \cdot \phi = \phi\}.
\]

{\noindent \bf Winding number functions.}
Exactly as in the setting of Section \ref{section:framings}, relative isotopy classes of framings on pronged surfaces are in bijection with suitably--defined winding number functions. The definition of winding number of a closed curve needs no modification. To set up a theory of winding numbers for arcs on pronged surfaces, we adopt the natural counterparts of the definitions of legal basepoint and legal arc from Section \ref{section:framings}. 

Suppose $(\Sigma, \vec P)$ is a pronged surface equipped with a compatible framing $\phi$. Let $\vec p_i$ be a prong point of order $k_i$. The prongs $v_1, \dots, v_{k_i} \in T_{p_i} \Sigma$ correspond to $k_i$ distinct points on the corresponding boundary component $\Delta_i$ of $\Sigma^*$, and by the compatibility assumption, each $v_i$ is a legal basepoint in the sense of Section \ref{section:framings}. We say that an arc $\alpha$ on the pronged surface $(\Sigma, \vec P)$ (or equivalently on the blowup $\Sigma^*$) is {\em legal} if $\alpha$ is properly embedded, each endpoint is some legal basepoint on $\Sigma^*$, and if $\alpha'(0)$ is orthogonally inward-pointing and $\alpha'(1)$ is orthogonally outward-pointing. In short, the theory of legal arcs on pronged surfaces differs from the theory on surfaces with boundary only in that we allow arcs to be based at {\em any} legal basepoint, not at one fixed point per boundary component. Fractional twists about the boundary may change the basepoint, and so we must consider all legal basepoints at once, instead of a single one.

Under this definition, legal arcs have half-integral winding number as before, and moreover $\Mod_{g,n}^*$ acts on the set of isotopy classes of relative arcs. The twist--linearity formula (Lemma \ref{lemma:HJ}.\ref{item:TL}) generalizes to fractional twists as follows (the proof is straightforward and is omitted).
\begin{lemma}\label{lemma:fractionalTL}
Let $a$ be a legal arc on a pronged surface $(\Sigma, \vec P)$ equipped with a compatible framing $\phi$. If $a$ has an endpoint at a legal basepoint on $\Delta_i$, then
\[
\phi(T_{\Delta_i}^{1/k_i}(a)) = \phi(a) \pm 1,
\]
with the sign positive if and only if $a$ is oriented as to be incoming at $\Delta_i$.
\end{lemma}

We also have the following straightforward extension of Lemma \ref{lemma:framingWNF}.

\begin{lemma}\label{lemma:prongframingWNF}
Let $\phi$ and $\psi$ be two framings of the pronged surface $(\Sigma, \vec P)$, both compatible with the prong structure. Then $\phi$ and $\psi$ are relatively isotopic if and only if the associated relative winding number functions are equal. Moreover, $\phi = \psi$ as relative winding number functions if and only if $\phi(b) = \psi(b)$ for all elements $b$ of a distinguished geometric basis $\mathcal B$.
\end{lemma}

We come now to the main result of the section. This describes the relationship between the stabilizer $\Mod_{g,n}^*[\phi]$ in the pronged mapping class group, and its subgroup $\Mod_{g,n}[\phi]$ where each prong is required to be individually fixed. In order to do so, we define a certain subgroup of $PR$. For convenience, we will switch to additive notation and identify $\mu_k \cong \Z/k\Z$, writing
\[
PR = \left\{ \sum c_i e_i \mid c_i \in \Z/k_i \Z \right\}.
\]
We furthermore write $\sum'$ to indicate a sum over all indices $i$ such that $k_i$ is even. Then define
\begin{equation}\label{eqn:PRprime}
PR' := \left\{\sum c_i e_i \in PR \mid \sum{}' c_i \equiv 0 \pmod 2 \right\}.
\end{equation}
Observe that if all $k_i$ are odd then $PR' = PR$.

\begin{proposition}\label{theorem:pronged}
Let $(\Sigma, \vec P)$ be a pronged surface and let $\phi$ be a compatible framing. Then the map $D: \Mod_{g,n}^* \to PR$ induces the short exact sequence
\begin{equation}\label{ses:pronged}
1 \to \Mod_{g,n}[\phi] \to \Mod_{g,n}^*[\phi] \to PR' \to 1.
\end{equation}
\end{proposition}

Before we begin the proof, we introduce the notion of an {\em auxiliary curve}. 

\begin{definition}[Auxiliary curves]
Let $(\Sigma_{g,n}, \phi)$ be a framed surface with boundary components $\Delta_1, \dots, \Delta_n$.
An {\em auxiliary curve for $\Delta_k$} is a separating curve $d_k$ such that $d_k$ separates $\Delta_k$ from the remaining boundary components, and so that $\phi(d_k) = \pm 1$ or $\pm 2$ according to whether $\phi(\Delta_k)$ is odd or even. 
An {\em auxiliary curve for $\Delta_i$ and $\Delta_j$} is any separating curve $c_{i,j}$ that separates boundary components $\Delta_i, \Delta_j$ from the remaining components.
\end{definition}

\begin{proof}[Proof of Proposition \ref{theorem:pronged}]
By definition,
\[
\Mod_{g,n}[\phi] := \Mod_{g,n} \cap \Mod_{g,n}^*[\phi],
\]
and $\Mod_{g,n}$ is the kernel of $D$. This establishes exactness at $\Mod_{g,n}[\phi]$ and at $\Mod_{g,n}^*[\phi]$.

It remains to be seen that $D(\Mod_{g,n}^*[\phi]) = PR'$. We first show that $PR' \le D(\Mod_{g,n}^*[\phi])$ by explicit construction. Observe that $PR'$ is generated by elements of three kinds:
\begin{enumerate}[(G1)]
\item $e_i$ for $i$ such that $k_i$ is odd,
\item $2 e_i$ for $i$ such that $k_i$ is even, 
\item $m(e_i + e_j)$ for $m$ odd and $i,j$ such that $k_i$ and $k_j$ are even. 
\end{enumerate}
Let $\Delta_i \in \partial \Sigma^*$ be given. Choose a curve $d_i$ in the following way: if $k_i$ is odd, pick $d_i \subset \Sigma^*$ such that $d_i$ separates $\Delta_i$ from all remaining boundary components and the subsurface bounded by $\Delta_i$ and $d_i$ has genus $(k_i+1)/2$. If $k_i$ is even, $d_i$ may be defined identically except that $\Delta_i \cup d_i$ must cobound a surface of genus $(k_i+2)/2$. By Remark \ref{remark:WN}, if $\phi$ is compatible with $\vec P$, then $\phi$ is of holomorphic type, and hence $k_i \le 2g - 1$ for all $i$. Thus the genus of the surface cobounded by $\Delta_i \cup d_i$ is at most $g$, and hence such $d_i$ exist for all boundary components $\Delta_i$. In both cases, orient $d_i$ so that $\Delta_i$ is on its left.

By Remark \ref{remark:WN} and the homological coherence property (Lemma \ref{lemma:HJ}.\ref{item:HC}), if $k_i$ is odd, then $\phi(d_i) = -1$ and if $k_i$ is even, then $\phi(d_i) = -2$. Therefore, $d_i$ is an auxiliary curve for $\Delta_i$.

If $k_i$ is odd, we define an {\em auxiliary twist of type 1} to be the mapping class
\[
A_i := T_{\Delta_i}^{1/k_i} T_{d_i}^{-1}
\]
(where $d_i$ is as above), and if $k_i$ is even, we define it to be
\[
A_i:= T_{\Delta_i}^{2/k_i} T_{d_i}^{-1}.
\]
Observe that $D(A_i) = e_i$ if $k_i$ is odd, and $D(A_i) = 2e_i$ if $k_i$ is even. By the twist-linearity formula (Lemma \ref{lemma:HJ}.\ref{item:TL}) and its extension to fractional twists Lemma (\ref{lemma:fractionalTL}), one verifies that $A_i \in \Mod_{g,n}^*[\phi]$.

Thus we have exhibited generators of the form $(G1), (G2)$ for $PR'$. It remains to construct elements mapping to generators of type $(G3)$. Let $k_i, k_j$ be even and choose an auxiliary curve $c_{i,j}$ for $\Delta_i$ and $\Delta_j$. By homological coherence (Lemma \ref{lemma:HJ}.\ref{item:HC}), $\phi(c_{i,j})$ is odd. Define the {\em auxiliary twist of type 2} to be the mapping class
\[
B_{i,j} := T_{\Delta_i}^{\phi(c_{i,j})/k_i}T_{\Delta_j}^{\phi(c_{i,j})/k_j} T_{c_{i,j}}^{-1}.
\]
Then $D(B_{i,j}) = \phi(c_{i,j})(e_i + e_j)$ represents a generator of type (G3), and as before, the twist--linearity formula shows that $B_{i,j} \in \Mod_{g,n}^*[\phi]$. \smallskip

We now establish the converse assertion $D(\Mod_{g,n}^*[\phi]) \le PR'$. For this, we recall that there is a set--theoretic splitting $s: PR \to \Mod_{g,n}^*$ given by fractional twists, and define the set--theoretic retraction $r: \Mod_{g,n}^* \to \Mod_{g,n}$ by $r(f) = sD(f^{-1}) f$. 

Let $f \in \Mod_{g,n}^*[\phi]$ be given. Then $r(f) \in \Mod_{g,n}$ by construction. The Arf invariant classifies orbits of relative framings under the action of $\Mod_{g,n}$, and hence we must have
\[
\Arf(r(f) \cdot \phi) = \Arf(\phi).
\]
Let $\mathcal B = \{x_1, \dots, y_g\} \cup \{a_2, \dots, a_n\}$ be a distinguished geometric basis. By hypothesis, $\phi(f \cdot b) = \phi(b)$ for all $b \in \mathcal B$. Also note that 
\[
\phi(T_{\Delta_i}^{1/k_i}(a_i)) = \phi(a_i) + 1
\]
while fixing the $\phi$ values of all other elements of $\mathcal B$. Likewise,
\[
\phi(T_{\Delta_1}^{1/k_1}(a_i)) = \phi(a_i) - 1
\]
for $i = 2, \dots, n$ while $T_{\Delta_1}^{1/k_1}$ fixes the $\phi$ value of each of the curves $x_1, \dots, y_g$. Considering the Arf invariant formula \eqref{equation:arf}, it follows that
\[
\Arf(r(f) \cdot \phi) - \Arf(\phi) = \sum{}'D_i(f).
\]
Thus $\Arf(r(f)\cdot \phi) = \Arf(\phi)$ exactly when $\sum' D_i(f) = 0$, i.e., when $D(f) \in PR'$. 
\end{proof}

\begin{remark}\label{remark:meroprong}
As noted above, this theory generalizes to arbitrary framings compatible with a prong structure. When boundary components have positive winding number the signs of the formula in \ref{lemma:fractionalTL} reverse. More substantially, the proof of Proposition \ref{theorem:pronged} must be altered, for there do not exist auxiliary curves for a boundary component $\Delta$ of arbitrary winding number. Instead, one must take a combination of twists on curves separating $\Delta$ from the other boundary components (together with a fractional twist about $\Delta$) to produce the generators (G1) and (G2).

If one wishes to include boundary components of winding number 0 in this theory, the easiest method is to introduce them separately and consider framings on surfaces with both boundary (of winding number 0) and prongs. In this case, the corresponding factor in the prong rotation group is trivial (see below), and boundary twists about winding number 0 curves are all in the stabilizer of the framing.
\end{remark}

\subsection{Pointed surfaces and absolute framed mapping class groups}\label{subsection:punctured}
The second variant of a framed mapping class group we consider is the most coarse. As in the pronged setting, we consider a closed genus $g$ surface $\Sigma$ with a collection $P = \{p_1, \dots, p_n\}$ of marked points, but we do not equip each $p_i$ with the structure of a prong point. Thus the mapping class group acting up to isotopy on $(\Sigma, P)$ is the familiar punctured mapping class group $\Mod_g^n$. 

Forgetting the prong structure induces the following short exact sequence of mapping class groups (c.f. \cite[Lemma 2.4]{BSW_horo}):
\begin{equation}\label{eqn:BSW_tw}
1\rightarrow FT  \rightarrow \Mod_{g,n}^*
\rightarrow \PMod_g^n \rightarrow 1,
\end{equation}
where we recall that $FT$ is the group generated by the fractional twists at each $p_i$.

Suppose that $\xi$ is a vector field on $\Sigma$ vanishing only at $P$. Then $\xi$ determines a framing of the punctured surface $\Sigma \setminus P$. Since we do not fix any boundary data, the notion of relative isotopy is ill--defined. To emphasize this, we use the term {\em absolute} in this setting, and so we speak of {\em absolute isotopy classes} of framings, {\em absolute winding number functions} (which measure winding numbers only of oriented simple closed curves, not arcs), and the {\em absolute framed mapping class group}.
In this language, the pointed mapping class group $\Mod_g^n$ acts on the set of isotopy classes of absolute framings, or equivalently on the set of absolute winding number functions. If $\overline \phi$ is an absolute framing/winding number function, we write $\Mod_g^n[\overline \phi]$ to denote the stabilizer. The main result concerning $\Mod_g^n[\overline \phi]$ that we will need for later use is the following.

\begin{theorem}\label{theorem:absolute}
Let $(\Sigma, \vec P)$ be a pronged surface equipped with a compatible framing $\phi$. The forgetful map $p: (\Sigma, \vec P) \to (\Sigma, P)$ induces a surjection
\[
p_*: \Mod_{g,n}^*[\phi] \to \PMod_g^n[\overline \phi].
\]
\end{theorem}
\begin{proof}
Let $\overline f \in \Mod_g^n[\overline \phi]$ be given, and choose a lift $f \in \Mod_{g,n}^*$. The set of lifts is a torsor on the kernel $FT$ of the forgetful map $p_*: \Mod_{g,n}^* \to \Mod_g^n$. The group of fractional twists $FT$ preserves all absolute winding numbers. If $\mathcal B = \{x_1, \dots, y_g\} \cup \{a_2, \dots, a_n\}$ is a distinguished geometric basis, then by Lemma \ref{lemma:fractionalTL}, $FT$ acts transitively on the set of values $(\phi(a_2), \dots, \phi(a_n))$. Thus there is $g \in FT$ such that $\phi(gf(b)) = \phi(b)$ for all $b \in \mathcal B$. By Lemma \ref{lemma:prongframingWNF}, such $gf$ is an element of $\Mod_{g,n}^*[\phi]$, and by construction $p_*(gf) = \overline f$. 
\end{proof}

Combining this result with Proposition \ref{theorem:pronged} yields an explicit generating set for $\PMod_g^n[\overline \phi]$.

\begin{definition}
Let $\Sigma_{g,n}$ have boundary components $\Delta_1, \ldots, \Delta_n$. An {\em auxiliary curve system} is a collection of the following auxiliary curves:
\begin{itemize}
\item Auxiliary curves $c_{i,j}$ for all pairs $i,j$ such that both $\phi(\Delta_i), \phi(\Delta_j)$ are even,
\item Auxiliary curves $d_k$ for all indices $k$
\end{itemize}
No requirements are imposed on the intersection pattern of curves in an auxiliary curve system.
\end{definition}

\begin{corollary}\label{cor:genset_abs}
Let $\sing$ be a partition of $2g-2$ by positive integers. Let $\phi$ be a relative framing with signature $-1-\sing$ and let $\bar \phi$ denote the absolute framing induced on $\Sigma_g^n$ by capping off the boundary components of $\Sigma_{g,n}$ with punctured disks. Then $\PMod_g^n[\bar \phi]$ is generated by $p_*( \Mod_{g,n}[\phi])$ together with the twists about an auxiliary curve system $\mathcal A$.

In particular, when $g \ge 5$, then $\PMod_g^n[\bar \phi]$ is generated by the Dehn twists in the curves of $\mathcal C \cup \mathcal A$, where $\mathcal C$ is as in Figure \ref{figure:genset}.
\end{corollary}
\begin{proof}
As in the proof of Proposition \ref{theorem:pronged}, the group of auxiliary twists $\pair{ A_k, B_{i,j} }$ surjects onto $PR'$ and hence by Proposition \ref{theorem:pronged}
\[ \left\langle \Mod_{g,n}[\phi], A_k, B_{i,j} \right\rangle = \Mod_{g,n}^*[\phi].\]
Now it remains to observe that by Theorem \ref{theorem:absolute} this group surjects onto $\PMod_g^n[\bar \phi]$ and that by construction $p_*(A_k) = T_{d_k}^{-1}$ and $p_*(B_{i,j}) = c_{i,j}^{-1}$.
\end{proof}

\begin{remark}\label{remark:mero_absgens}
As observed in Remark \ref{remark:meroprong}, auxiliary curve systems do not always exist for arbitrary framings. The substitutions outlined there can be similarly be used to give a generating set for arbitrary $\PMod_g^n[\bar \phi]$ in terms of $p_*( \Mod_{g,n}[\phi])$ and combinations of separating twists.
\end{remark}

\subsection{The image of the relatively framed mapping class group}\label{subsec:reltoabs} The final result we consider here determines the image of $\Mod_{g,n}[\phi]$ in $\PMod_g^n[\bar \phi]$ induced by the boundary-capping map $\Sigma_{g,n} \to \Sigma_{g}^n$. Again, we restrict to the case when $\phi$ is of holomorphic type; the corresponding statements and proofs for framings of arbitrary type are left to the interested reader (one needs only change the signs of some generators).

\begin{proposition}\label{proposition:relimg}
The image of $\Mod_{g,n}[\phi]$ in $\PMod_g^n[\bar \phi]$ is a normal subgroup with quotient isomorphic to $PR' / \pair{(1,\dots, 1)}$.
\end{proposition}
\begin{proof}
By Proposition \ref{theorem:pronged}, $\Mod_{g,n}[\phi] \normal \Mod_{g,n}^*[\phi]$ is a normal subgroup with quotient $PR'$ induced by the ``prong rotation map'' 
\[
D: \Mod_{g,n}^* \to PR.
\] 
By Theorem \ref{theorem:absolute}, the boundary-capping map $p_*: \Mod_{g,n}^* \to \Mod_g^n$ restricts to a surjection $p_*: \Mod_{g,n}^*[\phi] \to \Mod_g^n[\bar \phi]$. Thus the image of $\Mod_{g,n}[\phi]$ is a normal subgroup of $\Mod_g^n[\bar \phi]$. 

The quotient $\Mod_g^n[\bar \phi] / p_*(\Mod_{g,n}[\phi])$ can be identified by the Isomorphism Theorems. Suppose that $G$ is a group and $N_1, N_2$ are normal subgroups. Then $N_1 N_2$ is normal in $G$, and by the third isomorphism theorem, 
\[
(G/N_1)/(N_1 N_2/N_1) \cong G/(N_1 N_2) \cong (G/N_2)/ (N_1 N_2/ N_2).
\]
We apply this here with 
\[
G = \Mod_{g,n}^*[\phi], \qquad N_1 = \ker p_*, \qquad N_2 = \Mod_{g,n}[\phi].
\]
Then
\[
G/N_1 \cong \Mod_g^n[\bar \phi], \qquad G/N_2 \cong PR',
\] \[
N_1N_2/N_1 \cong p_*(N_2) \cong p_*(\Mod_{g,n}[\phi]), \qquad N_1 N_2/ N_2 \cong D(N_1) \cong D(\ker p_*).
\]
Altogether,
\[
\Mod_g^n[\bar \phi] / p_*(\Mod_{g,n}[\phi])\cong PR' / D(\ker p_*).
\]

To complete the argument, it therefore suffices to show that $D(\ker p_*) \cong \pair{(1,\dots, 1)}$. According to \eqref{eqn:BSW_tw}, the kernel of $p_*$ on $\Mod_{g,n}^*$ is the group $FT$ of fractional twists. Thus we must identify $FT \cap \Mod_{g,n}^*[\phi]$. We claim that 
\[
FT \cap \Mod_{g,n}^*[\phi] \cong \Z
\]
generated by the fractional twist $\prod_{i = 1}^n T_{\Delta_i}^{1/k_i}$. Note that 
\[
D(\prod_{i = 1}^n T_{\Delta_i}^{1/k_i}) = (1,\dots, 1),
\]
so that showing this isomorphism will complete the argument. 

To show this claim, consider an arbitrary element 
\[
f : = \prod_{i = 1}^n T_{\Delta_i}^{a_i/k_i} 
\] 
of $FT \cap \Mod_{g,n}^*[\phi]$. Let $\alpha_{i,j}$ be a legal arc connecting $\Delta_i$ to $\Delta_j$. By Lemma \ref{lemma:fractionalTL}, 
\[
\phi(f(\alpha_{i,j})) - \phi(\alpha_{i,j}) = a_j - a_i,
\]
so that if $f \in \Mod_{g,n}^*[\phi]$, necessarily $a_i = a_j$ for all pairs $i,j$ as claimed.
\end{proof}

We now unravel the condition that $PR' / \pair{(1,\dots, 1)}$ is trivial using some elementary group theory.

\begin{corollary}\label{corollary:whentrivial}
Let $\sing$ be a partition of $2g-2$ and write $\sing = (\eta_1, \ldots, \eta_p, \upsilon_1, \ldots, \upsilon_q)$ where $\eta_i$ are even, $\upsilon_j$ are odd, and $p+q=n$. Then $\Mod_{g,n}[\phi]$ surjects onto $\PMod_g^n[\bar \phi]$ if and only if $q \le 2$ and
\[ \left\{\eta_1 + 1, \ldots, \eta_p + 1, \frac{\upsilon_1 + 1}{2}, \ldots, \frac{\upsilon_q + 1}{2} \right\}\]
are pairwise coprime.
\end{corollary}
\begin{proof}
By Proposition \ref{proposition:relimg}, it suffices to determine when $PR'$ is cyclic (with generator $(1, \dots, 1)$).

Using the additive notation introduced above, write
\[ PR_e := \left\{ \sum_{j=1}^q c_j e_j \mid c_j \in \Z/(\upsilon_j+1) \Z \right\}
\text{ and } 
PR_{o} := \left\{ \sum_{i=1}^p c_i e_i \mid c_i \in \Z/(\eta_i+1) \Z \right\}.\]
Even though $PR_e$ is a product over the odd $\upsilon_j$'s, our notation reflects the fact that each of its factors has even order.

Now by definition $PR = PR_e \times PR_o$, and likewise we can write $PR' = PR_e' \times PR_o$ where
\[PR_e' := \left\{\sum_j c_j e_j \in PR_e \mid \sum c_j \equiv 0 \pmod 2 \right\}\]
as in \eqref{eqn:PRprime}. We observe that $PR'$ is cyclic if and only if $PR_o$ and $PR_e'$ are cyclic of coprime order, and that $PR_o$ is cyclic if and only if the set of $\eta_i+1$ are all pairwise coprime.

Suppose that $PR_e'$ is cyclic. If $PR_e$ (and hence also $PR_e'$) is trivial, then the claim holds. Otherwise, there is a short exact sequence
\[ 1 \rightarrow PR_e' \rightarrow PR_e \rightarrow \Z/2 \rightarrow 1.\]
It follows that $PR_e$ is either cyclic of order $2 |PR_e'|$ or else is isomorphic to $\Z/2 \times PR_e'$. Now as each factor of $PR_e$ is a cyclic group of even order, this implies that $PR_e$ has at most two factors, i.e., $q \le 2$. 
Since $PR_e$ is a product over the odd $\upsilon_j$'s and we have that $\sum \eta_i + \sum \upsilon_j = 2g-2$, this implies that $q$ is even and is therefore either $0$ or $2$. So if $PR_e$ is nontrivial it must be isomorphic to $(PR_e') \times \Z/2$. Necessarily then $(\upsilon_1 + 1)/{2}$ and $(\upsilon_2 + 1)/{2}$ are coprime and
\[
PR_e' \cong \Z / \left(\frac{\upsilon_1 + 1}{2}\right) \Z \times \Z / \left(\frac{\upsilon_2 + 1}{2}\right)\Z.
\]

The remaining hypothesis that $PR_o$ and $PR_e'$ be cyclic of coprime order readily implies the claim that the elements of
\[ \left\{\eta_1 + 1, \ldots, \eta_p + 1, \frac{\upsilon_1 + 1}{2}, \ldots, \frac{\upsilon_q + 1}{2} \right\}\]
are pairwise coprime as required.
\end{proof}

\section{Monodromy of strata}\label{section:monodromy}

In this section, we fuse our discussion of framings and mapping class groups with the theory of abelian differentials to deduce Theorem \ref{mainthm:absolute_monodromy} from Theorem \ref{mainthm:genset}.

We begin in Section \ref{subsec:flatbasics} by collecting basic results about abelian differentials and their strata. We also record Kontsevich and Zorich's seminal classification (Theorem \ref{thm:KZ_strata}) of the connected components of strata, together with a slight variation in which one labels the zeros of the differential (Lemma \ref{lemma:permutation_monodromy}).

With these foundations laid, we proceed to discuss the relationship between abelian differentials and framings on punctured, bordered, and pronged surfaces in Section \ref{subsec:flattoframed}. This section also contains a detailed description of the real oriented blow-up of an abelian differential along its zero locus (Construction \ref{construction:blowup}). The main result of this subsection is Boissy's classification of the components of strata of {\em prong--marked} differentials (Theorem \ref{theorem:Boissy_framed}) and its implications for monodromy (Corollary \ref{corollary:PR_monodromy}).

Mapping class groups enter the picture in Section \ref{subsec:markings}, in which we define a family of coverings of strata (first introduced in \cite{BSW_horo}) whose deck groups are mapping class groups of punctured, bordered, and pronged surfaces (see Diagram \eqref{spaces}). Using the relations between these spaces, we then prove that the monodromy of each covering must preserve the appropriate framing datum (Lemma \ref{lemma:mon1_preserves_framing} and Corollaries \ref{cor:mon2_preserves_framing} and \ref{cor:mon3_preserves_framing}).

We establish the reverse inclusions (that the monodromy is the entire stabilizer of the framing) in Section \ref{subsec:genmon} as Theorems \ref{thm:relative_monodromy}, \ref{thm:blowup_monodromy}, and \ref{mainthm:absolute_monodromy}. 

\subsection{Abelian differentials}\label{subsec:flatbasics}

An {\em abelian differential} $\omega$ is a holomorphic 1--form on a Riemann surface $X$. The collection of all abelian differentials forms a vector bundle (in the orbifold sense) $\OM_g$ over the moduli space of curves. The complement of the zero section of this bundle is naturally partitioned into disjoint subvarieties called {\em strata} which have a fixed number and order of zeros. For $\sing = (\kappa_1, \ldots, \kappa_n)$, we will let $\OM_g(\sing)$ denote the space of all  pairs $(X, \omega)$ where $\omega$ is an abelian differential on $X$ which has zeros of order $\kappa_1, \ldots, \kappa_n$. Throughout this section, we will use $Z$ to denote the set of zeros of $\omega$.

Away from its zeros, an abelian differential has canonical local coordinates in which it can be written as $dz$; the transition maps between these coordinate charts are translations, so the data of an abelian differential $\omega$ on a Riemann surface $X$ is also sometimes called a {\em translation structure}. Pulling back the Euclidean metric of $\C$ along the canonical coordinates defines a flat metric on $X$ with a cone point of angle $2\pi(k+1)$ at each zero of order $k$.

Every abelian differential $\omega$ also defines a horizontal vector field $H_\omega:= 1/\omega$ on $X$ with singularities of order $-\kappa_1, \ldots, -\kappa_n$, and hence gives rise to a prong structure $\vec Z$ with a prong point of order $\kappa_i+1$ at the $i^\text{th}$ zero of $\omega$. Forgetting the prong structure and the marked points, the differential induces a $\gcd(\sing)$--spin structure $\phi$ on $X$ (see Definition \ref{definition:rspin} and the discussion which follows it). In particular, if $\gcd(\sing)$ is even, then there is a well-defined mod-$2$ reduction of the spin structure.

\begin{remark}
For a more thorough treatment of the relationship between (higher) spin structures and abelian differentials, the reader is directed to \cite{C_strata1} and \cite{CS_strata2}.
\end{remark}

While it does not make sense to compare spin structures on different (unmarked) surfaces, when $\gcd(\sing)$ is even the Arf invariant of $\phi$ is well--defined even without choice of marking. Moreover, $\Arf(\phi)$ is invariant under deformation and classifies the non-hyperelliptic components of $\OM_g(\sing)$.

\begin{theorem}[Theorem 1 of \cite{KZ_strata}]\label{thm:KZ_strata}
Let $g \ge 4$ and $\sing=(\kappa_1, \ldots, \kappa_n)$ be a partition of $2g-2$. Then $\OM_g(\sing)$ has at most three components:
\begin{itemize}
\item If $\sing = (2g-2)$ or $(g-1, g-1)$ then there is a unique component of $\OM_g(\sing)$ that consists entirely of hyperelliptic differentials.
\footnote{Recall that an abelian differential is {\em hyperelliptic} if arises as the global square root of a quadratic differential on $\widehat{\C}$ with at worst simple poles.}
\item If $\gcd(\sing)$ is even then there are two components containing non-hyperelliptic differentials, classified by the Arf invariant of their induced $2$--spin structure.
\item If $\gcd(\sing)$ is odd then there is a unique component containing non-hyperelliptic differentials.
\end{itemize}
\end{theorem}

We will focus our attention on the non-hyperelliptic components of $\OM_g(\sing)$; the hyperelliptic components are $K(G, 1)$'s for (finite extensions of) spherical braid groups \cite[\S1.4]{LM_strata} and therefore their monodromy can be understood entirely through Birman--Hilden theory. See \cite[\S2]{C_strata1} for a more thorough discussion.

We will often find it convenient to label the zeros of an abelian differential so as to distinguish them. The corresponding stratum $\OM_g^{\lab}(\sing)$ of abelian differentials with labeled singularities is clearly a finite cover of $\OM_g(\sing)$ with deck group 
\[\Sym(\sing) = \prod_{j=1}^m \Sym(r_j)\]
where $\sing=(k_1^{r_1}, \ldots, k_m^{r_m})$ and $\Sym(n)$ is the symmetric group on $n$ letters. Moreover, it is not hard to show that each preimage of a connected component of $\OM_g(\sing)$ is itself connected (see, e.g., \cite[Proposition 4.1]{Boissy_Rauzy}), and hence the monodromy of this covering map is the entirety of the deck group.

\begin{lemma}\label{lemma:permutation_monodromy}
Let $\comp$ be a component of $\OM(\sing)$. Then the monodromy homomorphism
\[\pi^{\orb}_1(\comp) \rightarrow \Sym(\sing)\]
associated to the covering $\OM_g^{\lab}(\sing) \rightarrow \OM_g(\sing)$ is surjective.
\end{lemma}

\subsection{From differentials to framings}\label{subsec:flattoframed}

To relate our results on framed mapping class groups to strata, we must first understand the different types of framings which an abelian differential induces.

As observed above, the horizontal vector field of an abelian differential $(X,\omega) \in \OM_g(\sing)$ induces a $\gcd(\sing)$--spin structure on the underlying closed surface $X$. Moreover, since $\omega$ has canonical coordinates (in which it looks like $dz$) away from its zeros, we see that $\omega$ induces a trivialization of $T(X \setminus Z)$ and hence an absolute framing of $X \setminus Z$ (in the sense of
\S\ref{subsection:punctured}).

To obtain a relative framing from $\omega$, we must first identify the surface $X^*$ which is to be framed. Informally, this ``real oriented blow-up'' $X^*$ is obtained by replacing each zero of $\omega$ by the circle of directions at that point; the horizontal vector field then extends by continuity along rays to a vector field (and eventually, a framing) on $X^*$ whose boundary data can be read off from the order of the singularities.
For more on the real oriented blow-up construction in the context of translation surfaces, see \cite[\S2.5]{BSW_horo}.

\begin{construction}[Real oriented blow-ups]\label{construction:blowup}
We begin by first describing the real oriented blow-up of $0 \in \C$; this toy example will provide a local model for the blow-up of a translation surface. Equipping $\C$ with polar coordinates $z = re^{i\theta}$ gives a parametrization of $\C \setminus \{0\}$ by the infinite open half--cylinder $\R_{>0} \times [0, 2\pi]/(0=2\pi)$.
The {\em real oriented blow-up} of $0 \in \C$ is the closed half--cylinder $\R_{\ge 0} \times [0, 2\pi]/(0=2\pi)$, which has a natural surjective map onto $\C$ extending polar coordinates. The fiber of this map above $0$ is therefore identified with the circle of directions at $0$.

To blow up a cone point, let $k \ge 1$ and consider the branched cover of $\C$ given by $z \mapsto z^k$. The Euclidean metric of $\C$ pulls back to a cone metric with cone angle $2k \pi$ at $0$, and similarly the polar parametrization of $\C \setminus \{0\}$ pulls back to a parametrization by $\R_{>0} \times [0, 2k\pi]/(0=2k\pi)$. Therefore, the blow-up of a cone point of angle $2k\pi$ is the corresponding closed cylinder $\R_{\ge 0} \times [0, 2k\pi]/(0=2k\pi)$, and the fiber above $0$ corresponds to the $2k\pi$'s worth of directions at $0$.

Now suppose that $(X,\omega) \in \OM^{\lab}_g(\sing)$ with zeros at points $p_1, \ldots, p_n$. The {\em real oriented blow-up} $X^*$ of $X$ is the space obtained after blowing up each cone point $p_i$ via the above construction.
\end{construction}

Observe that $X^*$ is naturally a surface of the same genus as $X$ with boundary components $\Delta_1, \ldots, \Delta_n$, the $i^{\text{th}}$ of which comes with an identification with the $(k_i+1)$--fold cover of the circle (which is of course just a circle itself, equipped with a cyclic symmetry of order $(k_i+1)$).

Moreover, the unit horizontal vector field $H$ of $\omega$ induces a (nonvanishing) unit vector field $H^*$ on $X^*$ by extending $H$ continuously along rays into each cone point. For each boundary component $\Delta_i$, the vector field $H^*|_{\Delta_i}$ is invariant under the cyclic symmetry described above, and its winding number is $-1-k_i$.

Hence $H^*$ induces a framing $\phi$ of $X^*$ with boundary signature
\[ \sig(\phi) = (-1 - k_1, \ldots, -1 - k_n) \]
which is compatible (in the sense of Definition \ref{definition:prong_compat}) with the prong structure $(X, \vec Z)$ induced by $\omega$.

\para{Prong markings}
While the blowup $X^*$ of $(X, \omega)$ is topologically a surface with boundary, it is more accurate to view $X^*$ as (the blow-up of) a surface with prong structure. In particular, there exist loops
\footnote{For example, the $\textsf{SO}(2)$ action on $\OM_g^{\lab}(\sing)$ \cite[\S2.10]{BSW_horo}.}
in $\OM_g^{\lab}(\sing)$ which rotate the prongs of $(X, \vec Z)$ and hence act by fractional twists on $\partial X^*$. This phenomenon can be interpreted as the monodromy of $\OM_g^{\lab}(\sing)$ taking values in $\Mod^*_{g,n}$ rather than $\Mod_{g,n}$ (see Section \ref{subsec:markings} below).

By passing to a finite cover of $\OM_g^{\lab}(\sing)$ which remembers more information, we may therefore constrain the monodromy to lie in $\Mod_{g,n}$. To that end, define a {\em prong marking} of an abelian differential $\omega$ to be a labeling of its zeros together with a choice of positive horizontal separatrix at each zero (these are also sometimes called {\em framings of $\omega$}, as in \cite{Boissy_framed}). In terms of the prong structure $\vec Z$ on $X$ induced by $\omega$, a prong marking chooses a prong at each zero.

The (components of the) space $\OM^{\pr}_g(\sing)$ of prong--marked abelian differentials are finite covers of (the components of) a stratum $\OM^{\lab}_g(\sing)$. Moreover, the deck transformations of this covering rotate the choice of specified prong at each zero and hence the deck group is exactly $PR$.

Any loop in $\OM^{\pr}_g(\sing)$ preserves the prong marking and so acts as the identity on $\partial X^*$ by the correspondence outlined in Section \ref{section:pronged}. Therefore, the real oriented blow-up of a prong--marked abelian differential can be consistently interpreted as a surface with boundary (and hence the monodromy of $\OM^{\pr}_g(\sing)$ is in $\Mod_{g,n}$, as we will see in \S\ref{subsec:markings}).

We now record Boissy's classification of the components of $\OM^{\pr}(\sing)$. Our statement of following theorem looks rather different than that which appears in \cite{Boissy_framed}; we reconcile these differences in Remark \ref{remark:Boissy_Arf} at the end of this section.

\begin{theorem}[c.f. Theorem 1.3 of \cite{Boissy_framed}] \label{theorem:Boissy_framed}
Suppose that $\comp$ is a non-hyperelliptic connected component of $\OM^{\lab}_g(\sing)$. Then the preimage of $\comp$ in $\OM_g^{\pr}(\sing)$ has
\begin{itemize}
\item one connected component if $\gcd(\sing)$ is even, and
\item two connected components if $\gcd(\sing)$ is odd, distinguished by the generalized Arf invariant \eqref{equation:arf} of the relative framing on the real oriented blow-up.
\end{itemize}
\end{theorem}

Combining this classification with Theorem \ref{thm:KZ_strata} immediately implies that for $g \ge 4$ there are exactly two non-hyperelliptic  components of $\OM^{\pr}_g(\sing)$, classified by generalized Arf invariant (compare with Proposition \ref{proposition:orbits}).

Translating Theorem \ref{theorem:Boissy_framed} into the action of the deck group $PR$ therefore identifies the monodromy of the covering $\OM_g^{\pr}(\sing) \rightarrow \OM^{\lab}_g(\sing)$:

\begin{corollary}\label{corollary:PR_monodromy}
Suppose that $g \ge 3$ and let $\comp$ be a non-hyperelliptic connected component of $\OM^{\lab}_g(\sing)$. Then the image of the monodromy homomorphism $\pi^{\orb}_1(\comp) \rightarrow PR$ is exactly the subgroup $PR'$ of \eqref{eqn:PRprime}.
\end{corollary}

In particular, when $\gcd(\sing)$ is even the monodromy is all of $PR$.

\begin{proof}
This is an easy consequence of Lemma \ref{lemma:fractionalTL} together with the formula for the generalized Arf invariant \eqref{equation:arf}; simply observe that a prong rotation changes the winding number of every arc incident to the prong point by $\pm 1$.

More explicitly, pick $\omega \in \comp$ and a preimage $\omega^{\pr}_1 \in \OM_g^{\pr}(\sing)$. Choose a distinguished geometric basis $\mathcal B_1$ on $X^*$ with its basepoints specified by $\omega^{\pr}_1$. Then given another preimage $\omega^{\pr}_2$ of the same $\omega \in \comp$ there is a corresponding basis $\mathcal B_2$ which differs from $\mathcal B_1$ by a fractional multitwist $\tau$ (this choice is not unique, but is determined up to full twists about the boundary of the blowup $X^*$ of $\omega$).

Computing the Arf invariants of $\omega^{\pr}_1$ and $\omega^{\pr}_2$ with respect to these bases, we see that the two Arf invariants agree (and hence $\omega^{\pr}_1$ and $\omega^{\pr}_2$ live in the same component of $\OM_g^{\pr}(\sing)$) if and only if
\[\sum_{i=2}^{n} \phi(a_i)\kappa_i = \sum{}' \phi(a_i)
 = \sum{}' \phi(\tau(a_i)) = \sum_{i=2}^{n}  \phi(\tau(a_i)) \kappa_i \pmod 2.\]
In particular, this equality holds if and only if $D(\tau) \in PR'$. 

We observe that the choice of $\tau$ does not matter as any two choices $\tau$ and $\tau'$ differ by full twists about the boundary of $X^*$: twists about boundary components with even winding number do not change the parity of $\phi(\tau(a_i))$ while twists about odd boundaries change parity of arcs which do not end up contributing to $\sum{}'$.
\end{proof}

\begin{remark}\label{remark:Boissy_Arf}
In addition to the differences in terminology, the statement and proof of Theorem \ref{theorem:Boissy_framed} in \cite{Boissy_framed} distinguishes the two non-hyperelliptic components of $\OM_g^{\pr}(\sing)$ by a different invariant. There, Boissy differentiates the two by first choosing a set of arcs on $\omega \in \OM_g^{\pr}(\sing)$ which pair up the odd order zeros, are transverse to the horizontal foliation, and are tangent to the specified prongs. Applying the ``parallelogram construction'' of \cite{EMZ} to $\omega$ along these arcs results in a new differential $\omega'$ of higher genus with all even zeros; the Arf invariant of the 2--spin structure induced by $\omega'$ then distinguishes the components of $\OM_g^{\pr}(\sing)$. 

The reader may verify that Boissy's invariant coincides with the generalized Arf invariant by computing the contribution to $\Arf(\omega')$ of each new handle and comparing it to the corresponding term in the expression for $\Arf(\phi)$ (where $\phi$ is interpreted as a framing of $X^*$).
\end{remark}

\subsection{Markings and monodromy}\label{subsec:markings}

In order to compare the framings induced by differentials on different Riemann surfaces, we pull them back to a framings of a reference topological surface. To that end, we need to understand markings of $X$, $X^*$, and $(X, \vec Z)$, together with the corresponding spaces of marked differentials.

The coarsest type of marking data we consider in this section are homeomorphisms from a closed surface $\Sigma_g$ to $X$ which take a specified set of (labeled) marked points $P$ to the (labeled) zeros $Z$ of $\omega$.
With this data, we can define the corresponding space $\OT_g^{\lab}(\sing)$ of marked differentials with marked points as the space of triples $(X, \omega, f)$, where $(X, \omega) \in \OM_g^{\lab}(\sing)$ and $f:(\Sigma, P) \rightarrow (X, Z)$ is an isotopy class of homeomorphisms of pairs. 
The space of marked differentials with marked points is naturally a (disconnected, orbifold) covering space of $\OM_g^{\lab}(\sing)$ whose deck transformations correspond to changing the marking, so its deck group is $\PMod_g^n$.

On the other end of the spectrum, we may also mark a differential $\omega \in \OM_g^{\lab}(\sing)$ by a pronged surface.
Fix a topological pronged surface $(\Sigma, \vec P)$ with prong points of order $\kappa_1 +1, \ldots, \kappa_n+1$, and recall that that any differential $\omega$ naturally equips its underlying surface $X$ with a prong structure $\vec Z$ of the same prong type.
Then a {\em marking} of $(X, \omega)$ by $(\Sigma, \vec P)$ is a diffeomorphism of pairs $f$ from $(\Sigma, P)$ to $(X, Z)$ such that $Df$ takes the prong structure of $\vec P$ to that of $\vec Z$.
Equivalently, $f$ is a diffeomorphism from the real oriented blow-up $\Sigma^*$ of $\Sigma$ to $X^*$ which takes the distinguished points
\footnote{Recall that $\partial \Sigma^*$ can be identified with the circle of directions above each blown--up point, hence a prong point of order $k$ corresponds to a boundary component with $k$ distinguished points.} 
of $\partial \Sigma^*$ to those of $\partial X^*$.

We may now record the definition of the corresponding space of marked differentials:

\begin{definition}[c.f. \S2.9 of \cite{BSW_horo}]
$\OT_g^{\pr}(\sing)$ is the space of marked differentials with a prong--marking, that is, the space of triples $(X, \omega, f)$ where $(X, \omega)$ is an abelian differential in $\OM_g^{\lab}(\sing)$ and $f:(\Sigma, \vec P) \rightarrow (X, \vec Z)$ is an isotopy class of marking of pronged surfaces.
\end{definition}

Forgetting the the prong structure induces a covering map from $\OT_g^{\pr}(\sing)$ to $\OT_g^{\lab}(\sing)$, the deck group of which is exactly $FT$. Similarly, forgetting the marking of the surface but {\em remembering the marking of its boundary} (i.e., remembering $\partial f: \partial \Sigma_{g,n} \rightarrow \partial X^*$) induces a map from $\OT_g^{\pr}(\sing)$ to $\OM_g^{\pr}(\sing)$. Since this map remembers the boundary marking, hence the specified prong, the deck group of this covering is thus a change--of--marking group which preserves the boundary pointwise.

\begin{lemma}[Corollary 2.7 of \cite{BSW_horo}]\label{lemma:OT_cover}
The space $\OT_g^{\pr}(\sing)$ is a (disconnected, orbifold) covering of both $\OM_g^{\pr}(\sing)$ and $\OT_g^{\lab}(\sing)$. Moreover, the deck group of the former covering is $\Mod_{g,n}$, and the latter $FT$.
\end{lemma}

Putting this Lemma together with either \eqref{ses:PR1} or \eqref{eqn:BSW_tw}, we see that the deck group of the covering
$\OT_g^{\pr}(\sing) \rightarrow \OM_g^{\lab}(\sing)$
is exactly the pronged mapping class group $\Mod_{g,n}^*$.

We summarize the relationship between all of these spaces (and their deck groups) in the following diagram, in which arrow labels correspond to the deck group of the covering:
\begin{equation}\label{spaces}
\begin{tikzcd}[row sep=7em, column sep = 7em, font = \normalsize]
\OT_g^{\pr}(\sing)
		\arrow[r, "FT"] 	
		\arrow[d, "\Mod_{g,n}"]
		\arrow[dr, "\Mod^*_{g,n}"]
	& \OT_g^{\lab}(\sing)
		\arrow[d, "\PMod_g^n"] 
		\arrow[dr, "\Mod_g^n"] \\
\OM_g^{\pr}(\sing)
		\arrow[r, swap, "PR"] 
	& \OM_g^{\lab}(\sing)
		\arrow[r, swap, "\Sym(\sing)"] 
	& \OM_g(\sing)
\end{tikzcd}
\end{equation}
Observe that the furthest left triangle demonstrates the exact sequence \eqref{ses:PR1}, while the center triangle corresponds to \eqref{eqn:BSW_tw}.

\para{Constraining the monodromy}
While the deck groups of the coverings in \eqref{spaces} are easy to describe, their elements generally permute the components of the corresponding covers. We now shift our focus to the {\em stabilizer} of a component of one of these covers; for concreteness, throughout the rest of this section we will focus on the covering $\OT_g^{\pr}(\sing) \rightarrow \OM_g^{\lab}(\sing)$ and use this discussion to deduce the corresponding results for each of the intermediate coverings.

We observe that by the path--lifting property, understanding the stabilizer of a component of $\OT_g^{\pr}(\sing)$ is equivalent to understanding the monodromy of the component of $\OM_g^{\lab}(\sing)$ which it covers (c.f. \cite[Proposition 3.7]{CS_strata2}).

Monodromy groups are always only defined up to conjugacy, so we fix some reference marking $f: (\Sigma, \vec P) \rightarrow (X, \vec Z)$ (equivalently, a lift of a basepoint in $\OM_g^{\lab}(\sing)$ to $\OT^{\pr}(\sing)$). Pulling back the framing of $(X, \vec Z)$ induced by $\omega$ along $f$ induces a framing $\phi$ of $\Sigma$ compatible with $\vec P$. Path--lifting then allows us to place the following constraint on the monodromy, which is just a version of \cite[Corollary 4.8]{C_strata1} adapted to the setting of framings rather than $r$--spin structures.

\begin{lemma}\label{lemma:mon1_preserves_framing}
Let $\comp$ be a component of $\OM_g^{\lab}(\sing)$ and fix some basepoint $(X, \omega) \in \comp$. Then the monodromy of the covering of $\comp$ by (a component of) $\OT^{\pr}(\sing)$ preserves the induced framing $\phi$ on the pronged surface $X$. In other words, the image of 
\[\rho: \pi_1^{\orb}(\comp, (X, \omega)) \rightarrow \Mod_{g,n}^*\]
is contained inside of $\Mod_{g,n}^*[\phi]$.
\end{lemma}
\begin{proof}
Let $f$ be a marking of $(X, \vec Z)$ by $(\Sigma, \vec P)$ and let $\gamma$ be a loop in $\comp$ based at $(X, \omega)$. 

By Lemma \ref{lemma:OT_cover}, the loop $\gamma$ lifts to a path $\tilde \gamma$ in $\OT_g^{\pr}(\sing)$, and hence a path of marked prong--marked abelian differentials $(X_t, \omega_t, f_t)$ with horizontal vector fields $H_t$. Pulling back $H_t$ by $f_t$ yields a continuous family of vector fields on $(\Sigma, \vec P)$, all compatible with the prong structure. Let $\phi_t$ denote the associated framing of $(\Sigma, \vec P)$.

Then since the vector fields vary continuously (and do not vanish on $\Sigma \setminus P$), the winding number $\phi_t(a)$ of every simple closed curve or legal arc $a$ is continuous in $t$. However, $\phi_t(a)$ takes values only in $\Z$ or $\Z + (1/2)$ and must therefore be constant over the entire path $\tilde \gamma$. Thus $\rho(\gamma)$ preserves the winding number of every simple closed curve and legal arc, hence the entire framing.
\end{proof}

In view of the sequences \eqref{ses:pronged} and \eqref{eqn:BSW_tw}, together with \eqref{spaces}, this implies the following two results, where $\phi$ and $\bar \phi$ denote the relative and absolute framings induced on $X^*$ and $X \setminus Z$, respectively.

\begin{corollary}\label{cor:mon2_preserves_framing}
The monodromy of the covering $\OT^{\pr}_g(\sing) \rightarrow \OM^{\pr}_g(\sing)$ lies inside of $\Mod_{g,n}[\phi]$.
\end{corollary}

\begin{corollary}\label{cor:mon3_preserves_framing}
The monodromy of the covering $\OT^{\lab}_g(\sing) \rightarrow \OM^{\lab}_g(\sing)$ lies inside of $\PMod_g^n[\bar \phi]$.
\end{corollary}

We note that these corollaries can also be proven directly using the same argument as Lemma \ref{lemma:mon1_preserves_framing}.

For brevity, in the sequel we use $\mathcal G^*$, $\mathcal G^{\pr}$, and $\mathcal G^{\lab}$ to denote the monodromy groups of the coverings appearing in Lemma \ref{lemma:mon1_preserves_framing} and Corollaries \ref{cor:mon2_preserves_framing} and \ref{cor:mon3_preserves_framing}, respectively.\\

\subsection{Generating framed mapping class groups}\label{subsec:genmon}
Now that we have shown that the monodromy of each of the coverings under consideration stabilizes the relevant framing, we can use Theorem \ref{mainthm:genset} to show that the group is the {\em entire} stabilizer of the framing:

\begin{theorem}\label{thm:relative_monodromy}
Suppose that $g \ge 5$ and let $\comp$ be a non-hyperelliptic component of $\OM_g^{\pr}(\sing)$. Then 
\[\mathcal G^{\pr}(\comp) \cong \Mod_{g,n}[\phi]\]
where $\phi$ is the relative framing of the real oriented blow-up $X^*$ induced by the horizontal vector field on any $(X, \omega) \in \OM_g^{\pr}(\sing)$.
\end{theorem}

Assuming this theorem, we can leverage our understanding of the relationship between framing stabilizers to characterize both $\mathcal G^*$ and $\mathcal G^{\lab}$.

\begin{theorem}\label{thm:blowup_monodromy}
Suppose that $g \ge 5$ and let $\comp$ be a non-hyperelliptic component of $\OM_g^{\lab}(\sing)$. Then 
\[\mathcal G^*(\comp) \cong \Mod_{g,n}^*[\phi]\]
where $\phi$ is the relative framing of a pronged surface induced by the horizontal vector field of any $(X, \omega) \in \OM_g^{\lab}(\sing)$.
\end{theorem}
\begin{proof}
By Lemma \ref{lemma:mon1_preserves_framing}, $\mathcal G^* \le \Mod_{g,n}^*[\phi]$. By Theorem \ref{theorem:Boissy_framed}, $\mathcal G^*$ surjects onto $PR'$, and so by Theorems \ref{thm:relative_monodromy} and \ref{theorem:pronged}, it follows that $\Mod_{g,n}^*[\phi] \le \mathcal G^*$.
\end{proof}

Forgetting the prongs, we can push this result down to the mapping class group with marked points to complete the proof of Theorem \ref{mainthm:absolute_monodromy}.

\begin{proof}[Proof of Theorem \ref{mainthm:absolute_monodromy}]
Observe that $\mathcal G^{\lab}$ is the image of $\mathcal{G}^*$ under the surjection of \eqref{eqn:BSW_tw}. Therefore, Theorems \ref{thm:blowup_monodromy} and \ref{theorem:absolute} together imply that 
\[\mathcal G^{\lab} = \PMod_g^n[\bar \phi].\]
Combining this result with Lemma \ref{lemma:permutation_monodromy} and the short exact sequence
\[1 \rightarrow \PMod_g^n[\bar \phi] \rightarrow \Mod_g^n[\bar \phi] \rightarrow \Sym(\sing) \rightarrow 1\]
completes the proof of the theorem.
\end{proof}

\para{Cylinder shears and prototypes}
It therefore remains to show that $\Mod_{g,n}[\phi] \le \mathcal G^{\pr}$. In order to demonstrate this, we will need to build a collection of loops of abelian differentials with prescribed monodromy.

Recall that a {\em cylinder} on an abelian differential is an embedded Euclidean cylinder with no singularities in its interior. Shearing the cylinder while leaving the rest of the surface fixed results in a loop in a stratum whose monodromy is the Dehn twist of the core curve of the cylinder.

\begin{lemma}[c.f. Lemma 6.2 of \cite{C_strata1}]\label{lemma:cyltwist}
Let $\comp$ be a component of $\OM_g^{\pr}(\sing)$ and suppose that $\omega \in \comp$ has a cylinder with core curve $c$. Then $T_c \in \mathcal G^{\pr}$.
\end{lemma}

Now that we can use cylinder shears to exhibit Dehn twists, it remains to show that there exist differentials in a stratum with a configuration of cylinders to which we can apply Theorem \ref{mainthm:genset}.

\begin{figure}[ht]
\labellist
\small
\pinlabel $a_0=b_0$ [tl] at 115.20 133.92
\pinlabel $a_1$ [br] at 65.56 202.36
\pinlabel $a_2$ [bl] at 90.04 177.12
\pinlabel $a_3$ [t] at 37.44 146.88
\pinlabel $a_4$ [t] at 72.00 157.20
\pinlabel $a_5$ [bl] at 120.96 167.04
\pinlabel $a_6$ [t] at 141.12 157.16
\pinlabel $a_{2g-1}$ [t] at 296.64 146.88
\pinlabel $b_1$ [l] at 169.92 138.24
\pinlabel $b_2$ [l] at 230.40 138.24
\pinlabel $b_{g-3}$ [bl] at 314.32 164.16
\pinlabel $b_{2g-8}$ [l] at 224.64 198.72
\pinlabel $b_{2g-7}$ [l] at 167.04 198.72
\pinlabel $b_{2g-6}$ [l] at 109.44 198.72
\pinlabel $b_{2g-5}$ [tr] at 74.88 223.20
\pinlabel $b_{2g-4}$ [tr] at 43.20 195.84
\pinlabel $b_{2g-3}$ [tr] at 23.64 171.84
\pinlabel $a_0=b_0$ [tl] at 164.16 80.64
\pinlabel $a_1$ [b] at 48.96 63.36
\pinlabel $a_2$ [t] at 74.88 51.84
\pinlabel $a_{2g-1}$ [t] at 293.76 38.88
\pinlabel $b_1$ [l] at 103.68 83.52
\pinlabel $b_2$ [t] at 17.28 51.84
\pinlabel $b_3$ [l] at 103.68 20.16
\pinlabel $b_4$ [l] at 161.28 20.16
\pinlabel $b_5$ [l] at 218.88 20.16
\pinlabel $b_{g+1}$ [t] at 325.44 53.28
\pinlabel $b_{2g-3}$ [l] at 224.64 78.52
\pinlabel Type\hspace{3pt}1 [c] at 270 180
\pinlabel Type\hspace{3pt}2 [c] at 270 10
\pinlabel $\Delta_1$ at 200 138.24
\pinlabel $\Delta_1$ at 35 80.52
\endlabellist
\includegraphics[scale=1.2]{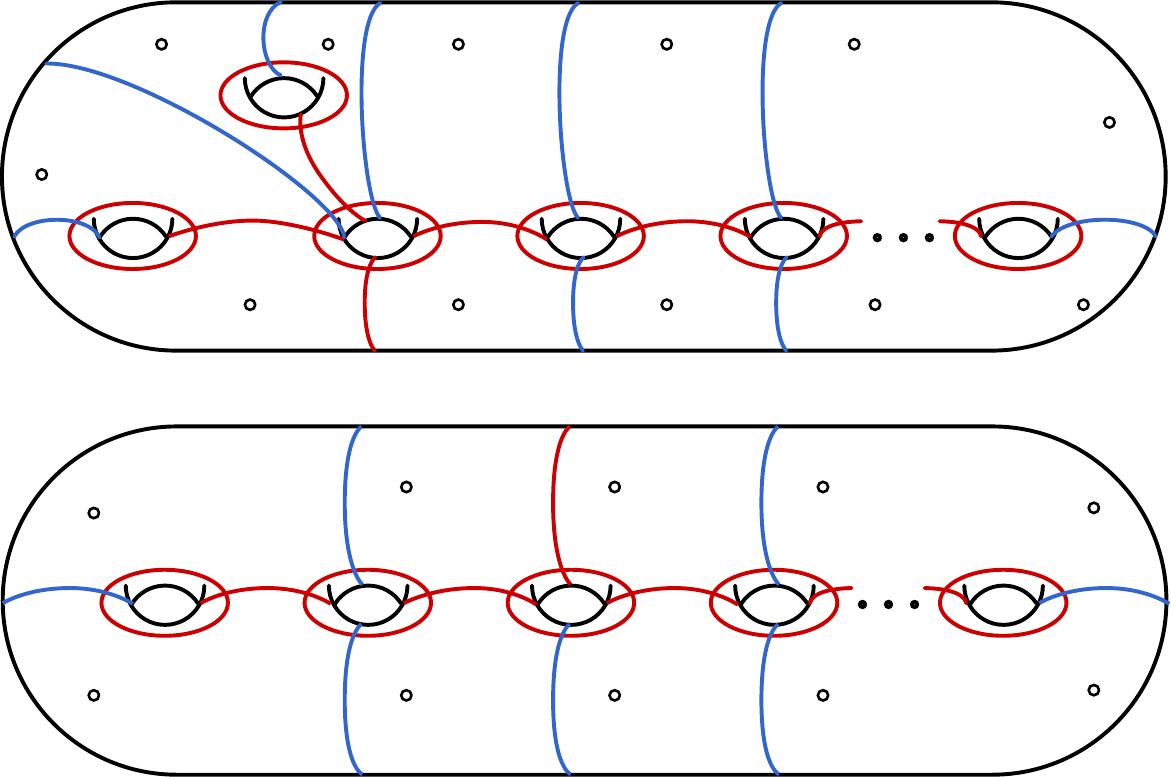}
\caption{Configurations of types 1 and 2 determining generating sets for $\Mod_{g,n}[\phi]$. We label the boundary components $\Delta_{i}$ for $i = 0, \dots, 2g-3$, with $\Delta_i$ positioned between $b_i$ and $b_{i+1}$ (for clarity most of the labels have been omitted). Given a partition $\sing = (\kappa_1, \dots, \kappa_n)$ of $2g-2$, only the $n$ boundary components $\Delta_{i_\ell}$ for $i_\ell$ of the form $\sum_{j = 1}^\ell \kappa_j$ (interpreted mod $2g-2$) are included, and $\Delta_{i_\ell}$ is assigned the signature $-1-\kappa_{\ell}$. Under this scheme, each complementary region determined by $\mathcal C$ contains exactly one boundary component, and the signatures are arranged so that each curve in $\mathcal C$ is admissible.}
\label{figure:genset}
\end{figure}
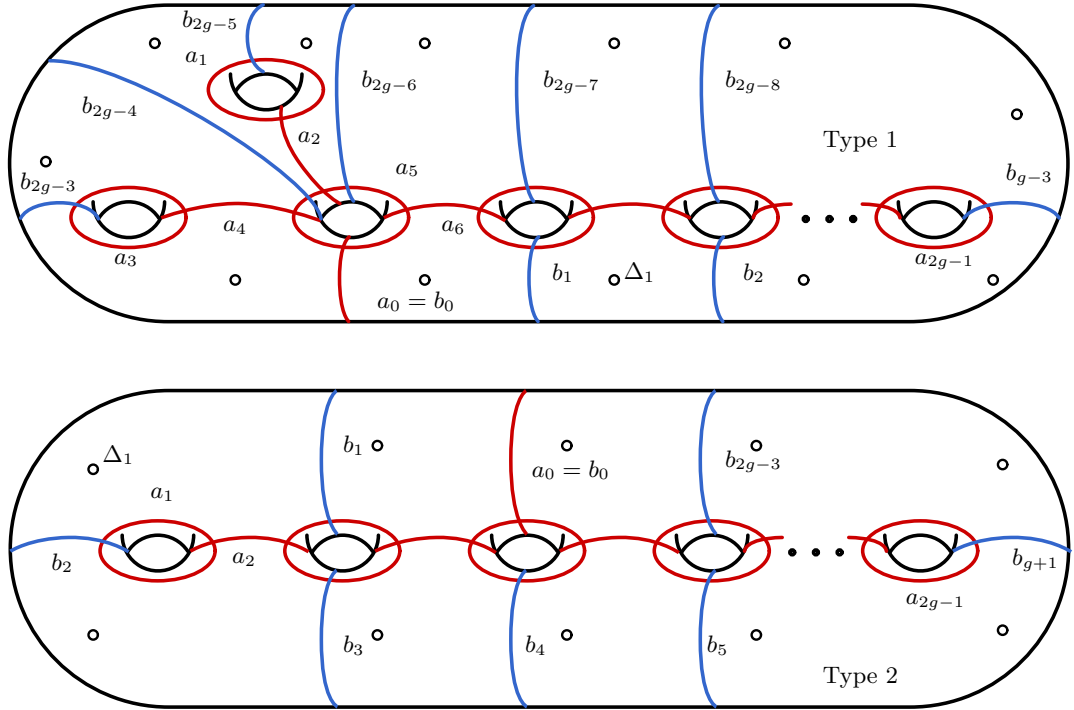

\begin{lemma}[Lemma 3.14 of \cite{CS_strata2}]\label{lemma:prototypes}
Suppose that $g \ge 5$ and $\sing = (\kappa_1, \ldots, \kappa_n)$ is a partition of $2g-2$. Let $\comp$ be a non-hyperelliptic component of $\OM_g(\sing)$. Set $\mathcal C$ to be the curve system specified in Figure \ref{figure:genset}, of type 1 if $\gcd(\kappa_1, \ldots, \kappa_n)$ is even and
\[
\Arf(\comp) = \begin{cases}
1 	& g \equiv 0,3 \pmod 4\\
0	& g \equiv 1,2 \pmod 4,
\end{cases}
\]
and of type 2 otherwise. Then there exists some $\omega \in \comp$ whose horizontal and vertical cylinders are exactly the curves of $\mathcal{C}$.
\end{lemma}

By labeling a separatrix at each zero, these prototype surfaces also give rise to prong--marked abelian differentials with specified Arf invariant.

\begin{lemma}\label{lemma:framed_prototypes}
Let $g$, $\sing$, $\comp$, and $\mathcal C$ be as above. Let $\comp_{\pr}$ be a (non-hyperelliptic) component of the stratum of prong--marked abelian differentials which covers $\comp$. Then there exists some $\omega \in \comp_{\pr}$ whose horizontal and vertical cylinders are exactly the curves of $\mathcal{C}$.
\end{lemma}
\begin{proof}
Let $\omega \in \comp$ be the differential constructed in Lemma \ref{lemma:prototypes}.
If $\gcd(\sing)$ is even, then by Theorem \ref{theorem:Boissy_framed} the entire preimage of $\comp$ in $\OM_g^{\lab}(\sing)$ is $\comp_{\pr}$ and so any prong marking of $\omega$ lives in $\comp_{\pr}$.
If $\gcd(\sing)$ is odd, then choose an arbitrary prong marking $\omega_{\pr}$ of $\omega$ and let $p_i$ be some zero of odd order (corresponding to a prong point of even order).

Now by Corollary \ref{corollary:PR_monodromy}, since $T_{\vec p_i} \notin PR'$ the differentials $\omega_{\pr}$ and $T_{\vec p_i}(\omega_{\pr})$ must lie in different components of $\OM_g^{\pr}(\sing)$. Since there are only two components (by Theorem \ref{theorem:Boissy_framed}), one of $\{T_{\vec p_i}(\omega_{\pr}), \omega_{\pr} \}$ must be in $\comp_{\pr}$.
\end{proof}

It is now completely straightforward to deduce that the monodromy group of a component $\comp_{\pr}$ of $\OM_g^{\pr}(\sing)$ is exactly the stabilizer of the relative framing induced on the real oriented blow-up.

\begin{proof}[Proof of Theorem \ref{thm:relative_monodromy}]
Let $(X, \omega) \in \comp_{\pr}$ be the prototype surface built in Lemma \ref{lemma:framed_prototypes} above. Observe that each configuration of curves specified in Figure \ref{figure:genset} spans the indicated surface, has intersection graph a tree, and contains $E_6$ as a subgraph. Then by Theorem \ref{mainthm:genset}, Lemma \ref{lemma:cyltwist}, and Corollary \ref{cor:mon2_preserves_framing}, we have
\[\Mod_{g,n}[\phi] =  \langle T_c \mid c \in \mathcal C \rangle
\le \mathcal G^{\pr} \le \Mod_{g,n}[\phi]\] 
completing the proof of Theorem \ref{thm:relative_monodromy}.
\end{proof}

\section{Further corollaries}\label{section:corollaries}
In this final section we collect some further corollaries of the work we have done.

\subsection{Classification of components}\label{section:markingcor} The monodromy computations in Theorems \ref{mainthm:absolute_monodromy}, \ref{thm:relative_monodromy}, and \ref{thm:blowup_monodromy} lead to the following classification of the non-hyperelliptic components of strata of marked differentials (c.f. Theorem A of \cite{CS_strata2}):

\begin{corollary}\label{cor:compOTlab}
There is a bijection between the non-hyperelliptic components of $\OT_g^{\lab}(\sing)$ and the isotopy classes of absolute framings of $\Sigma_{g}^n$ with signature $-1 - \sing$.

If $\gcd(\sing)$ is odd then the permutation action of $\Mod_{g}^n$ is transitive, while if $\gcd(\sing)$ is even there are two orbits, classified by Arf invariant.
\end{corollary}

\begin{corollary}\label{cor:compOTpr}
There is a bijection between the non-hyperelliptic components of $\OT_g^{\pr}(\sing)$ and the (relative) isotopy classes of relative framings of $\Sigma_{g,n}$ with signature $-1 - \sing$.

The action of $\Mod_{g,n}^*$ is transitive if $\gcd(\sing)$ is odd and has two orbits if $\gcd(\sing)$ is even, classified by the Arf invariant of the absolute framing. The action of $\Mod_{g,n}$ has two orbits no matter the parity of $\gcd(\sing)$, classified by the generalized Arf invariant.
\end{corollary}

The proofs of both corollaries are simply a consequence of the orbit--stabilizer theorem applied to the $\PMod_g^n$ action on the set of components of $\OT^{\lab}_g(\sing)$, respectively the $\Mod_{g,n}$ action on the components of $\OT^{\pr}_g(\sing)$, together with the classification of  orbits (Theorems \ref{thm:KZ_strata} and \ref{theorem:Boissy_framed}, respectively).

\subsection{Cylinder shears and fundamental groups of strata}\label{section:gencor}
For a component $\mathcal H$ of a general stratum $\OM_g(\sing)$, no explicit set of generators for $\pi_1^{\orb}(\mathcal H)$ is known.
Cylinder shears play a role analogous to Dehn twists in the theory of the mapping class group, and it is natural to wonder about the extent to which shears generate $\pi_1^{\orb}(\mathcal H)$. If the partition $\sing$ contains any repeated elements, the corresponding zeros can be exchanged, but this certainly cannot be accomplished using shears. Thus one must first pass to a ``labeled stratum,'' i.e., a component $\mathcal H_{\lab}$ of the cover $\OM_g^{\lab}(\sing)$. Even here, the work of Boissy (in the guise of Corollary \ref{corollary:PR_monodromy}) implies that $\pi_1^{\orb}(\mathcal H_{\lab})$ is {\em never} generated by shears alone, since shears map trivially onto the prong rotation group $PR$. However, prong rotation is not detected by the monodromy representation $\rho: \pi_1^{\orb}(\mathcal H_{\lab}) \to \PMod_g^n[\bar \phi]$, only by the refinement whose target is the pronged mapping class group. 

As a corollary of our monodromy computations and Corollary \ref{corollary:whentrivial}, we find that the arithmetic of the partition $\sing$ provides an obstruction for the subgroup of $\pi_1^{\orb}(\mathcal H_{\lab})$ generated by cylinder shears to generate the monodromy group in $\PMod_g^n$. Thus, the prong rotation group ``leaves a trace'' in the group $\PMod_g^n[\bar \phi]$, even though there is no way of measuring prong rotation in $\PMod_g^n[\bar \phi]$ directly. 

\begin{corollary}\label{cor:pi1notgenbycyl}
Let $\sing$ be a partition of $2g-2$ for $g \ge 5$ and let $\mathcal H_{\lab}$ be a non-hyperelliptic component of the stratum $\OM_g^{\lab}(\sing)$. Write $\sing = (\eta_1, \ldots, \eta_p, \upsilon_1, \ldots, \upsilon_q)$ where $\eta_i$ are even, $\upsilon_j$ are odd, and $p+q=n$. If $q > 2$, or if the elements of
\[ \left\{\eta_1 + 1, \ldots, \eta_p + 1, \frac{\upsilon_1 + 1}{2}, \ldots, \frac{\upsilon_q + 1}{2} \right\}\]
are not pairwise coprime, then the subgroup of $\pi_1^{\orb}(\mathcal H_{\lab})$ generated by cylinder shears does not surject onto $\mathcal G^{lab}$.
\end{corollary}

\begin{proof}
Choose an arbitrary component $\mathcal H _{\text{pr}}$ of the preimage of $\mathcal H_{\lab}$ in $\OM_g^{\pr}(\sing)$ and let $\phi$ denote the induced relative framing. Let $\mathcal{C}$ denote the subgroup of $\pi_1^{\orb}(\mathcal H_{\lab})$ generated by cylinder shears (or rather, by elements which are conjugate to cylinder shears by some path connecting a basepoint $(X, \omega)$ to the surface $(Y, \omega)$ realizing the relevant cylinder). As cylinder shears preserve prong--markings, we see that $\mathcal{C} \le \pi_1^{\orb}(\mathcal H_{\text{pr}}).$

Recall that $\rho_{\pr}: \pi_1^{\orb}(\mathcal H_{\text{pr}}) \rightarrow \Mod_{g,n}$ denotes the monodromy representation of $\mathcal H_{\pr}$ with image $\mathcal G^{\pr}$, and that $\mathcal G^{\lab}$ denotes the image of the monodromy representation of $\mathcal H_{\lab}$ into $\PMod_g^n$. Now Theorem \ref{mainthm:absolute_monodromy} finds that $\mathcal G^{\lab} = \PMod_g^n[\bar \phi]$, and Theorem \ref{thm:relative_monodromy} finds that $\rho_{\pr}(\mathcal C) \le \mathcal G^{\pr} = \Mod_{g,n}[\phi]$. 
\footnote{In fact, the proof of Theorem \ref{thm:relative_monodromy} shows that the image of $\mathcal C$ under the monodromy map is all of $\Mod_{g,n}[\phi]$.}
By Corollary \ref{corollary:whentrivial}, under our current hypotheses, the map $\Mod_{g,n}[\phi] \to \PMod_g^n[\bar \phi]$ is never surjective. Therefore, the image of $\mathcal C$ in $\mathcal G^{\lab}$ is a strict subgroup.
\end{proof}

\subsection{Change--of--coordinates for saddle connections}\label{subsec:saddleCCP}

In this section, we use the machinery of prong--markings and the framed change--of--coordinates principle to prove Corollary \ref{cor:when_saddle} (realization of arcs as saddles). As in Corollary \ref{cor:when_cyl}, the idea is to use the framed change--of--coordinates principle to take a given arc to some target arc which is realized as a saddle connection. However, unlike cylinders, saddle connections do not share a common winding number. The main difficulty in the proof of Corollary \ref{cor:when_saddle} is therefore to construct a sufficiently large set of saddle connections to play the role of target arcs (Lemma \ref{lem:sc_wn}).

We begin by clarifying some conventions with regards to arcs on surfaces with boundary versus surfaces with marked points.
Recall that if $(\Sigma_{g,n}, \phi)$ is a relatively framed surface, then we have fixed once and for all a legal basepoint on each boundary component and we only consider arcs ending on this prescribed basepoint. When $(\Sigma_{g,n}, \phi)$ arises from the blow-up of a prong--marked abelian differential, this is equivalent to stipulating that arcs must leave and enter the zeros with prescribed tangent directions.

Upon capping each boundary component of $\Sigma_{g,n}$ with a punctured disk (equivalently, forgetting the prong structure coming from the differential), each (relative) isotopy class of arc $a$ on $\Sigma_{g,n}$ projects to an (absolute) isotopy class of arc on $\Sigma_g^n$, which we will denote by $\pi(a)$. Observe that the map $\pi$ is not injective; elements of its fibers are related by Dehn twists about the endpoint(s) of the arc.

\para{Saddle connections on one--cylinder differentials}
In order to exhibit the desired collection of saddles, it will be convenient to use different model surfaces than the ones introduced in Section \ref{subsec:genmon}. 

To that end, recall that an abelian differential is called {\em Jenkins--Strebel} if it is completely horizontally periodic, i.e., if it can be written as the union of closed horizontal cylinders. We will in particular be interested in those which can be obtained by identifying boundary edges of a single cylinder, called {\em one--cylinder} Jenkins--Strebel differentials.

The existence result we will use is the following; see \cite[Section 2]{Zorich_JS} for an explicit construction.

\begin{proposition}\label{prop:JS}
There exists a one--cylinder Jenkins--Strebel differential $(Y, \eta)$ in every (nonempty) component of every stratum $\OM_g(\sing)$.
\end{proposition}

As in the proof of Lemma \ref{lemma:framed_prototypes}, we can also upgrade these $\eta$ to yield one--cylinder Jenkins--Strebel differentials in each component of $\OM_g^{\lab}(\sing)$ and $\OM_g^{\pr}(\sing)$ by labeling zeros and prongs, respectively. 

Using these new model differentials, we may now exhibit saddle connections which have (a preimage under $\pi$ which has) arbitrary winding number.

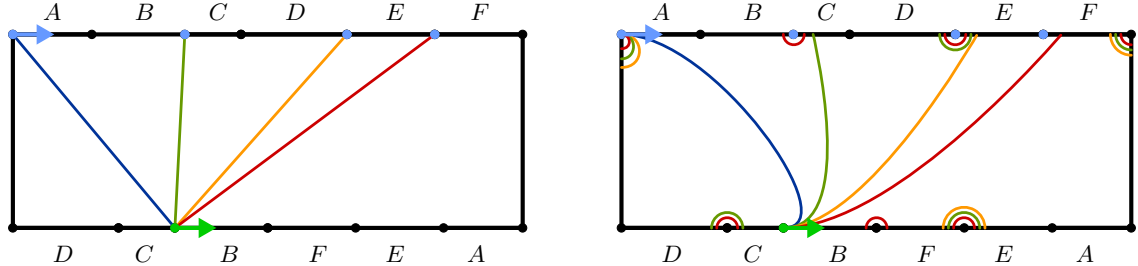
\begin{figure}[ht]
\labellist
\small
\pinlabel $D$ [t] at 15 0
\pinlabel $C$ [t] at 37 0
\pinlabel $B$ [t] at 60 0
\pinlabel $F$ [t] at 84 0
\pinlabel $E$ [t] at 105 0
\pinlabel $A$ [t] at 127 0
\pinlabel $D$ [t] at 180 0
\pinlabel $C$ [t] at 202 0
\pinlabel $B$ [t] at 225 0
\pinlabel $F$ [t] at 249 0
\pinlabel $E$ [t] at 270 0
\pinlabel $A$ [t] at 292 0
\pinlabel $A$ [b] at 12 57
\pinlabel $B$ [b] at 37 57
\pinlabel $C$ [b] at 57 57
\pinlabel $D$ [b] at 78 57
\pinlabel $E$ [b] at 105 57
\pinlabel $F$ [b] at 128 57
\pinlabel $A$ [b] at 177 57
\pinlabel $B$ [b] at 202 57
\pinlabel $C$ [b] at 222 57
\pinlabel $D$ [b] at 243 57
\pinlabel $E$ [b] at 270 57
\pinlabel $F$ [b] at 293 57
\endlabellist
\includegraphics[scale=1.4]{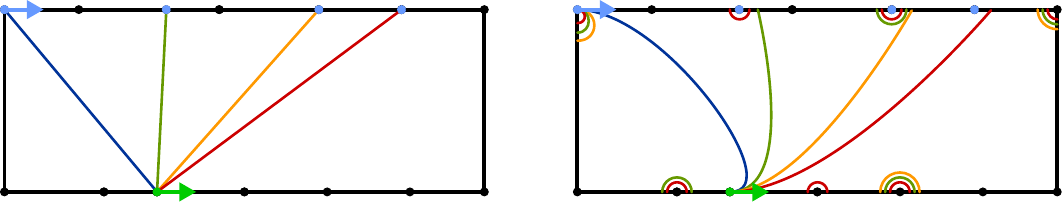}
\caption{A one--cylinder Jenkins--Strebel differential in $\OM_3(3,1)$ and saddle connections on it. On the left the arcs have been realized geodesically; on the right they have been realized with prescribed tangential data.}
\label{figure:JS}
\end{figure}

\begin{lemma}\label{lem:sc_wn}
Let $\comp_{\pr}$ be a component of $\OM_g^{\pr}(\sing)$ and let $p, q$ be two distinct zeros. Then for every $s \in \Z + \frac{1}{2}$, there is a differential $(Y, \eta) \in \comp_{\pr}$ and nonseparating arc $a$ on $Y^*$ from $\Delta_p$ to $\Delta_q$ (based at the legal basepoints prescribed by the prong--marking) such that 
\begin{enumerate}
\item $\phi(a) = s$ for the relative framing $\phi$ induced by $1/\eta$ and
\item the geodesic representative of $\pi(a)$ on $(Y, \text{Zeros}(\eta))$ is a saddle connection.
\end{enumerate}
\end{lemma}
\begin{proof}
As twisting around $\Delta_q$ changes $\phi(a)$ by $\pm \phi(\Delta_q)$, it suffices to prove that there exist such arcs for each residue class mod $\phi(\Delta_q)$.

Let $(Y, \eta)$ be a one--cylinder Jenkins--Strebel differential in $\comp_{\pr}$ (which exists by Proposition \ref{prop:JS}). Now by definition, we can write $Y = C / \sim$, where $C = S^1 \times [0,1]$ is a closed cylinder and $\sim$ identifies segments of the ``top boundary'' $S^1 \times \{1\}$ with segments of its ``bottom boundary'' $S^1 \times \{0\}$. Let $Q: C \rightarrow Y$ denote the quotient map.

We observe that our choice of prong at $p$ determines a unique half--separatrix on $(Y, \eta)$ and hence a pair of half--segments in $\partial C$. In particular, the prong--marking determines a unique point $\tilde p \in Q^{-1}(p) \cap (S^1 \times \{0\})$.

Now consider the set $\mathcal A$ of all arcs in $C$ which start at $\tilde p$ and end at a point of $Q^{-1}(q)\cap S^1 \times \{1\}$; up to Dehn twisting along the core curve of $C$, there are exactly $\phi(\Delta_q)$ such arcs (see Figure \ref{figure:JS}). Moreover, since the arcs of $\mathcal A$ are each realized as straight line segments on $C$, the arcs of $\pi (Q (\mathcal A))$ are all realized as saddle connections on $(Y, \eta)$. Isotoping these arcs to leave $p$ and enter $q$ with the prescribed tangential data, we may measure the winding numbers of $a \in Q(\mathcal A)$. Careful inspection of Figure \ref{figure:JS} shows that
\[\{ \phi(a) \text{ mod } \phi(\Delta_q) : a \in Q(\mathcal A) \} = \left\{ \frac{1}{2}, \frac{3}{2}, \ldots, \frac{2\phi(\Delta_q)-1}{2} \right\},\]
finishing the proof of the lemma.
\end{proof}

Now that we have a sufficiently large collection of target saddle connections, we may apply the framed change--of--coordinates principle to deduce Corollary \ref{cor:when_saddle}.

\begin{proof}[Proof of Corollary \ref{cor:when_saddle}]
Let $\bar{a}$ be an arc on $(X, \omega)$ with endpoints $p$ and $q$. Choose an arbitrary prong--marking of $(X, \omega)$ and an arc $a$ on $X^*$ such that $\pi( a) = \bar a$. Let $(Y, \eta)$ be the 1--cylinder Jenkins--Strebel differential from Lemma \ref{lem:sc_wn} in the same component of $\OM_g^{\pr}(\sing)$ as $(X, \omega)$ and let $\phi$ denote the induced relative framing on $Y^*$.

Choose a path $\alpha$ connecting $(X, \omega)$ to $(Y, \eta)$ and let $\alpha_*(a)$ denote the parallel transport of $a$ along $\alpha$ (equivalently, lift $\alpha$ to a path in $\OT_g^{\pr}(\sing)$ and use the marking maps). Then by Lemma \ref{lem:sc_wn} there is an arc $b$ on $Y^*$ with $\phi(b) = \phi( \alpha_*(a))$ and so that $\pi(b)$ is a saddle connection on $(Y, \eta)$. 

Now by the framed change--of--coordinates principle (Proposition \ref{proposition:CCP}), there is an element of $\Mod_{g,n}[\phi]$ taking $\alpha_*(a)$ to $b$. By Theorem \ref{thm:relative_monodromy}, this element can be represented by loop $\beta$ in $\comp_{\pr}$. The concatenated path $\alpha \cdot \beta$ therefore takes $a$ on $X^*$ to $b$ on $Y^*$, and so its projection to $\OM_g(\sing)$ is the desired path.
\end{proof}

\bibliographystyle{alpha}
\bibliography{library}

\end{document}